\documentclass{article}
\usepackage{authblk}
\usepackage{graphicx} 
\usepackage[margin = 1in]{geometry}
\usepackage{amsfonts}
\usepackage{mathtools}
\usepackage{amsmath}
\usepackage{amssymb}
\usepackage{amsthm}
\usepackage{authblk}
\usepackage{natbib}
\usepackage{babel}
\usepackage{graphicx}
\usepackage{graphics}
\usepackage[dvipsnames]{xcolor}
\usepackage{hyperref}
\usepackage{pgfplots}
\usepackage{amsthm}
\pgfplotsset{width = 5.4cm, compat = 1.18}

\renewcommand{\P}{\mathbb{P}}
\renewcommand{\Pr}{\mathrm{P}}

\newcommand{\R}{\mathbb{R}}
\newcommand{\E}{\mathbb{E}}

\newcommand{\var}{\text{Var}}

\newcommand{\bm}{\boldsymbol}
\newcommand{\xvec}{\mbox{\bf x}}

\newcommand{\Xvec}{\mbox{\bf X}}
\newcommand{\Yvec}{\mbox{\bf Y}}

\newcommand{\yvec}{\mbox{\bf y}}
\newcommand{\thetavec}{\mbox{\boldmath $\theta$}}

\newcommand{\sigmat}{\mbox{\boldmath $\Sigma$}}

\newcommand{\betavec}{\mbox{\boldmath $\beta$}}

\theoremstyle{definition}
\newtheorem{thm}{Theorem}[section]

\newtheorem{rem}{Remark}
\newtheorem{prop}{Proposition}[section]
\newtheorem{lemma}{Lemma}[section]

\title{A nonparametric 
test of spherical symmetry applicable to high dimensional data}
\author{Bilol Banerjee}
\author{Anil K. Ghosh}
\affil{Theoretical Statistics and Mathematics Unit,\\ Indian Statistical Institute, Kolkata}
\date{\null}

\newcommand{\tikzcircle}[2][red,fill=red]{\tikz[baseline=-0.5ex]\draw[#1,radius=#2] (0,0) circle ;}
\newcommand{\tikzcirclev}[2][violet,fill=violet]{\tikz[baseline=-0.5ex]\draw[#1,radius=#2] (0,0) circle ;}
\newcommand{\tikzcirclep}[2][purple,fill=purple]{\tikz[baseline=-0.5ex]\draw[#1,radius=#2] (0,0) circle ;}

\definecolor{applegreen}{rgb}{0.55,0.71,0.0}

\begin{document}

\maketitle

\begin{abstract}
   We develop a test for spherical symmetry of a multivariate distribution $\Pr$ that works well even when the dimension of the data $d$ is larger than the sample size $n$. We propose a non-negative measure of spherical asymmetry $\zeta(\Pr)$ such that $\zeta(\Pr)=0$ if and only if $\Pr$ is spherically symmetric. We construct a consistent estimator of $\zeta(\Pr)$ using the data augmentation method and investigate its large sample properties. The proposed test based on this estimator is calibrated using a novel resampling algorithm. Our test controls the type I error, and it is consistent against general alternatives. We also study its behavior for a sequence of alternatives $(1-\delta_n) F+\delta_n G$, where $\zeta(G)=0$ but $\zeta(F)>0$, and $\delta_n \in [0,1]$. When $\lim\sup\delta_n<1$, for any $G$, the power of our test converges to unity as $n$ increases. However, if $\lim\sup\delta_n=1$, the asymptotic power of our test depends on $\lim n(1-\delta_n)^2$. We establish this by proving the minimax rate optimality of our test over a suitable class of alternatives and showing that it is Pitman efficient when $\lim n(1-\delta_n)^2>0$. Moreover, our test is provably consistent for high-dimensional data even when $d$ grows with $n$. When the center of symmetry is not specified by the null hypothesis, most of the existing tests often fail to satisfy the level property. To take care of this problem, we propose a general recipe for constructing modified tests based on pairwise differences of the observations. Our numerical results amply demonstrate the superiority of the proposed test over some state-of-the-art methods.

\vspace{0.05in}

\end{abstract}

\section{Introduction}

An inherent property of nature is that it tends to exhibit some form of symmetry within itself. In the nineteenth century, such symmetric patterns were approximated by the normal distribution \citep[see, e.g.,][for a history of statistical methods]{lehmann2012history}. But with time, more general notions of symmetry were introduced. One of the most popular notions is spherical symmetry or elliptic symmetry (i.e., spherical symmetry after standardization) \citep[see, e.g.][]{ chmielewski1981elliptically,fang1990book,fourdrinier2018shrinkage}. A $d$-dimensional ($d>1$) random vector ${\bf X}$ is said to follow a spherically symmetric distribution  (about the origin) if ${\bf X}$ has the same distribution as $ {\bf H}{\bf X}$ (i.e., ${\bf X} \stackrel{D}{=} {\bf H}{\bf X}$) for any $d \times d$ orthogonal matrix ${\bf H}$.  This is an important class of distributions, and several statistical methods have been developed motivated by the sphericity or the ellipticity of the underlying distribution
\citep[see, e.g.,][]{randles1989distribution, chaudhuri1993sign, jornsten2004clustering, ghosh2005maximum}. Therefore, testing the sphericity of a distribution is an important statistical problem. Several methods have been proposed in the literature for 
testing the sphericity of a distribution $\Pr$
based on a sample $\mathcal{D} = \{{{\bf X}_1,{\bf X}_2,\ldots, {\bf X}_n}\}$ of $n$ independent realizations of the random vector ${\bf X} \sim \Pr$. 
For example, \cite{fang1993} proposed an asymptotically distribution-free test for spherical symmetry using the projection pursuit technique. They computed the Wilcoxon-Mann-Whitney statistic for several pairs of projection directions and used the minimum over all such projection-pairs as their test statistic. However, this is only a necessary test for sphericity and does not have consistency under general alternatives. \cite{kolchinskii1998} proposed a test based on the difference between the empirical spatial rank function and the theoretical spatial rank function under spherical symmetry, where the unknown components of these theoretical ranks were estimated from the data. 
\cite{smith1977} proposed a test statistic for bivariate data using the fact that if ${\bf X}$ is spherically symmetric then $\|{\bf X}\|$ and ${\bf X}/\|{\bf X}\| $ are independent and ${\bf X}/\|{\bf X}\|$ follows a uniform distribution over the perimeter of the unit circle in $\R^{2}$. Later, \cite{baringhaus1991} modified the test statistic and generalized the test to any arbitrary dimension. However, this test involves a complex function of the dimension $d$, which makes it difficult to study its high dimensional behaviour. \cite{diks1999} proposed a Monte Carlo test for multivariate spherical symmetry conditionally on minimal sufficient statistics, but the pairwise distances used in the test statistic need an appropriate scaling for high dimensional data, which is missing from the literature. 
\cite{liang2008} proposed some necessary tests for spherical symmetry based on the fact that under spherical symmetry, ${\bf X}/\|{\bf X}\|$ is uniformly distributed on $\mathcal{S}^{d-1}$, the surface of the unit sphere in $\R^d$. But, they did not consider the independence between $\|{\bf X}\|$ and  ${\bf X}/\|{\bf X}\|$, and therefore the test is not consistent against general alternatives. \cite{henze2014} proposed a test based on characteristic functions and calibrated the test using a bootstrap algorithm. However, this test requires the generation of the uniform grid over the unit sphere, which becomes computationally prohibitive even in moderately large dimensions. \cite{albisetti2020} proposed another test utilizing the fact that ${\bf X}$ is spherically symmetric if and only if $\E\{{\bf u}^{\top} {\bf X}\mid {\bf v}^{\top}{\bf X}\} = 0$ for all ${\bf u}$ and ${\bf v}$ with ${\bf u}^{\top}{\bf v}=0$. They constructed a Kolmogorov-Smirnov-type test statistic over suitable choices of test functions. 
Recently, \cite{huang2023multivariate} proposed some tests for different notions of symmetry using optimal transport for multivariate data. However, these tests are not consistent against general alternatives. 

These above-mentioned tests can be used when the dimension of the data $d$ is small compared to the sample size $n$ (i.e., $d<<n$), and some of them are large-sample consistent for any fixed dimension $d$. However, the applicability of these tests for high-dimensional data (i.e., when $d$ is comparable to or larger than $n$) is not clear. \cite{zou2014} and \cite{feng2017high} proposed tests of sphericity for high-dimensional data where the test statistics were constructed using the multivariate sign function assuming the ellipticity of the underlying distribution.  \cite{ding2020some} proposed a test based on the ratio of traces of different powers of the sample variance-covariance matrix and established its high dimensional consistency. However, these tests may fail when the underlying distribution is not spherically symmetric but ${\bf X}/\|{\bf X}\|\sim$ Unif(${\cal S}^{d-1}$), the uniform distribution on $\mathcal{S}^{d-1}$, (e.g. angular symmetric distribution) or the variance-covariance matrix of the underlying distribution is a constant multiple of the identity matrix. 

To overcome these limitations, in this article, we propose a test of spherical symmetry that is computationally feasible, applicable for a general class of alternatives and consistent for high-dimensional data even when $d$ is much larger than $n$. The efficacy of our method is demonstrated through various theoretical and numerical results in the following sections. A brief summary of our contributions is given below.

\begin{itemize}
    \item In Section 2, we propose a new measure of spherical asymmetry $\zeta(\Pr)$ of a multivariate distribution $\Pr$. It is based on  the fact that ${\bf X}\sim \Pr$ is spherically symmetric if and only if ${\bf X}$ and $\|{\bf X}\| {\bf U}$ are identically distributed, where $\|{\bf X}\|$ and ${\bf U}$ are independent and ${\bf U}\sim$ Unif$(\mathcal{S}^{d-1})$. 
    The proposed measure $\zeta(\Pr)$ is non-negative, and it takes the value zero if and only if $\Pr$ is spherically symmetric (see Proposition \ref{validity}). However, it involves some terms that are not estimable from the observed data $\mathcal{D}$ only. To overcome this, we adopt a data augmentation approach which enables us to estimate $\zeta(\Pr)$ unbiasedly. We study the large sample properties of this estimator { (see Theorem \ref{consistency},~\ref{large-sample-distribution-1}~and \ref{large-sample-distribution-2})} and propose a test that rejects $H_0${ , the null hypothesis of spherical symmetry,} for large values of the estimate. Utilizing the exchangeability among observed and augmented variables, we develop a new resampling algorithm to calibrate the test. We show the resulting test has the desired control over the type I error rate and is consistent against general alternatives. 

    \item In Section 3, we study the asymptotic behaviour of our test under mixture alternatives of the form $(1-\delta_n)F+\delta_n G$, where $G$ is spherically symmetric, but $F$ is non-spherical. For any fixed $d$, $F$ and a sequence $\{\delta_n\}$ that remains bounded away from one, we prove that our test is consistent in the sense of the minimum power over the class of spherical distributions ${\mathcal G}=\{G: \zeta(G)=0\}$, i.e., for any $G \in {\mathcal G}$, the power of our test converges to $1$ as $n$ increases (see Theorem \ref{robustness}). We also establish that our test is minimax rate optimal over the class of alternatives $\mathcal{F} = \{\Pr~:~\zeta(\Pr)>\epsilon\}$ with minimax rate of separation $\epsilon(n)\asymp n^{-1}$ (see {  Theorem} \ref{minimax-lower-bound} and \ref{minimax-upper-bound}). This in turn proves the consistency of the test over a fairly general class of shrinking alternatives when $d$ and $n$ both diverge to infinity. We also establish that the proposed test is Pitman efficient against contamination alternatives for which $n(1-\delta_n)^2$ {  remains} bounded away from $0$ (see Theorem \ref{local-limit} and \ref{local-limit-per-2}). 

    \item In Section 4, extensive simulation studies are carried our to compare the empirical performance of our test with some state-of-the-art methods. We observe that while most of the existing tests fail to detect spherical asymmetry even for moderately high dimensional data, our test works well in high dimensions, even when the dimension is larger than the sample size. 

    \item {  In Section 5, we consider the case where the null hypothesis does not specify the center of symmetry $\thetavec$. In such cases, we can find ${\hat \thetavec}$, an estimate of the center from the data, and carry out a test based on $\Xvec_1-{\hat \thetavec},\Xvec_2-{\hat \thetavec},\ldots,\Xvec_n-{\hat \thetavec}$. If ${\hat \thetavec}$
    is a consistent estimator of $\thetavec$, our test based on these centered observations has large sample consistency (follows from Theorem \ref{thm:consistent-estimation-centering}). However, this may not have a good level property even when the sample size is large. To address this problem, we propose a modified test based on the differences between the observations. An interesting result (see Lemma \ref{thm:key-transformation}) is presented in this context which ensures that the resulting test has desirable level and power properties.  }
    \end{itemize}

In Section 6, we analyze a benchmark dataset where our test significantly outperforms all other tests considered in this article. Our methodology can also be extended to test for other types of symmetry. Some other possible extensions are discussed in Section 7. We also evaluate the exact expression of $\zeta(\Pr)$ for Gaussian distributions (see Section A in the supplementary material) that helps us to understand the behaviour of our test for high dimensional data (see Example 4). All proofs and some auxiliary theoretical results are given in the supplementary
material.

\section{The proposed methodology}

We know that a random vector ${\bf X}$ follows a spherically symmetric distribution if and only if ${\bf X}$ and ${\bf H}{\bf X}$ have the same distribution for any orthogonal matrix ${\bf H}$. However, we can also characterize spherical symmetry using the following lemma.

\vspace{0.025in}
\begin{lemma}
    A $d$-dimensional random vector ${\bf X}$ is spherically symmetric if and only if ${\bf  X} \stackrel{D}{=} \| {\bf X}\| {\bf U}$, where $ {\bf U}$ is independent of $\|{\bf X}\|$, and it is uniformly distributed over $\mathcal{S}^{d-1}$.
\end{lemma}

\vspace{0.05in}
For proof of Lemma 2.1, see page 31 in \cite{fang1990book}. Using this characterization, we construct a new measure of spherical asymmetry for any probability distribution $\Pr$. Let $\varphi_{1}$  be the characteristic function of ${\bf X}\sim{\Pr}$ and $\varphi_{2}$ be that of its spherically symmetric variant  $\|{\bf X}\| {\bf U}$, where ${\bf U} \sim$ Unif(${\mathcal S}^{d-1})$ and $\|{\bf X}\|$ are independent. We define
\vspace{-0.05in}
\begin{equation}
        \zeta_W(\Pr) = \int \big|\varphi_1({\bf t}) - \varphi_{2} ({\bf t})\big|^2 \mathrm{d}W({\bf t}),
    \label{definition}
\end{equation}
where $W$ is a non-negative measure equivalent to the Lebesgue measure { (i.e., $W$ is absolutely continuous with respect to the Lebesgue measure and vice versa)}. The integral is taken in the principal value sense. {  Note that if $W$ is a probability measure, $\zeta_W(\Pr)$ is always finite. Otherwise, one needs to make suitable assumptions on $\Pr$ for the finiteness of $\zeta_W(\Pr)$, such as in the case of energy distance (see \cite{zekely2023book}).} The following proposition proves the characterization property of $\zeta_W(\Pr)$ by showing $\zeta_W(\Pr)\ge 0$, where the equality holds if and only if $\Pr$ is spherically symmetric. 
\begin{prop}
$\zeta_W(\Pr)$ is non-negative and $\zeta_W(\Pr)=0$ if and only if $\Pr$ is a spherically symmetric distribution.
\label{validity}
\end{prop}
For suitable choices of $W$, $\zeta_W(\Pr)$ has nice closed-form expressions. If $W$
is taken as the probability measure corresponding to ${\mathcal N}_d({\bf 0}_d, \frac{1}{d}{\bf I}_d)$, the $d$-dimensional Gaussian distribution with the mean ${\bf 0}_d=(0,0,\ldots,0)^{\top}$ and the dispersion matrix $\frac{1}{d}{\bf I}_d$, where ${\bf I}_d$ is the $d \times d$ identity matrix, the closed form expression is denoted by $\zeta(\Pr)$. We derive this expression using the following lemma.

\begin{lemma}
    For any two non-zero vectors ${\bf X}_1,{\bf X}_2\in \R^d$,
    $$\int \exp\big\{i\langle {\bf t}, {\bf X}_1-{\bf X}_2\rangle\big\} \frac{d^{d/2}}{(2\pi)^{d/2}} e^{-{d\|{\bf t}\|^2}/{2}} {\mathrm d}{\bf t} = \exp\big\{-\frac{1}{2d}\|{\bf X}_1-{\bf X}_2\|^2\big\}.$$
    \label{basis-alt-form}
\end{lemma}

\vspace{-0.25in}
Using Lemma \ref{basis-alt-form} we get the following expression of $\zeta(\Pr)$.
\noindent
\begin{thm}
 If ${\bf X}_1,{\bf X}_2$ are independent copies of ${\bf X}\sim \Pr$ and $W$
is the probability measure corresponding to $\mathcal{N}_d({\bf 0}_d, \frac{1}{d}{\bf I}_d)$, then 
    \begin{align}
    \label{eq:closed-form-exp}
        \zeta(\Pr) = \E\big\{\exp\{-\frac{1}{2d}\|{ {\bf X}_1-{\bf X}_2}\|^2\}\big\} + &\E\big\{\exp\{-\frac{1}{2d}\|{ {\bf X}_1^\prime-{\bf X}_2^\prime}\|^2\}\big\}\\ \nonumber & -2\E\big\{\exp\{-\frac{1}{2d}\|{ {\bf X}_1-{\bf X}_2^\prime}\|^2\}\big\},
\vspace{-0.1in}
    \end{align}
    where $ {\bf X}_i^\prime= \|{\bf X}_i\|{\bf U}_i$ for $i=1,2$, and ${\bf U}_1,{\bf U}_2$ are independent and identically distributed (i.i.d.) as Unif $(\mathcal{S}^{d-1})$, which are independent of ${\bf X}_1 $ and ${\bf X}_2$.
    \label{closed-form-expression}
\end{thm}

There are several other choices of $W$ for which we have closed-form expressions for $\zeta_W(\Pr)$. But, unless mentioned otherwise, throughout this article, {  we shall assume $W$ is the probability measure corresponding to $\mathcal{N}_d({\bf 0}_d, \frac{1}{d}{\bf I}_d)$ as considered in Theorem \ref{closed-form-expression} and explore some interesting features of $\zeta(\Pr)$.}

\begin{rem}
    Theorem \ref{closed-form-expression} shows that $\zeta(\Pr)$ can be expressed as a function of the pairwise distance between ${\bf X}_1,{\bf X}_2$ and their spherically symmetric variants ${\bf X}_1^\prime,{\bf X}_2^\prime$. Since, the pairwise distances are invariant under orthogonal transformation, $\zeta(\Pr)$ is also invariant under rotation. In particular, we have $\zeta({\mathcal N}_d({\bf 0},\sigmat)) = \zeta({\mathcal N}_d({\bf 0},{\bf H} \sigmat {\bf H}^T))$ for any orthogonal matrix ${\bf H}$. The exact expression of $\zeta({\mathcal N}_d({\bf0},\sigmat))$ for any positive-semidefinite variance-covariance matrix $\sigmat$ is interesting. Interested readers are referred to Section A in the supplementary document for a detailed derivation. 
\end{rem}

{  

Note that $\zeta(\Pr)$ in \eqref{eq:closed-form-exp} can also be viewed as the maximum mean discrepancy (MMD) \citep[see, e.g.,][]{gretton2007kernel} between the distributions of ${\bf X}$ and ${\bf X}^\prime$ with respect to the kernel $K({\bf x},{\bf y}) = \exp\{-\|{\bf x}-{\bf y}\|^2/(2d)\}$. We know that MMD measures the difference between two distributions by embedding them into a Reproducing Kernel Hilbert Space (RKHS). As long as $W$ is the measure corresponding to a symmetric probability distribution in $\R^d$, we can write $\zeta_W(\Pr)$ as the MMD between the distributions of ${\bf X}$ and ${\bf X}^\prime$ with respect to the kernel $K({\bf x},{\bf y}) = \int \exp\big\{i\langle {\bf t}, {\bf x} - {\bf y} \rangle\big\} \mathrm d W({\bf t}) = \varphi_W({\bf x}-{\bf y})$. This can be viewed as a generalized version of $\zeta(\Pr)$ based on MMD. We can also introduce a scale parameter (bandwidth) within the distribution $W$, which will give us the kernel $K_\sigma({\bf x},{\bf y}) = \varphi_W(\frac{{\bf x}-{\bf y}}{\sigma})$. However, finding the optimal kernel and its associated bandwidth is a difficult problem, and we do not address it here. Our empirical experience suggests that the use of the Gaussian kernel gives highly satisfactory performance. Therefore, throughout this article, we assume that $K({\bf x},{\bf y})$ is the Gaussian kernel and concentrate only on the resulting measure $\zeta(\Pr)$. However, our theoretical findings are valid even if the Gaussian kernel is replaced by other bounded kernels, such as the Laplacian or the Inverse Quadratic kernel.
}



\subsection{Estimation of $\zeta({\Pr})$}
\label{estimation}

Let $\mathcal{D}=\{{\bf X}_1,{\bf X}_2,\ldots, {\bf X}_n\}$ be a random sample of size $n$ from a $d$-dimensional distribution ${\Pr}$. Note that the term 
$\E[K({\bf X}_1,{\bf X}_2)]$ in $\zeta({\Pr})$ can be easily estimated by its empirical analog, but $\zeta(\Pr)$ involves two other terms
$\E[K({\bf X}_1^\prime,{\bf X}_2^\prime)]$ and $\E[K({\bf X}_1,{\bf X}_2^\prime)]$ which are not estimable from $\mathcal{D}$ alone. Therefore, we adopt the following data augmentation approach.
\begin{itemize}
    \item Generate ${\bf U}_1,{\bf U}_2,\ldots, {\bf U}_n$, a random sample of size $n$ from Unif$(\mathcal{S}^{d-1})$ and define $ {\bf X}_i^\prime = \|{\bf X}_i\|{\bf U}_i$ for $i=1,2,\ldots, n$. Also, define $\mathcal{D}^\prime = \{{ ({\bf X}_1,{\bf X}_1^\prime), ({\bf X}_2,{\bf X}_2^\prime), \ldots, ({\bf X}_n,{\bf X}_n^\prime)}\}$ to be the augmented dataset.

    \item Using observations from $\mathcal{D}^\prime$, we propose an estimator of $\zeta(\Pr)$ as
    \begin{align*}
        \hat\zeta_n = {\binom{n}{2}}^{-1}\sum_{i=1}^n\sum_{j=i+1}^n &\Big\{ K({\bf X}_i,{\bf X}_j) + K({\bf X}_i^\prime,{\bf X}_j^\prime)-2 K({\bf X}_i,{\bf X}_j^\prime) \Big\}.
    \end{align*}
\end{itemize}


\noindent Clearly $\hat\zeta_n$ can be viewed as a $U$-statistic with the symmetrized kernel function 
\begin{align}
    g\big({ ({\bf x}_1,{\bf x}_1^\prime),({\bf x}_2,{\bf x}_2^\prime)}\big) = K({\bf x}_1,{\bf x}_2)+K({\bf x}_1^\prime,{\bf x}_2^\prime) -K({\bf x}_1,{\bf x}_2^\prime)-K({\bf x}_2,{\bf x}_1^\prime).
    \label{kernel-function}
\end{align}
It is easy to see that $\hat{\zeta}_n$ is an unbiased estimator of $\zeta(\Pr)$. Since $g$ is a bounded kernel, using the bounded difference inequality, we can establish the following bound for the deviation of $\hat\zeta_n$ from its population analog $\zeta(\Pr)$. 

\vspace{0.05in}
\begin{thm}
    If ${\bf X}_1,{\bf X}_2,\ldots, {\bf X}_n$ are independent copies of a $d$-dimensional random vector ${\bf X} \sim {\Pr}$, then 
    $\P\Big\{\big|\hat\zeta_n - \zeta({\Pr})\big|>\epsilon\Big\} \leq 2\exp\{-\frac{n\epsilon^2}{32}\},$
    and this inequality holds irrespective of the dimension $d$.
    \label{consistency}
\end{thm}

\vspace{0.05in}
Theorem \ref{consistency} shows that ${\hat \zeta}_n$ is a strongly consistent estimator of $\zeta(\Pr)$ (follows from the Borel-Cantelli Lemma) and the exponential bound is free from $d$. The large sample distribution of $\hat\zeta_n$ can be derived using the theory of $U$-statistics \citep[see][]{lee2019u}. The asymptotic null distribution of ${\hat \zeta}_n$ is given by the following theorem. 

\vspace{0.05in}
\begin{thm}
Let ${\bf X}_1,{\bf X}_2,\ldots, {\bf X}_n$ be independent copies of the random vector ${\bf X}$, which follows a spherically symmetric distribution ${\Pr}$. Define ${\bf X}_i^\prime = \|{\bf X}_i\|{\bf U}_i$ for $i=1,2,\ldots,n$, where ${\bf U}_1,{\bf U}_2,\ldots,{\bf U}_n$ are i.i.d. Unif$(\mathcal{S}^{d-1})$. Let $\lambda_k$ ($k=1,2,\ldots$) be the eigenvalue corresponding to the eigenfunction $\psi_k$ of the integral equation
$$\E\{g\big(({\bf x},{\bf x}^\prime),({\bf X},{\bf X}^\prime)\big)\psi_k\big({\bf X},{\bf X}^\prime)\} = \lambda_k \psi_{k}({\bf x},{\bf x}^\prime),$$
where $g$ is as in equation (\ref{kernel-function}). Then as $n$ goes to infinity, $n\hat\zeta_n\stackrel{D}{\longrightarrow}\sum_{k=1}^\infty \lambda_k (Z_i^2-1),$
where $\{Z_k: k\ge 1\}$ is a sequence of i.i.d. standard normal variables.
\label{large-sample-distribution-1}
\end{thm}

\vspace{0.05in}
In Theorem \ref{large-sample-distribution-1}, the limiting null distribution of $\hat\zeta_n$ (after appropriate adjustment for location and scale) turns out to be a weighted sum of independent chi-squares due to the first order degeneracy of $g$ under $H_0$ (the null hypothesis of symmetry). However, under $H_1$ (the alternative hypothesis), $g$ is a non-degenerate kernel. Therefore, the limiting distribution of $\hat\zeta_n$ (after appropriate adjustment for location and scale) turns out to be Gaussian, which is asserted by the following theorem.

\vspace{0.05in}
\begin{thm}
Let ${\bf X}_1,{\bf X}_2,\ldots, {\bf X}_n$ be independent copies of a random vector ${\bf X}$, which follows a non-spherical distribution ${\Pr}$. Define ${\bf X}_i^\prime = \|{\bf X}_i\|{\bf U}_i$ where ${\bf U}_1,{\bf U}_2,\ldots,{\bf U}_n$ is an i.i.d. Unif$(\mathcal{S}^{d-1})$. Then as $n$ goes to infinity,
$\sqrt{n}\big(\hat\zeta_n-\zeta(\Pr)\big)\stackrel{D}{\longrightarrow}{\mathcal N}_1(0,4\sigma^2),$
where $g$ is as defined in equation (\ref{kernel-function}) and $\sigma^2 = \var\Big\{\E\big\{g\big(({\bf X}_1,{\bf X}_1^\prime),({\bf X}_2,{\bf X}_2^\prime)\big)\mid ({\bf X}_1,{\bf X}_1^\prime)\big\}\Big\}$. 
\label{large-sample-distribution-2}
\end{thm}

{ 
\begin{rem}
    Recently, \cite{tang2024elliptic} proposed a test of elliptic symmetry using the kernel mean embedding of the distribution of the standardized random vectors. A similar framework can also be adapted for testing spherical symmetry for multivariate distributions. However, their theoretical results rely heavily on the Frechet differentiability of their proposed statistical functionals (see Theorem 2 in \cite{tang2024elliptic}), which is difficult to verify in practice. In contrast, we do not need any such assumptions on the underlying distribution $\Pr$ or the functional $\zeta(\Pr)$. Also, unlike their test, our method is applicable to high dimensional data even when the dimension is much larger than the sample size (see Theorem \ref{high-dim-consistency}).  
\end{rem}
}

\subsection{Test of Spherical Symmetry}
\label{test}
Since, $\hat\zeta_n$ is a strongly consistent estimator of $\zeta(\Pr)$ (see Theorem \ref{consistency}), 
from Proposition~\ref{validity}, it is clear that under $H_0$, $\hat\zeta_n$ converges almost surely to zero, but under $H_1$, it converges to a positive constant. Therefore, the power of a test that rejects $H_0$ for higher values of $\hat\zeta_n$ converges to one as the sample size increases. However, it is difficult to find the critical value based on the asymptotic null distribution of $n\hat\zeta_n$ (see Theorem \ref{large-sample-distribution-1}) since it involves an $\ell_2$ sequence $\{\lambda_k\}$, which depends on the underlying distribution $\Pr$. Though our test has a similarity with the test based on Maximum Mean Discrepancy (MMD) \citep[see][]{gretton2012kernel}, it can not be calibrated correctly using the permutation method. This is due to the existing dependence between the observed and the augmented data even under $H_0$. However, under $H_0$, ${\bf X}_i$ and ${\bf X}_i^\prime$ are exchangeable for every $i=1,2,\ldots, n$. Therefore, under $H_0$, a random swap between ${\bf X}_i$ and ${\bf X}_i^\prime$ does not change the distribution of $\hat\zeta_n$. Utilizing this fact, we construct a novel resampling algorithm to compute the cut-off. The algorithm is given below.

\vspace{0.1in}
\noindent\fbox{%
    \parbox{1\textwidth}{%
\vspace{-0.1in}
\begin{enumerate}
    \item[{ }] {\underline {Resampling algorithm}} 
    \item[A.] Given the augmented data $\mathcal{D}^\prime$, compute the test statistic $\hat\zeta_n$.
    \item[B.] Let $\pi = (\pi(1),\ldots, \pi(n))$ be an element in $\{0,~1\}^n$. Define ${{\bf Y}_i} = \pi(i) {\bf X}_i + (1-\pi(i)) {{\bf X}_i^\prime}$ and ${{\bf Y}_i^\prime } = (1-\pi(i)) {{\bf X}_i}+\pi(i) {{\bf X}_i^\prime}$  for $i=1,2,\ldots,n$. Use ${ ({\bf Y}_1,{\bf Y}_1^\prime),\ldots, ({\bf Y}_n,{\bf Y}_n^\prime)}$ to compute $\hat{\zeta}_n(\pi)$, the resampling analogue of $\hat\zeta_n$.
    \item[C.] Repeat step B for all possible $\pi$ to get the critical value for a level $\alpha$ ($0<\alpha<1$) test given by 
    $$c_{1-\alpha} = \inf\{t\in\R: \frac{1}{2^n}\sum_{\pi\in \{0,~1\}^n}\text{I}[\hat\zeta_n(\pi)\leq t]\geq 1-\alpha\}.$$
\end{enumerate}
\vspace{-0.1in}
   }
}

\vspace{0.1in}
Naturally, we reject $H_0$ if $\hat\zeta_n$ is larger than $c_{1-\alpha}$. Note that the p-value of this conditional test is given by 
\begin{align}
    p_n = \frac{1}{2^n}\sum_{\pi\in \{0,~1\}^n}\text{I}[\hat\zeta_n(\pi)\geq \hat\zeta_n].
    \label{eq:pvalue}
\end{align}
So, alternatively, we can reject $H_0$ if $p_n<\alpha$. The following lemma shows that this resampling algorithm gives a valid level $\alpha$ test.

\vspace{0.05in}
\begin{lemma}
    Let $\hat\zeta_n$ be the estimator of $\zeta(\Pr)$ as defined in \eqref{closed-form-expression}. If $p_n$ denotes the conditional p-value as defined in \eqref{eq:pvalue}, then under $H_0: \zeta(\Pr)=0$, we have $\P\{p_n<\alpha\}\leq \alpha$ irrespective of the values of $n$ and $d$.
    \label{level-property}
\end{lemma}

\vspace{0.05in}
Interestingly, we can control the threshold $c_{1-\alpha}$ by a deterministic sequence that does not depend on $d$ and converges to $0$ as $n$ increases. This is shown by the following theorem.

\vspace{0.05in}
\begin{lemma}
    If ${\bf X}_1,{\bf X}_2,\ldots, {\bf X}_n$ are independent copies of a $d$-dimensional random vector ${\bf X}\sim \Pr$, then for any $\alpha~ (0<\alpha<1)$, the inequality $c_{1-\alpha}\leq 2(\alpha (n-1))^{-1}$ holds with probability one.
    \label{cutoff-bound}
\end{lemma}

So, irrespective of the value of $d$, $c_{1-\alpha}$ is of order $O_P(n^{-1})$ and it converges to zero almost surely as $n$ diverges to infinity. 
Since, under $H_1$, $\hat\zeta_n$ converges to a positive number, the conditional p-value $p_n$ converges to zero as $n$ increases. Hence, the power of the resulting conditional test converges to one. This is formally stated in the following theorem.

\vspace{0.05in}
\begin{thm}
    For any fixed alternative, the power of the conditional test based on $p_n$ converges to one as $n$ diverges to infinity.
    \label{test-consistency}
\end{thm}

\vspace{0.05in}
Though the above resampling algorithm leads to a consistent level $\alpha$ test, it has a computational complexity of the order $O(n^22^n)$. So, it is not computationally feasible to implement even if the sample size is moderately large. Therefore, in practice, we generate $\pi_1,\pi_2,\ldots, \pi_B$ uniformly from $\{0,~1\}^n$ and compute the randomized p-value
\begin{align}
    p_{n,B} = \frac{1}{(B+1)}\big(\sum_{i=1}^B I\{\hat\zeta_n(\pi_i)\geq \hat\zeta_n\}+1\big).
\end{align}
We reject $H_0$ if $p_{n,B}< \alpha$. The following theorem shows that $p_{n,B}$ closely approximates $p_n$ for large $B$ and thereby justifies the use $p_{n,B}$ for the practical implementation of the test.

\vspace{0.05in}
{ \begin{thm}
    Given the augmented data $\mathcal{D}^\prime$, for any $\epsilon>0$ and $B\in \mathbb{N}$, we have 
    $$\P\left[|p_{n,B}-p_n|>\epsilon + \frac{1}{B+1}\right] \leq 2\exp\Big\{-2B \epsilon^2\Big\}.$$
    As a consequence, $|p_{n,B}-p_n|\stackrel{a.s.}{\rightarrow}0$ as $B$ diverges to infinity.
    \label{monte-carlo-consistency}
\end{thm}}

{  
\begin{rem}
  Note that Theorem \ref{monte-carlo-consistency} ensures that our Monte Carlo approximation $p_{n,B}$ is very close to the conditional resampling p-value $p_{n}$ even if $B$ is moderately large. Therefore, in practice, choosing a moderately large $B$ will ensure that our test based on the randomized p-value $p_{n,B}$ is close to that based on $p_n$.
  \end{rem}
}

\section{Asymptotic properties of the test}
In this section, we study some large sample properties of the proposed test. First, we investigate the performance of our test against contamination alternatives. Next, we establish its minimax rate optimality against a suitable class of nonparametric alternatives and prove its consistency even when the dimension of data increases with the sample size.
Finally, we show that our test is efficient in the Pittman sense under contiguous contamination alternatives. 

\subsection{{  Performance against contamination alternatives}}
\label{robusteness-section}
Consider a distribution $F$, which is not spherically symmetric. Let ${\bf X}_1,{\bf X}_2,\ldots, {\bf X}_n$ be i.i.d. random vectors from a contaminated distribution $F_\delta = (1-\delta)F+\delta G$ for some $\delta\in (0,1)$. 
The following lemma shows the effect of this contamination on $\zeta(\cdot)$ by providing a relation among $\zeta(F),\zeta(G)$ and $\zeta(F_\delta)$. 

\vspace{0.05in}
\begin{lemma}
    For any $\delta \in(0,1)$, we have 
    $$\zeta(F_\delta) = (1-\delta)^2 \zeta(F)+\delta^2 \zeta(G) + 2\delta(1-\delta) \zeta^\prime (G,F)$$
    where 
    $\zeta^\prime(G,F) = \E\{K({\bf X}_1,{\bf Y}_1)\}+\E\{K({\bf X}_1^\prime, {\bf Y}_1^\prime)\}-\E\{K({\bf X}_1,{\bf Y}_1^\prime)\}-\E\{K({\bf X}_1^\prime,{\bf Y}_1)\}.$ 
    Here ${\bf X}_1 \sim G$ and ${\bf Y}_1\sim F$ are independent, ${\bf X}_1^\prime = \|{\bf X}_1\| {\bf U}_1$, ${\bf Y}_1^\prime = \|{\bf Y}_1\| {\bf U}_2$ for ${\bf U}_1,{\bf U}_2$ being i.i.d. Unif$(\mathcal{S}^{d-1})$ and $K({\bf x},{\bf y}) = \exp\{-\frac{1}{2d}\|{\bf x}-{\bf y}\|^2\}$.
    \label{relation-contamination}
\end{lemma}
Note that if $G$ is spherically symmetric, $\zeta(G)=0$ and $\zeta^\prime(G,F)=0$. Then, for any fixed $\delta$, we have $\zeta(F_\delta) = (1-\delta)^2\zeta(F)$. 
 The following theorem also shows that for any  $\delta \in (0,1)$ and $G$ spherically symmetric, the power of our test against the contaminated alternative $F_{\delta}$ converges to one as the sample size increases. 

\vspace{0.05in}
\begin{thm}
    Let ${\bf X}_1,{\bf X}_2,\ldots, {\bf X}_n$ be independent copies of ${\bf X}\sim F_\delta=(1-\delta)F+\delta G$, where $0<\delta<1$, and $\zeta(F)>0$. Then, the minimum power of the proposed test over the class of the spherical distributions $G$, i.e.,  $\inf\limits_{G: \zeta(G)=0} \P\{\hat\zeta_n>c_{1-\alpha}\}$, converges to one as the sample size diverges to infinity.
    \label{robustness}
\end{thm}

\vspace{0.05in}
This theorem shows that our test is { able to detect the spherical asymmetry of the underlying distribution} for any fixed proportion of contamination { by a non-spherical \linebreak distribution}. Since $\zeta(F_{1-\delta})=\delta^2 \zeta(F)$, it shows the convergence of the power of our test for $F_{1-\delta}$ as well. So, if a sample from a spherically symmetric distribution has a small proportion of contamination by observations from a non-spherical distribution, our test can successfully detect the presence of those contaminations when the sample size is large. The result in Theorem \ref{robustness} holds even for a sequence $\{\delta_n\}$ that remains bounded away from one. However, if that is not the case, the asymptotic power of the test will depend on the convergence rate of $1-\delta_n$ and may yield non-trivial limits for certain choices of $\{\delta_n\}$. This is explored in the following subsection.

\subsection{Minimax rate optimality and high-dimensional behaviour}

Let us consider a testing problem involving a pair of hypotheses $H_0:\zeta(\Pr)=0$ and $H_1^\prime: \zeta(\Pr)>\epsilon(n)$, where $\epsilon(n)$ is a positive number that depends on the sample size $n$. Let $\mathcal{F}(\epsilon(n)):=\{\Pr\mid \zeta(\Pr)>\epsilon(n)\}$ be the class of alternatives under $H_1^\prime$ and $\mathbb{T}_{n}(\alpha)$ be the class of all level $\alpha$ test. The minimax type II error rate for this problem is defined as
$$R_{n}\big(\epsilon(n)\big) = \inf_{\phi\in\mathbb{T}_{n}(\alpha)}\sup_{F\in\mathcal{F}(\epsilon(n))} \P_{F^n}\{\phi = 0\},$$
where $\P_{F^n}$ denotes the probability corresponding to the joint distribution of $({\bf X}_1,{\bf X}_2,\ldots, {\bf X}_n)$, and the ${\bf X}_i$s are independent copies of ${\Xvec}\sim F$. Here we want to find an optimum choice of $\epsilon(n)$ (call it $\epsilon_{0}(n)$) such that the following two conditions hold.
\begin{enumerate}
    \item[(a)] \label{cond_a} 
    For any $0<\beta<1-\alpha$, there exists a constant $c(\alpha,\beta)>0$ such that for all $0 < c < c(\alpha,\beta)$, we have 
    $\liminf\limits_{n \rightarrow \infty} R_{n}(c~\epsilon_0({n})) \geq \beta.$
    
    \item[(b)] \label{cond_b} There exists a level $\alpha$ test $\phi_0$ such that for any $0<\beta<1-\alpha$, there exists $C(\alpha,\beta) > 0$ for which $\limsup\limits_{n \rightarrow \infty} \sup\limits_{F\in\mathcal{F}(c~\epsilon_0({n}))}\P_{F^n}\{\phi_0=0\}\leq \beta$ for all $c>C(\alpha,\beta)$, or in other words, $\limsup\limits_{n \rightarrow \infty} R_{n}(c~\epsilon_0({n})) \leq \beta$ for all $c > C(\alpha,\beta)$. 
  
\end{enumerate}

This optimal rate $\epsilon_0({n})$ is called the minimax rate of separation for the above problem, and the test $\phi_0$ is called the minimax rate optimal test. Theorem \ref{minimax-lower-bound} below shows that here $\epsilon_0(n)$ cannot be of order smaller than $O(n^{-1})$. So, for any $0<\beta<1-\alpha$ and any $\phi\in \mathbb{T}_n(\alpha)$, we can always find a distribution $F$ with $\zeta(F)$ of the order $O(n^{-1})$ or smaller such that the type II error rate of the test $\phi$, i.e., $\P_{F^n}\{\phi=0\}$ is { larger} than $\beta$.

\vspace{0.05in}
\begin{thm}
For $0<\beta<1-\alpha$, there exists a constant $c_0(\alpha,\beta)$ such that the minimax type II error rate $R_{n}(cn^{-1})$ is lower bounded by $\beta$ for all $n$ and all $c\in(0,c_0(\alpha, \beta))$.
\label{minimax-lower-bound}
\end{thm}

\vspace{0.05in}
\begin{rem}
    In particular, consider the distribution $F_{\delta_n} = (1-\delta_n)F+\delta_n G$ where $\zeta(G)=0$ and $\zeta(F)>0$ (note that $\zeta(F_{\delta_n})=(1-\delta_n)^2 \zeta(F)$). 
    If $\delta_n$ is such that $n(1-\delta_n)^2\rightarrow 0$ as $n \rightarrow \infty$ (i.e. $\zeta(F_{\delta_n})$ is of smaller asymptotic order than $O(n^{-1})$), 
    then the power of any level $\alpha$ test will fall below the nominal level $\alpha$. 
    \label{remark4}
\end{rem}

\vspace{0.05in}
In the next theorem, we establish that in the case of $\epsilon_0(n)=n^{-1}$, our test based on ${\hat \zeta}_n$ satisfies the condition (b) stated above. Therefore, these two theorems (Theorem \ref{minimax-lower-bound} and \ref{minimax-upper-bound}) together show that the minimax rate of separation is $\epsilon_0({n})=n^{-1}$, and our proposed test has the minimax rate optimality for the class of alternatives $\mathcal{F}(\epsilon(n))$. 

\vspace{0.05in}
\begin{thm}
For any $\beta \in (0,1-\alpha)$, there exists a constant $C_0(\alpha,\beta)$ (independent of $d$) such that asymptotically the maximum type II error of the test based on $\hat\zeta_n$ over $\mathcal{F}(cn^{-1})$  is uniformly bounded above by $\beta$ for all $c>C_0(\alpha,\beta)$, i.e.,
$\displaystyle \limsup\limits_{n \rightarrow \infty}\sup_{F\in \mathcal{F}(c\lambda({n}))}\P_{F^n}(\hat\zeta_n\leq c_{1-\alpha})\leq \beta$ for all $c>C_0(\alpha,\beta).$
\label{minimax-upper-bound}
\end{thm}

\vspace{0.05in}
{ \begin{rem}
    Consider the same example as in Remark \ref{remark4}. Since $\zeta(F_{\delta_n}) = (1-\delta_n)^2 \zeta(F)$, from the proof of Theorem \ref{minimax-upper-bound}, it can be shown that if $n(1-\delta_n)^2 \rightarrow \infty$ as $n \rightarrow \infty$
    (i.e. $\zeta(F_{\delta_n})$ is of higher asymptotic order than $O(n^{-1})$), the power of our test converges to one. 
\end{rem}}


\vspace{0.05in}
Note that the constant $C(\alpha,\beta)$ in Theorem \ref{minimax-upper-bound} does not depend on the dimension $d$. However, $\zeta(F)$ may vary with the dimension. We know that under certain regularity conditions \citep[see, e.g.,][]{hall2005geometric, ahn2007high, jung2009pca}, as the dimension increases, pairwise distances among the observations  (after appropriate scaling) converge to a constant, and all observations tend to lie on the surface of a sphere of increasing radius. So, in such situations, $\zeta(F)$ converges to $0$ as $d$ increases. Therefore, one may be curious to know how this test will perform if the dimension and the sample size increase simultaneously.
{The following theorem answers this question.}

\vspace{0.05in}
\begin{thm}
    Suppose that ${\bf X}_1,{\bf X}_2,\ldots,{\bf X}_n$ are independent copies of ${\bf X}\sim {F}^{(d)}$, a $d$-dimensional distribution. If $d=d(n)$ grows with the sample size $n$  in such a way that $n\zeta(F^{(d)})$ diverges to infinity as $n$ increases, then the power of the proposed test converges to one as $n$ and $d$ both diverge to infinity. 
    \label{high-dim-consistency}
\end{thm}

\vspace{0.05in}
So, even if $\zeta(F^{(d)})$ converges to $0$ as $d$ increases with $n$, the power of our test converges to one as long as $\zeta(F^{(d)})$ converges at a slower rate than $O(n^{-1})$. However, if the distance convergence does not hold and we have $\liminf_{d\to\infty}\zeta(F^{(d)})>0$, the power of our test converges to one even if the sample size increases at a very slow rate. An example of such a distribution is given in Section \ref{sec:simulation} (see Example 4(a)). For such examples, one can expect the test to have good performance even in the High Dimension Low Sample Size (HDLSS) setup, where $n$ is fixed (but suitably large) and $d$ diverges to infinity. However, in the case of distance concentration in the HDLSS set-up, where we have $\liminf_{d\to\infty}\zeta(F^{(d)})=0$, we need to increase the sample size suitably to get good performance. 
This is further explored in our simulation studies in Section \ref{sec:simulation}. 

\subsection{Pitmann Efficiency}

Now, consider the alternative $F_{1-\beta_n n^{-1/2}} = (1-\beta_n n^{-1/2}) G + (\beta_n n^{-1/2}) F$, but assume that $\beta_n$ is a sequence of positive numbers converging to some $\beta\in(0,\infty)$. Let $f$ and $g$ be the densities corresponding to $F$ and $G$, respectively. To study the asymptotic behaviour of our test for such an alternative, we first study the asymptotic behaviour of $n\hat \zeta_n$ and its resample analog $n\hat\zeta_n(\pi)$. The following result shows that under suitable assumption on $F$ and $G$, the sequence of alternative asymmetric distributions $F_{1-\beta_n n^{-1/2}}$ is contiguous and locally asymptotically normal.

\vspace{0.05in}
\begin{prop}
    Let ${\bf X}_1,\ldots, {\bf X}_n$ be independent copies of ${\bf X}\sim G$. Then, under the assumption that $\int \big(f({\bf u})/g({\bf u})-1\big)^2 g({\bf u}) \mathrm d{\bf u}<\infty$  as $n$ grows to infinity, we have
    $$\left|\log\bigg\{\prod_{i=1}^n\Big(1+\frac{\beta_n}{\sqrt{n}}\Big\{\frac{f({\bf X}_i)}{g({\bf X}_i)}-1\Big\}\Big)\bigg\}-\frac{\beta_n}{\sqrt{n}}\sum_{i=1}^n \bigg\{\frac{f({\bf X}_i)}{g({\bf X}_i)}-1\bigg\}+\frac{\beta_n^2}{2}\E\bigg\{\frac{f({\bf X}_1)}{g({\bf X}_1)}-1\bigg\}^2\right|\stackrel{P}{\rightarrow} 0.$$
    \label{lanfd}
\end{prop}

\vspace{-0.25in}
Now using Proposition \ref{lanfd} and Le Cam's third lemma, we establish the local asymptotic behaviour of $n\hat\zeta_n$ in the following theorem.

\vspace{0.05in}
\begin{thm}
Let ${\bf X}_1,{\bf X}_2,\ldots, {\bf X}_n$ be independent copies of ${\bf X}\sim F_{1-\beta_n n^{-1/2}}$ and ${\bf X}_i^\prime = \|{\bf X}_i\|{\bf U}_i$ where ${\bf U}_1,{\bf U}_2,\ldots,{\bf U}_n$ are i.i.d. Unif$(\mathcal{S}^{d-1})$. Also let $\lambda_k$ be the eigenvalue with corresponding eigenfunction $\psi_k$ ($k=1,2,\ldots$) of the integral equation
$$\E\{g\big(({\bf x}_1,{\bf x}_1^\prime),({\bf X}_1,{\bf X}_1^\prime)\big)\psi_k\big({\bf X}_1,{\bf X}_1^\prime\big)\} = \lambda_k \psi_{k}\big({\bf x}_1,{\bf x}_1^\prime\big),$$
where $g$ is as in equation (\ref{kernel-function}). Then, as $n$ tends to infinity,
$$n\hat\zeta_n\stackrel{D}{\longrightarrow}\sum_{k=1}^\infty \lambda_k \left(\big(Z_k+\beta ~\E_F\{\psi_k({\bf X}_1,{\bf X}_1^\prime)\}\big)^2-1\right),$$
where ${Z_i}$ is a sequence of i.i.d. standard normal random variables.
    \label{local-limit}
\end{thm}

\vspace{0.05in}
Theorem \ref{local-limit} shows that for $\beta>0$, the local limit distribution of $n\hat\zeta_n$ is stochastically larger than its limiting null distribution as in Theorem \ref{large-sample-distribution-1}. Now, the following theorem establishes that the local limiting distribution of the permuted statistic $n\hat\zeta_{n}(\pi)$ under the sequence of alternatives $F_{1-\beta_nn^{-1/2}}$ is identical to the asymptotic null distribution of $n\hat\zeta_n$.

\vspace{0.05in}
\begin{thm}
  Let ${\bf X}_1,{\bf X}_2,\ldots, {\bf X}_n$ be independent copies of ${\bf X}\sim\Pr_n$ and $\hat\zeta_n(\pi)$ be the resampling analog of the proposed test statistic obtained using Algorithm A-C in Section \ref{test}. Then, under any fixed alternative (i.e., $\Pr_n = F$ for some distribution $F$ with $\zeta(F)>0$) or a contiguous alternative (i.e., $\Pr_n = F_{1-\beta_nn^{-1/2}}$), as $n$ grows to infinity, {  given the augmented data $\mathcal{D}^\prime$, the conditional distribution of $n\hat\zeta_{n}(\pi)$ weakly converges to the distribution of $\sum_{k=1}^\infty \lambda_k (Z_k^2-1)$ in probability,} where $\{Z_k\}$ is a sequence of i.i.d standard normal random variables and $\{\lambda_k\}$ is a square summable sequence of real numbers.
 \label{local-limit-per-2}
\end{thm}

Theorems \ref{local-limit} and \ref{local-limit-per-2} together show that under $F_{1-\beta_n n^{-1/2}}$, the power of our test converges to a non-trivial limit, and as $\beta$ starts increasing from zero, the power of our test gradually increases from $\alpha$ to one. This establishes that our test is efficient in the Pitman sense. However, the exact expression of the limit is not analytically tractable.

\begin{table}[t]
    \centering
     \caption{\centering Powers of the proposed test for the alternative $F_{1-\beta_n n^{-1/2}}$
     $= (1-\beta_n n^{-1/2}){\mathcal N}_{10}(0,I)+\beta_n n^{-1/2} {\mathcal N}_{10}(0,0.5I+0.5J)$ when $\beta_n = 5n^{\gamma}$.}

    \begin{tabular}{cccc}
       Sample Size  & $\gamma=-0.1$ & $\gamma=0$ & $\gamma=0.1$\\
    \hline
       $50$  & 0.145 & 0.361 & 0.789\\
       $100$ & 0.147 & 0.375 & 0.932\\
       $250$ & 0.108 & 0.373 & 0.991\\
       $500$ & 0.104 & 0.376 & 0.999\\
    \hline  
    \end{tabular}
   
\label{tab:efficiency}
\end{table}

We now present a small simulation study. We generate $n$ observations from $F_{1-\beta_n n^{-1/2}}$ in $\R^{10}$ where $G$ is the standard normal distribution and $F$ is a normal distribution with mean zero and variance-covariance matrix ${\bf \Sigma} = (0.5) ~{\bf I}_{10}+ (0.5) ~{\bf 1}_{10}{\bf 1}_{10}^{\top}$, for ${\bf 1}_d=(1,1,\ldots,1)^{\top}$ being the $d$-dimensional vector with all elements
equal to one. We consider three different sequences (a) $\beta_n=5 n^{-0.1}$, (b) $\beta_n=5$ and (c) $\beta_n=5 n^{0.1}$ and evaluate the performance of our test against these alternatives. The p-value of the test is approximated using the randomized p-value with $B=500$ and the power of the test is evaluated by the proportion of times it rejects $H_0$ in 1000 repetitions of each experiment. In Table \ref{tab:efficiency}, we see that for case (a), the power of our test shows a decreasing trend with increasing sample size. For case (b), the power exhibits convergence towards $0.37$, which can be considered as the Pitman efficiency of our test when $\beta_n = 5$. For case (c), we see that the power of our test converges to one with increasing sample size. This behavior of our test supports our theoretical findings in this section.


\section{{  Simulation studies}}
\label{sec:simulation}
In this section, we investigate the empirical performance of the proposed test. First, we study its finite sample level property and then compare its empirical power with the tests based on optimal transport \citep[][]{huang2023multivariate}, Monte Carlo method \citep[][]{diks1999} and projection pursuit technique \citep[][]{fang1993}. Henceforth, we refer to these tests as the OT test, the DT test and the PP test respectively. Throughout this article, all tests are considered to have 5\% nominal level. The OT test is distribution-free and the PP test is asymptotically distribution-free. Following the suggestion of the authors, for these two tests we use the cut-offs based on the asymptotic distributions of the corresponding test statistics. Our test and the DT test are calibrated using the resampling method, where the cut-offs are computed based on 500 iterations. Each experiment is repeated 1000 times to estimate the power of a test by the proportion of times it rejects $H_0$. The R codes of all tests are available in the supplementary material.

\begin{figure}[t]
\centering
\begin{tikzpicture}[scale = 0.95]
\begin{axis}[xmin = 1, xmax = 10, ymin = 0, ymax = 0.25, xlabel = {$\log_2(d)$}, ylabel = {Estimates}, title = {\bf Example 1(a)}]
\addplot[color = red,   mark = *, step = 1cm,very thin]coordinates{(1,0.042)(2,0.044)(3,0.046)(4,0.039)(5,0.043)(6,0.047)(7,0.048)(8,0.052)(9,0.048)(10,0.05)};

\addplot[color = purple,   mark = *, step = 1cm,very thin]coordinates{(1,0.037)(2,0.05)(3,0.041)(4,0.051)(5,0.041)(6,0.052)(7,0.039)(8,0.056)(9,0.045)(10,0.035)};

\addplot[color = violet,   mark = *, step = 1cm,very thin]coordinates{(1,0.05)(2,0.048)(3,0.055)(4,0.051)(5,0.052)(6,0.051)(7,0.044)(8,0.049)(9,0.044)(10,0.052)};

\end{axis}
\end{tikzpicture}
\begin{tikzpicture}[scale = 0.95]
\begin{axis}[xmin = 1, xmax = 10, ymin = 0, ymax = 0.25, xlabel = {$\log_2(d)$}, ylabel = {Estimates}, title = {\bf Example 1(b)}]
\addplot[color = red,   mark = *, step = 1cm,very thin]coordinates{(1,0.05)(2,0.05)(3,0.041)(4,0.05)(5,0.063)(6,0.063)(7,0.054)(8,0.05)(9,0.044)(10,0.049)};

\addplot[color = purple,   mark = *, step = 1cm,very thin]coordinates{(1,0.047)(2,0.052)(3,0.044)(4,0.049)(5,0.044)(6,0.052)(7,0.052)(8,0.054)(9,0.046)(10,0.052)};

\addplot[color = violet,   mark = *, step = 1cm,very thin]coordinates{(1,0.056)(2,0.051)(3,0.065)(4,0.063)(5,0.053)(6,0.051)(7,0.043)(8,0.053)(9,0.058)(10,0.043)};

\end{axis}
\end{tikzpicture}
\begin{tikzpicture}[scale = 0.95]
\begin{axis}[xmin = 1, xmax = 10, ymin = 0, ymax = 0.25, xlabel = {$\log_2(d)$}, ylabel = {Estimates}, title = {\bf Example 1(c)}]
\addplot[color = red,   mark = *, step = 1cm,very thin]coordinates{(1,0.047)(2,0.056)(3,0.057)(4,0.052)(5,0.053)(6,0.043)(7,0.041)(8,0.05)(9,0.047)(10,0.042)};

\addplot[color = purple,   mark = *, step = 1cm,very thin]coordinates{(1,0.061)(2,0.057)(3,0.054)(4,0.048)(5,0.043)(6,0.05)(7,0.053)(8,0.044)(9,0.046)(10,0.059)};

\addplot[color = violet,   mark = *, step = 1cm,very thin]coordinates{(1,0.057)(2,0.064)(3,0.054)(4,0.055)(5,0.059)(6,0.057)(7,0.041)(8,0.049)(9,0.049)(10,0.051)};

\end{axis}
\end{tikzpicture}    
\caption{Observed levels of the proposed test for observations generated from the standard (a) Gaussian, (b) Cauchy, and (c) $t_4$ distributions with sample size $n = 20$(\textcolor{red}{$\tikzcircle{2pt}$}), $n=40$ (\textcolor{violet}{$\tikzcirclev{2pt}$}) and $n=60$ (\textcolor{pink}{$\tikzcirclep{2pt}$}) in dimensions $d=2^i, i=1,2,\ldots,10$.}
    \label{fig:level}
\end{figure}

First, we investigate the level property of our test by generating random samples from some spherically symmetric distributions. In particular, we consider the standard multivariate (a) Gaussian, (b) Cauchy and (c) $t_4$ ($t$ distribution with $4$ degrees of freedom) distributions and call them Examples 1(a), 1(b) and 1(c), respectively. In each case, we compute the powers for different sample sizes ($n = 20,40,60$) and dimensions ($d=2^i, i=1,2,\ldots, 10$), and they are reported in Figure \ref{fig:level}. This figure clearly shows that for all three distributions, our test rejects $H_0$ in nearly 5\% cases. 
The other three competing tests also exhibit similar behaviour, but to avoid repetition we do not report them here.

\begin{figure}[t]
\centering
    
\begin{tikzpicture}
\begin{axis}[xmin = 0, xmax = 0.91, ymin = 0, ymax = 1, xlabel = {$\rho$}, ylabel = {Estimates}, title = {\bf Example 2(a)}]
\addplot[color = red,   mark = *, step = 1cm,very thin]coordinates{(0,0.052)(0.1,0.064)(0.3,0.313)(0.5,0.927)(0.7,1)(0.9,1)};

\addplot[color = blue,   mark = diamond*, step = 1cm,very thin]coordinates{(0,0.052)(0.1,0.05)(0.3,0.054)(0.5,0.078)(0.7,0.1)(0.9,0.117)};

\addplot[color = black,   mark = star, step = 1cm,very thin]coordinates{(0,0.039)(0.1,0.053)(0.3,0.15)(0.5,0.509)(0.7,0.969)(0.9,1)};

\addplot[color = applegreen,   mark = square*, step = 1cm,very thin]coordinates{(0,0.054)(0.1,0.061)(0.3,0.055)(0.5,0.067)(0.7,0.072)(0.9,0.069)};


\end{axis}
\end{tikzpicture}
\begin{tikzpicture}
\begin{axis}[xmin = 0, xmax = 0.91, ymin = 0, ymax = 1, xlabel = {$\rho$}, ylabel = {Estimates}, title = {\bf Example 2(b)}]
\addplot[color = red,   mark = *, step = 1cm,very thin]coordinates{(0,0.053)(0.1,0.063)(0.3,0.19)(0.5,0.652)(0.7,0.99)(0.9,1)};

\addplot[color = blue,   mark = diamond*, step = 1cm,very thin]coordinates{(0,0.045)(0.1,0.044)(0.3,0.053)(0.5,0.07)(0.7,0.086)(0.9,0.12)};

\addplot[color = black,   mark = star, step = 1cm,very thin]coordinates{(0,0.052)(0.1,0.059)(0.3,0.118)(0.5,0.324)(0.7,0.785)(0.9,1)};

\addplot[color = applegreen,   mark = square*, step = 1cm,very thin]coordinates{(0,0.041)(0.1,0.045)(0.3,0.049)(0.5,0.062)(0.7,0.063)(0.9,0.068)};

\end{axis}
\end{tikzpicture}
\begin{tikzpicture}
\begin{axis}[xmin = 0, xmax = .91, ymin = 0, ymax = 1, xlabel = {$\rho$}, ylabel = {Estimates}, title = {\bf Example 2(c)}]
\addplot[color = red,   mark = *, step = 1cm,very thin]coordinates{(0,0.05)(0.1,0.063)(0.3,0.309)(0.5,0.871)(0.7,0.998)(0.9,1)};

\addplot[color = blue,   mark = diamond*, step = 1cm,very thin]coordinates{(0,0.047)(0.1,0.05)(0.3,0.064)(0.5,0.079)(0.7,0.099)(0.9,0.117)};

\addplot[color = black,   mark = star, step = 1cm,very thin]coordinates{(0,0.049)(0.1,0.058)(0.3,0.135)(0.5,0.424)(0.7,0.904)(0.9,1)};

\addplot[color = applegreen,   mark = square*, step = 1cm,very thin]coordinates{(0,0.052)(0.1,0.056)(0.3,0.06)(0.5,0.061)(0.7,0.058)(0.9,0.065)};

\end{axis}
\end{tikzpicture}
\caption{Powers of the proposed test (\textcolor{red}{$\tikzcircle{2pt}$}), OT test (\textcolor{blue}{$\blacklozenge$}), DT test (\textcolor{black}{$\star$}) and PP test (\textcolor{applegreen}{$\blacksquare$}) in Examples 2(a)-(c).}
    \label{fig:power-1}
\end{figure}

\begin{figure}[t]
    \centering
    \begin{tikzpicture}
\begin{axis}[xmin = 20, xmax = 500, ymin = 0, ymax = 1, xlabel = {$n$}, ylabel = {Estimates}, title = {\bf Example 3(a)}]
\addplot[color = red,   mark = *, step = 1cm,very thin]coordinates{(20,0.067)(60,0.12)(100,0.175)(140,0.243)(180,0.356)(220,0.393)(260,0.523)(300,0.652)(340,0.745)(380,0.81)(420,0.86)(460,0.905)(500,0.951)};

\addplot[color = blue,   mark = diamond*, step = 1cm,very thin]coordinates{(20,0.068)(60,0.042)(100,0.057)(140,0.051)(180,0.052)(220,0.057)(260,0.058)(300,0.049)(340,0.053)(380,0.055)(420,0.047)(460,0.066)(500,0.06)};

\addplot[color = black,   mark = star, step = 1cm,very thin]coordinates{(20,0.079)(60,0.113)(100,0.116)(140,0.154)(180,0.155)(220,0.151)(260,0.158)(300,0.166)(340,0.179)(380,0.195)(420,0.181)(460,0.183)(500,0.187)};

\addplot[color = applegreen,   mark = square*, step = 1cm,very thin]coordinates{(20,0.055)(60,0.055)(100,0.056)(140,0.045)(180,0.052)(220,0.047)(260,0.059)(300,0.046)(340,0.052)(380,0.053)(420,0.05)(460,0.047)(500,0.067)};


\end{axis}
\end{tikzpicture}
\begin{tikzpicture}
\begin{axis}[xmin = 20, xmax = 500, ymin = 0, ymax = 1, xlabel = {$n$}, ylabel = {Estimates}, title = {\bf Example 3(b)}]
\addplot[color = red,   mark = *, step = 1cm,very thin]coordinates{(20,0.052)(60,0.1)(100,0.112)(140,0.131)(180,0.197)(220,0.21)(260,0.279)(300,0.332)(340,0.394)(380,0.493)(420,0.534)(460,0.605)(500,0.666)};

\addplot[color = blue,   mark = diamond*, step = 1cm,very thin]coordinates{(20,0.054)(60,0.048)(100,0.052)(140,0.051)(180,0.047)(220,0.051)(260,0.06)(300,0.047)(340,0.056)(380,0.06)(420,0.051)(460,0.06)(500,0.051)};

\addplot[color = black,   mark = star, step = 1cm,very thin]coordinates{(20,0.082)(60,0.102)(100,0.147)(140,0.162)(180,0.2)(220,0.241)(260,0.288)(300,0.382)(340,0.462)(380,0.48)(420,0.591)(460,0.632)(500,0.704)};

\addplot[color = applegreen,   mark = square*, step = 1cm,very thin]coordinates{(20,0.05)(60,0.052)(100,0.046)(140,0.044)(180,0.047)(220,0.042)(260,0.057)(300,0.051)(340,0.05)(380,0.052)(420,0.046)(460,0.051)(500,0.055)};

\end{axis}
\end{tikzpicture}

\begin{tikzpicture}
\begin{axis}[xmin = 20, xmax = 500, ymin = 0, ymax = 1, xlabel = {$n$}, ylabel = {Estimates}, title = {\bf Example 3(c)}]
\addplot[color = red,   mark = *, step = 1cm,very thin]coordinates{(20,0.038)(60,0.081)(100,0.115)(140,0.145)(180,0.228)(220,0.249)(260,0.298)(300,0.377)(340,0.472)(380,0.557)(420,0.6)(460,0.686)(500,0.753)};

\addplot[color = blue,   mark = diamond*, step = 1cm,very thin]coordinates{(20,0.049)(60,0.046)(100,0.04)(140,0.052)(180,0.055)(220,0.054)(260,0.048)(300,0.061)(340,0.052)(380,0.048)(420,0.052)(460,0.046)(500,0.054)};

\addplot[color = black,   mark = star, step = 1cm,very thin]coordinates{(20,0.055)(60,0.073)(100,0.079)(140,0.08)(180,0.095)(220,0.111)(260,0.146)(300,0.142)(340,0.135)(380,0.162)(420,0.165)(460,0.188)(500,0.239)};

\addplot[color = applegreen,   mark = square*, step = 1cm,very thin]coordinates{(20,0.041)(60,0.042)(100,0.048)(140,0.048)(180,0.05)(220,0.059)(260,0.046)(300,0.058)(340,0.052)(380,0.046)(420,0.051)(460,0.055)(500,0.05)};

\end{axis}
\end{tikzpicture}
\begin{tikzpicture}
\begin{axis}[xmin = 20, xmax = 200, ymin = 0, ymax = 1, xlabel = {$n$}, ylabel = {Estimates}, title = {\bf Example 3(d)}]
\addplot[color = red,   mark = *, step = 1cm,very thin]coordinates{(20,0.142)(60,0.598)(100,0.94)(140,0.99)(180,1)(220,1)};

\addplot[color = blue,   mark = diamond*, step = 1cm,very thin]coordinates{(20,0.065)(60,0.064)(100,0.063)(140,0.066)(180,0.067)(220,0.063)};

\addplot[color = black,   mark = star, step = 1cm,very thin]coordinates{(20,0.106)(60,0.196)(100,0.397)(140,0.609)(180,0.802)(220,0.898)};

\addplot[color = applegreen,   mark = square*, step = 1cm,very thin]coordinates{(20,0.052)(60,0.046)(100,0.063)(140,0.059)(180,0.052)(220,0.053)};

\end{axis}
\end{tikzpicture}
\caption{Powers of the proposed test (\textcolor{red}{$\tikzcircle{2pt}$}), OT test (\textcolor{blue}{$\blacklozenge$}), DT test (\textcolor{black}{$\star$}) and PP test (\textcolor{applegreen}{$\blacksquare$}) in Examples 3(a)-(d).}
\label{fig:power-2}
\end{figure}

To compare the empirical powers of different tests, again we consider examples involving (a) Gaussian, (b) Cauchy and (c) $t_4$ distributions with the centre at the origin, but this time we consider the scatter matrix of the form $(1-\rho) ~{\bf I}_d+\rho ~{\bf 1}_d{\bf 1}^{\top}_d$. where 
$\rho\in (0,1)$ (call them Examples 2(a), 2(b) and 2(c), respectively). Note that here the distributions are elliptically symmetric. For each example, we consider $d=5$ and carry out different tests based on $100$ observations. Powers of these tests are reported in Figure \ref{fig:power-1}. As $\rho$ increases from zero to one (i.e., the distribution deviates more from sphericity) one would expect the power of a test to increase. But, for OT and PP tests, these increments are negligible. In these examples, the proposed test has the best performance followed by the DT test.


Next, we consider four examples (call them Examples 3(a)-(d)) involving symmetric but non-elliptic distributions. In Examples 3(a) and 3(b), we deal with $\ell_p$-symmetric distributions \citep[see, e.g.,][]{gupta1997lp,dutta2011} with $p=\infty$ and $p=1$, respectively. In both cases, we generate observations on  ${\bf X} = R {\bf U}$, where ${\bf U}$ and $R \sim \text{Unif}(9,10)$ are independent. In Example 3(a), we have ${\bf U} = {\bf Y}/\|{\bf Y}\|_\infty$ where ${\bf Y} = (Y_1,\ldots, Y_5)$ is uniformly distributed over the 5-dimensional unit hypercube $\{{\bf y}=(y_1,y_2.\ldots,y_5)^{\top}: \max\{|y_1|,|y_2|,\ldots,|y_5|\}\leq 1\}$, while in Example 3(b), we have ${\bf U} = {\bf Y}/\|{\bf Y}\|_1$ where $Y_1,\ldots, Y_5$ are independent standard Laplace variables. We carry out our experiment with different sample sizes, and the results are reported in Figure \ref{fig:power-2}. In these examples, OT and PP tests have powers close to the nominal level of 0.05. In Example 3(b), the proposed test and the DT have comparable performance, but in Example 3(a), our test significantly outperforms the DT test.

\begin{figure}[h]
\centering
\begin{tikzpicture}[scale = 0.95]
\begin{axis}[xmin = 1, xmax = 5, ymin = 0, ymax = 1, xlabel = {$\log_2(d)$}, ylabel = {Estimates}, title = {\bf Example 4 (a)}]
\addplot[color = red,   mark = *, step = 1cm,very thin]coordinates{(1,0.054)(2,0.098)(3,0.239)(4,0.481)(5,0.686)};

\addplot[color = blue,   mark = diamond*, step = 1cm,very thin]coordinates{(1,0.052)(2,0.056)(3,0.099)(4,0.13)(5,0.159)};

\addplot[color = black,   mark = star, step = 1cm,very thin]coordinates{(1,0.077)(2,0.069)(3,0.088)(4,0.136)(5,0.194)};

\addplot[color = applegreen,   mark = square*, step = 1cm,very thin]coordinates{(1,0.051)(2,0.051)(3,0.057)(4,0.044)(5,0.038)};


\end{axis}
\end{tikzpicture}
\begin{tikzpicture}[scale = 0.95]
\begin{axis}[xmin = 1, xmax = 5, ymin = 0, ymax = 1, xlabel = {$\log_2(d)$}, ylabel = {Estimates}, title = {\bf Example 4 (b)}]
\addplot[color = red,   mark = *, step = 1cm,very thin]coordinates{(1,0.044)(2,0.051)(3,0.066)(4,0.058)(5,0.07)};

\addplot[color = blue,   mark = diamond*, step = 1cm,very thin]coordinates{(1,0.051)(2,0.046)(3,0.06)(4,0.069)(5,0.082)};

\addplot[color = black,   mark = star, step = 1cm,very thin]coordinates{(1,0.066)(2,0.056)(3,0.055)(4,0.062)(5,0.079)};

\addplot[color = applegreen,   mark = square*, step = 1cm,very thin]coordinates{(1,0.051)(2,0.04)(3,0.047)(4,0.039)(5,0.031)};

\end{axis}
\end{tikzpicture}
\begin{tikzpicture}[scale = 0.95]
\begin{axis}[xmin = 1, xmax = 5, ymin = 0, ymax = 1, xlabel = {$\log_2(d)$}, ylabel = {Estimates}, title = {\bf Example 4 (c)}]
\addplot[color = red,   mark = *, step = 1cm,very thin]coordinates{(1,0.045)(2,0.07)(3,0.093)(4,0.235)(5,0.729)};

\addplot[color = blue,   mark = diamond*, step = 1cm,very thin]coordinates{(1,0.042)(2,0.044)(3,0.066)(4,0.071)(5,0.082)};

\addplot[color = black,   mark = star, step = 1cm,very thin]coordinates{(1,0.069)(2,0.06)(3,0.058)(4,0.071)(5,0.082)};

\addplot[color = applegreen,   mark = square*, step = 1cm,very thin]coordinates{(1,0.064)(2,0.042)(3,0.047)(4,0.053)(5,0.056)};


\end{axis}
\end{tikzpicture}
\caption{Results of the proposed test (\textcolor{red}{$\tikzcircle{2pt}$}), OT test (\textcolor{blue}{$\blacklozenge$}), DT test (\textcolor{black}{$\star$}) and PP test (\textcolor{applegreen}{$\blacksquare$}) for Example 4(a)-(c).}
    \label{fig:power-3}
\end{figure}

In Example 3(c), we consider an angular symmetric distribution, We generate observations on ${\bf X}=R{\bf U}$, where ${\bf U} \sim \text{Unif}(\mathcal S^4)$ but $R$ and ${\bf U}$ are not independent. Here for any given ${\bf U}= {\bf u} (= (u_1,u_2,\ldots,u_5)^{\top})$ the conditional distribution of $R$ is uniform on $(0, \theta_{\bf u})$, where $\theta_{\bf u} = 10 ~{\mathrm I}[u_1u_2>0]+50 ~{\mathrm I}[u_1u_2\leq 0; u_3u_4u_5>0]+100 ~{\mathrm I}[u_1u_2\leq 0; u_3u_4u_5\leq 0]$, and $\mathrm I[\cdot]$ is the indicator function. In Example 3(d), observations are generated from an equal mixture of four normal distributions with the same dispersion matrix ${\bf I}_5$ and mean vectors ${\bf 1}_5$,$-{\bf 1}_5$, $\betavec=(1,-1,1,-1,1)^{\top}$ and $-\betavec$, respectively. In Figure \ref{fig:power-2}, we see that in these examples also, our test outperforms its competitors. The DT test has the second-best performance but its power is much lower compared to our proposed test.



Finally, we consider some high-dimensional examples. In Examples 4(a) and 4(b), we generate $n=20$ observations from a normal spiked covariance model \citep[see.][]{johnstone} with mean zero and a diagonal covariance matrix with entries $(d,1,1,\ldots, 1)$ and $(d^{0.5},1,1,\ldots, 1)$, respectively.
We carry out our experiment with different choices of $d$, and the results are reported in Figure \ref{fig:power-3}. This figure shows that in Example 4(a) the power curve of our test exhibits a sharp increasing trend with increasing dimensions, while the other tests have non-satisfactory performances. But in Example 4(b), all tests including ours perform poorly. Note that in Example 4(a) and 4(b), the measure of sphericity \citep[see, e.g.,][]{john,jung2009pca} converges to $0$ and $1$, respectively, as $d$ increases. So, in high dimension, the data cloud in Example 4(b) appears out to be similar to that from a spherical distribution, whereas in Example 4(a), it has significant deviations from sphericity. This explains the diametrically opposite behaviour of our test in these two examples. However, Theorem \ref{high-dim-consistency} suggests that even in Example 4(b), our test can perform well if we allow the sample size to increase with the dimensions at a suitable rate. We observe this in {Example 4(c)}, where we consider the same model as in Example 4(b), but increase the sample size with the dimension. Here, we consider  $n = 20 + [d^{1.5}]$, where $[t]$ denotes the largest integer smaller than or equal to $t$. Figure \ref{fig:power-3} shows a sharp increasing trend in the power curve of our test. But the other competing tests have poor performance even in this set-up. This clearly indicates the superiority of our test over the OT, DT, and PP tests for high-dimensional data.

{ 
\section{Test of spherical symmetry about an unknown center}
\label{sec:unknown-center}
In the previous sections, we assumed that the null hypothesis specifies the center of symmetry (without loss of generality, it was taken to be the origin).
If it is not specified, one can think of getting an
estimate of a reasonable measure of centrality (call it $\thetavec$) and implementing the test based on centered observations. 
The usual moment-based estimate of location can be used for this purpose. One can also use spatial median \citep[see, e.g.,][]{chaudhuri1996,koltchinskii1997} or other robust estimates \citep[see, e.g.,][]{rousseeuw1985,rousseeuw1999}.
If ${\hat \thetavec}$ is a consistent estimator of its population analog $\thetavec$, the test statistic computed based on the centered observations 
$\Xvec_1-{\hat \thetavec},\Xvec_2-{\hat \thetavec}, \ldots, \Xvec_n-{\hat \thetavec}$ consistently estimates the population measure of spherical asymmetry about $\theta$, and the corresponding cut-off satisfies similar property as in Lemma \ref{cutoff-bound}. This is asserted by the following theorem.

\begin{thm}
\label{thm:consistent-estimation-centering}
    Let ${\bf X}_1,\ldots, {\bf X}_n$ be independent copies of $\Xvec \sim \Pr$, which has an unknown measure of centrality $\thetavec$. Let  $\hat{\bm \theta}$ be a consistent estimator of ${\bm \theta}$ and $\Tilde{\zeta}_n$ be our proposed test statistic computed based on the centered data $\{{\bf X}_i - \hat{\bm \theta}\}_{1\leq i\leq n}$. Then, as $n$ grows to infinity, $\Tilde{\zeta}_n$ converges in probability to $\zeta(\Pr_{{\bf X}-\bm \theta})$, where $\Pr_{{\bf X}-\bm \theta}$ is the distribution of ${\bf X}-\bm \theta$. Also, the resampling cutoff $c_{1-\alpha}$ based on $\{{\bf X}_i - \hat{\bm \theta}\}_{1\leq i\leq n}$ satisfies $c_{1-\alpha}\leq 2(\alpha(n-1))^{-1}$ almost surely.
\end{thm}

Theorem \ref{thm:consistent-estimation-centering} clearly suggests that the test based on the centered data is large sample consistent against general alternatives. However, the proper calibration of the test is challenging. 
To demonstrate this, we consider some simulated examples involving bivariate distributions. In Example 5(a), we consider the standard normal distribution, while in Example 5(b), we use its truncated version, which has the support $\{\xvec \in {\mathbb R}^2:~ \|\xvec\|\geq 1\}$. In Example 5(c), we consider an equal mixture of $\mathcal{U}(2,3)$ and $\mathcal{U}(10,11)$, where $\mathcal{U}(a,b)$ denotes the uniform distribution on the region $\{\xvec \in {\mathbb R}^2: ~ a\leq \|\xvec\|\leq b\}$. In each case, we assume the center of symmetry to be unknown and carry out our experiment based centered observations. In all cases, instead of sample mean, the sample spatial median is used for centering because of its better robustness properties.
We consider samples of different sizes and in each case, the experiment is repeated $1000$ times to compute the power of the test. Note that in all three examples, the underlying distributions are spherically symmetric. So, any reasonable test is supposed to have powers close to the nominal level of 0.05. But the observed power (i.e., the proportion of times $H_0$ is rejected) of the proposed test and those of the other competing tests (OT, DT and PP tests) reported in Figure \ref{fig:power-4} show a completely different picture. 

\begin{figure}[t]
\centering

\begin{tikzpicture}[scale = 0.9]
\begin{axis}[xmin = 100, xmax = 500, ymin = 0, ymax = 0.25, xlabel = {Sample Size}, ylabel = {Estimates}, title = {\bf Example 5 (a)}]
\addplot[color = red,   mark = *, step = 1cm,very thin]coordinates{(100,0.014)(200,0.008)(300,0.01)(400,0.012)(500,0.012)};

\addplot[color = blue,   mark = diamond*, step = 1cm,very thin]coordinates{(100,0.007)(200,0.006)(300,0.009)(400,0.005)(500,0.003)};

\addplot[color = black,   mark = star, step = 1cm,very thin]coordinates{(100,0.02)(200,0.013)(300,0.018)(400,0.018)(500,0.019)};

\addplot[color = applegreen,   mark = square*, step = 1cm,very thin]coordinates{(100,0.003)(200,0.002)(300,0.004)(400,0.003)(500,0.002)};

\addplot[color = black, step = 1cm,very thin, dashed]coordinates{(100,0.05)(200,0.05)(300,0.05)(400,0.05)(500,0.05)};

\end{axis}
\end{tikzpicture}    
\begin{tikzpicture}[scale = 0.9]
\begin{axis}[xmin = 100, xmax = 500, ymin = 0, ymax = 1, xlabel = {Sample Size}, ylabel = {Estimates}, title = {\bf Example 5 (b)}]
\addplot[color = red,   mark = *, step = 1cm,very thin]coordinates{(100,0.018)(200,0.023)(300,0.012)(400,0.021)(500,0.026)};

\addplot[color = blue,   mark = diamond*, step = 1cm,very thin]coordinates{(100,0.254)(200,0.358)(300,0.36)(400,0.351)(500,0.398)};

\addplot[color = black,   mark = star, step = 1cm,very thin]coordinates{(100,0.419)(200,0.471)(300,0.469)(400,0.441)(500,0.46)};

\addplot[color = applegreen,   mark = square*, step = 1cm,very thin]coordinates{(100,0.071)(200,0.096)(300,0.083)(400,0.083)(500,0.095)};

\addplot[color = black, step = 1cm,very thin, dashed]coordinates{(100,0.05)(200,0.05)(300,0.05)(400,0.05)(500,0.05)};

\end{axis}
\end{tikzpicture}
\begin{tikzpicture}[scale = 0.9]
\begin{axis}[xmin = 100, xmax = 500, ymin = 0, ymax = 1, xlabel = {Sample Size}, ylabel = {Estimates}, title = {\bf Example 5 (c)}]
\addplot[color = red,   mark = *, step = 1cm,very thin]coordinates{(100,0.309)(200,0.345)(300,0.346)(400,0.353)(500,0.355)};

\addplot[color = blue,   mark = diamond*, step = 1cm,very thin]coordinates{(100,0.255)(200,0.64)(300,0.721)(400,0.781)(500,0.798)};

\addplot[color = black,   mark = star, step = 1cm,very thin]coordinates{(100,0.821)(200,0.841)(300,0.844)(400,0.841)(500,0.839)};

\addplot[color = applegreen,   mark = square*, step = 1cm,very thin]coordinates{(100,0.011)(200,0.023)(300,0.018)(400,0.021)(500,0.025)};

\addplot[color = black, step = 1cm,very thin, dashed]coordinates{(100,0.05)(200,0.05)(300,0.05)(400,0.05)(500,0.05)};
\end{axis}
\end{tikzpicture}
\caption{Powers of the proposed test (\textcolor{red}{$\tikzcircle{2pt}$}), the OT test (\textcolor{blue}{$\blacklozenge$}), the DT test (\textcolor{black}{$\star$}) and the PP test (\textcolor{applegreen}{$\blacksquare$}) in Examples 5(a)-(c).}
    \label{fig:power-4}
\end{figure}

In Example 5(a), all tests satisfy the level property (had powers below $0.05$), but they are much conservative. However, in Example 5(b), only our test turns out to be a valid level $\alpha$ test ($\alpha=0.05$).   Powers of all other tests are much higher than the nominal level. In Example 5(c), our test also fails to satisfy the level property. 
These examples show that 
none of these tests based on centered observations can be considered as a valid 
$\alpha$ test, and hence they can yield misleading inferences.
To address this issue, here we propose a modification, which is motivated by the following result.
\begin{lemma}
\label{thm:key-transformation}
    Let ${\bf X}_1$ and ${\bf X}_2$ be two independent copies of $\Xvec \sim P_{\thetavec}$, which is symmetric about $\thetavec \in {\mathbb R}^d$ (i.e., ${\bf X}-\thetavec \stackrel{D}{=} \thetavec-{\bf X}$). Then ${\bf X}_1-{\bf X}_2$ follows a spherically symmetric distribution (about the origin) if and only if $P_{\thetavec}$  is spherically symmetric about $\thetavec$.
\end{lemma}

So, if the underlying distribution is symmetric about an unknown location $ \thetavec$, we can divide the dataset $\mathcal{D}$ into two equal halves (ignore one observation, if needed), 
$\mathcal{D}_1 = \{{\bf X}_i\}_{1\leq i\leq m}$ and $\mathcal{D}_2 = \{{\bf X}_i\}_{m+1\leq i \leq 2m}$, (where $m=\lfloor n/2 \rfloor$) and apply our test on the observational differences $\mathcal{Z} = \{{\bf Z}_i = {\bf X}_i - {\bf X}_{m+i}\}_{1\leq i\leq m}$. Clearly, the resulting test will have the desired level property. We can also establish its minimax rate optimality property with respect to the distribution of the difference. 
To investigate its empirical performance,in Examples 6 (a)-(c), we generate random observations on $\Yvec = \sigmat^{1/2} \Xvec$, where $\sigmat$ is a $2\times2$ matrix with diagonals one and off diagonals $0.5$ and $\Xvec$ follows the same distributions as in Examples 5(a)-(c), respectively. We compute the power of the tests both based on (i) centered (using sample spatial median) observations and (ii) pairwise differences of the observations as discussed above. Here also, the experiment is carried out for various sample sizes, and the results are displayed in Figure \ref{fig:power-5}.

The top row in Figure \ref{fig:power-5} shows the results for tests based on centered observations. The DT test and the proposed test have much higher powers compared to others, with the former having an edge in Examples 6(a) and 6(b). But note that while our test maintains the level property in Examples 5(a) and 5(b),  the DT test fails in Example 5(b). In Example 5(c), both of them fail to satisfy the level property. So, we cannot rely on the results in Example 6(c) for fair comparison.

The results for the tests based on the pairwise differences of the observations (reported in the bottom row)
are more reliable in this context since they satisfy the level property. Here we observe that in all three examples, while the powers of the proposed test and the DT test steadily increase with the sample size, OT and PP tests do not have satisfactory performance. In terms of overall performance, the proposed test has an edge.


\begin{figure}[t]
\centering

\begin{tikzpicture}[scale = 0.9]
\begin{axis}[xmin = 100, xmax = 500, ymin = 0, ymax = 1, xlabel = {Sample Size}, ylabel = {Estimates}, title = {\bf Example 6 (a)}]
\addplot[color = red,   mark = *, step = 1cm,very thin]coordinates{(100,0.189)(200,0.591)(300,0.89)(400,0.982)(500,1)};

\addplot[color = blue,   mark = diamond*, step = 1cm,very thin]coordinates{(100,0.004)(200,0.004)(300,0.007)(400,0.002)(500,0.002)};

\addplot[color = black,   mark = star, step = 1cm,very thin]coordinates{(100,0.402)(200,0.839)(300,0.984)(400,0.997)(500,1)};

\addplot[color = applegreen,   mark = square*, step = 1cm,very thin]coordinates{(100,0.003)(200,0.002)(300,0.003)(400,0.002)(500,0)};


\end{axis}
\end{tikzpicture}    
\begin{tikzpicture}[scale = 0.9]
\begin{axis}[xmin = 100, xmax = 500, ymin = 0, ymax = 1, xlabel = {Sample Size}, ylabel = {Estimates}, title = {\bf Example 6 (b)}]
\addplot[color = red,   mark = *, step = 1cm,very thin]coordinates{(100,0.555)(200,0.916)(300,0.996)(400,1)(500,1)};

\addplot[color = blue,   mark = diamond*, step = 1cm,very thin]coordinates{(100,0.162)(200,0.236)(300,0.224)(400,0.258)(500,0.267)};

\addplot[color = black,   mark = star, step = 1cm,very thin]coordinates{(100,0.901)(200,0.995)(300,1)(400,1)(500,1)};

\addplot[color = applegreen,   mark = square*, step = 1cm,very thin]coordinates{(100,0.035)(200,0.056)(300,0.05)(400,0.057)(500,0.052)};


\end{axis}
\end{tikzpicture}
\begin{tikzpicture}[scale = 0.9]
\begin{axis}[xmin = 100, xmax = 700, ymin = 0, ymax = 1, xlabel = {Sample Size}, ylabel = {Estimates}, title = {\bf Example 6 (c)}]
\addplot[color = red,   mark = *, step = 1cm,very thin]coordinates{(100,0.881)(200,0.999)(300,1)(400,1)(500,1)(600,1)(700,1)};

\addplot[color = blue,   mark = diamond*, step = 1cm,very thin]coordinates{(100,0.011)(200,0.038)(300,0.056)(400,0.072)(500,0.102)(600,0.105)(700,0.153)};

\addplot[color = black,   mark = star, step = 1cm,very thin]coordinates{(100,1)(200,1)(300,1)(400,1)(500,1)(600,1)(700,1)};

\addplot[color = applegreen,   mark = square*, step = 1cm,very thin]coordinates{(100,0.003)(200,0.012)(300,0.002)(400,0.005)(500,0.006)(600,0.002)(700,0.008)};

\end{axis}
\end{tikzpicture}

\begin{tikzpicture}[scale = 0.9]
\begin{axis}[xmin = 100, xmax = 500, ymin = 0, ymax = 1, xlabel = {Sample Size}, ylabel = {Estimates}, title = {\bf Example 6 (a)}]
\addplot[color = red,   mark = *, step = 1cm,very thin]coordinates{(100,0.153)(200,0.364)(300,0.59)(400,0.762)(500,0.89)};

\addplot[color = blue,   mark = diamond*, step = 1cm,very thin]coordinates{(100,0.057)(200,0.043)(300,0.059)(400,0.055)(500,0.059)};

\addplot[color = black,   mark = star, step = 1cm,very thin]coordinates{(100,0.176)(200,0.344)(300,0.582)(400,0.752)(500,0.883)};

\addplot[color = applegreen,   mark = square*, step = 1cm,very thin]coordinates{(100,0.046)(200,0.056)(300,0.06)(400,0.052)(500,0.058)};


\end{axis}
\end{tikzpicture}    
\begin{tikzpicture}[scale = 0.9]
\begin{axis}[xmin = 100, xmax = 500, ymin = 0, ymax = 1, xlabel = {Sample Size}, ylabel = {Estimates}, title = {\bf Example 6 (b)}]
\addplot[color = red,   mark = *, step = 1cm,very thin]coordinates{(100,0.194)(200,0.377)(300,0.643)(400,0.811)(500,0.896)};

\addplot[color = blue,   mark = diamond*, step = 1cm,very thin]coordinates{(100,0.057)(200,0.061)(300,0.056)(400,0.054)(500,0.056)};

\addplot[color = black,   mark = star, step = 1cm,very thin]coordinates{(100,0.174)(200,0.352)(300,0.544)(400,0.725)(500,0.869)};

\addplot[color = applegreen,   mark = square*, step = 1cm,very thin]coordinates{(100,0.061)(200,0.05)(300,0.043)(400,0.051)(500,0.071)};


\end{axis}
\end{tikzpicture}
\begin{tikzpicture}[scale = 0.9]
\begin{axis}[xmin = 100, xmax = 700, ymin = 0, ymax = 1, xlabel = {Sample Size}, ylabel = {Estimates}, title = {\bf Example 6 (c)}]
\addplot[color = red,   mark = *, step = 1cm,very thin]coordinates{(100,0.096)(200,0.217)(300,0.349)(400,0.494)(500,0.627)(600,0.74)(700,0.86)};

\addplot[color = blue,   mark = diamond*, step = 1cm,very thin]coordinates{(100,0.061)(200,0.058)(300,0.061)(400,0.055)(500,0.049)(600,0.06)(700,0.05)};

\addplot[color = black,   mark = star, step = 1cm,very thin]coordinates{(100,0.109)(200,0.176)(300,0.321)(400,0.467)(500,0.608)(600,0.678)(700,0.832)};

\addplot[color = applegreen,   mark = square*, step = 1cm,very thin]coordinates{(100,0.05)(200,0.06)(300,0.07)(400,0.052)(500,0.063)(600,0.055)(700,0.051)};

\end{axis}
\end{tikzpicture}
\caption{Powers of the proposed test (\textcolor{red}{$\tikzcircle{2pt}$}), OT test (\textcolor{blue}{$\blacklozenge$}), DT test (\textcolor{black}{$\star$}) and PP test (\textcolor{applegreen}{$\blacksquare$}) in Examples 6(a)-(c) for tests based on centered observations (top row) and pairwise differences of the observations (bottom row).}
    \label{fig:power-5}
\end{figure}


}

{ 
\section{Analysis of a benchmark dataset}

For further comparison among different tests, we analyze the ``MAGIC Gamma Telescope" data set available at the  \href{https://archive.ics.uci.edu/dataset/159/magic+gamma+telescope}{UCI machine learning repository}. This data set was generated by a Monte Carlo program called CORSIKA described in \cite{Heck1998}. It is used to simulate registration of high-energy gamma particles in a ground-based atmospheric Cherenkov gamma telescope. 
These observations are classified based on the patterns in the images the particles generate, called the shower images. Based on these shower images the particles are classified as ``primary gamma" and ``hadronic shower". For our analysis, we first divide the entire data set into two parts based on the class labels ``primary gamma" and ``hadronic shower" and call them ``MAGIC-1" and ``MAGIC-2", respectively. Here the observations are 10-dimensional \citep[see][for details]{Heck1998}. While MAGIC-1 contains $12332$ observations, there are $6688$ observations in MAGIC-2. 

Since the tests based on centered observations often fail to satisfy the level property, for valid comparison, here we only consider tests based on pairwise differences of the observations as discussed in Section \ref{sec:unknown-center}. When we apply these tests on the full data, all of them reject the null hypothesis of spherical symmetry in both cases. This gives us an indication that the underlying distributions are non-spherical, and different tests can be compared based on their powers. However, it is not possible to compare among different tests using a single experiment based on the full data set. Therefore, to compare the performances of different tests, we use them on random sub-samples taken from these two data sets, and in each case, we repeat the procedure 1000 times. Like before, the power of a test is computed by the proportion of times it rejects the null hypothesis. The results for different sub-sample sizes are reported in Figure \ref{fig:power-6}. In both of these data sets, the proposed test significantly outperforms all other competitors, especilly for larger sample sizes.  Other tests have slightly higher powers when the sample size is small. But, while the power of the proposed test rises sharply with increasing sample size, those of other tests either remain almost the same or increase at a slower rate. 


\begin{figure}[h]
\centering
    \begin{tikzpicture}
\begin{axis}[xmin = 10, xmax = 200, ymin = 0, ymax = 1, xlabel = {Sample Size}, ylabel = {Estimates}, title = {\bf MAGIC-1}]
\addplot[color = red,   mark = *, step = 1cm,very thin]coordinates{(10,0)(57,0.003)(105,0.207)(152,0.513)(200,0.665)};

\addplot[color = blue,   mark = diamond*, step = 1cm,very thin]coordinates{(10,0.118)(57,0.141)(105,0.131)(152,0.124)(200,0.12)};

\addplot[color = black,   mark = star, step = 1cm,very thin]coordinates{(10,0)(57,0.022)(105,0.068)(152,0.115)(200,0.181)};

\addplot[color = applegreen,   mark = square*, step = 1cm,very thin]coordinates{(10,0.054)(57,0.072)(105,0.065)(152,0.061)(200,0.059)};


\end{axis}
\end{tikzpicture}
\begin{tikzpicture}
\begin{axis}[xmin = 10, xmax = 500, ymin = 0, ymax = 1, xlabel = {Sample Size}, ylabel = {Estimates}, title = {\bf MAGIC-2}]
\addplot[color = red,   mark = *, step = 1cm,very thin]coordinates{(10,0)(132,0.009)(255,0.153)(377,0.523)(500,0.734)};

\addplot[color = blue,   mark = diamond*, step = 1cm,very thin]coordinates{(10,0.099)(132,0.102)(255,0.125)(377,0.129)(500,0.124)};

\addplot[color = black,   mark = star, step = 1cm,very thin]coordinates{(10,0)(132,0.016)(255,0.049)(377,0.095)(500,0.151)};

\addplot[color = applegreen,   mark = square*, step = 1cm,very thin]coordinates{(10,0.047)(132,0.063)(255,0.079)(377,0.078)(500,0.078)};

\end{axis}
\end{tikzpicture}
\caption{Powers of the proposed test (\textcolor{red}{$\tikzcircle{2pt}$}), OT test (\textcolor{blue}{$\blacklozenge$}), DT test (\textcolor{black}{$\star$}) and PP test (\textcolor{applegreen}{$\blacksquare$}) for the Magic Gamma Telescope data set.}
\label{fig:power-6}
\end{figure}

}

{ 

\section{Concluding remarks}

In this article, we have proposed a new measure of spherical asymmetry and a consistent estimator of this measure based on data augmentation. We have also constructed a test of spherical symmetry based on this proposed estimator and studied its large sample behaviour when the dimension may or may not grow with the sample size. Extensive simulation studies have been carried out to amply demonstrate the superiority of our test over some state-of-the-art methods. First, we have constructed our test assuming the center of symmetry under $H_0$ to be known (which is taken as ${\bf 0}$ throughout this article). If it is not known, and one estimates the center from the data, the resulting tests may not have the desired level property. We have demonstrated this and also proposed a simple solution based on the pairwise differences between the observations to overcome this problem. 
However, the resulting tests are based on data splitting, and this may lead to loss of power in some cases. Further research is needed to take care of this issue. 

Following our idea based on data augmentation, tests for other types of symmetry can also be constructed. For instance, one can construct a test for general symmetry (i.e., ${\bf X}\stackrel{D}{=}\epsilon{\bf X}$ where $\epsilon\sim$ Unif$(\{-1,1\})$), angular symmetry (i.e., ${\bf X}/{\|{\bf X}\|} \stackrel{D}{=}\epsilon{\bf X}/{\|{\bf X}\|}$ where $\epsilon\sim$ Unif$(\{-1,1\})$), coordinate wise symmetry (i.e., $(X_1,\ldots,X_d)\stackrel{D}{=}(\epsilon_1X_1,\ldots,\epsilon_dX_d)$ where $\epsilon_1,\ldots,\epsilon_d$ are i.i.d. Unif$(\{-1,1\})$ independent of the $X_i$s) or any $\mathcal{G}$-symmetry (i.e., ${\bf X} \stackrel{D}{=} G{\bf X}$ for all $G\sim$ Unif$(\mathcal{G})$). There also one needs to develop a suitable resampling algorithm for calibration by exploiting the exchangeability of observed and augmented data under $H_0$. 

Throughout this article, we have used the probability measure corresponding to ${\mathcal N}_d({\bf 0}_d,\frac{1}{d}{\bf I}_d)$ as the weighting measure $W$ to construct the
measure of asymmetry $\zeta(\Pr)$. One can introduce a scale parameter in $W$ or use other choices of $W$ (e.g., Laplace or Cauchy) as well.
Depending on the choice of $W$ the results may vary, and a suitable data-driven choice of $W$ may lead to further improvement in the performance of the proposed test.  
We leave these as possible future extensions of our work. 
}

{

In modern machine learning era, nowadays we often deal with huge sample sizes, where the traditional resampling algorithms like permutation or bootstrap become computationally demanding. To take care of this problem, several approximate algorithms  \citep[see, e.g.][]{politis1999subsampling, bickel1997noutofn, bagoflittlebootstrap, sengupta2016sdb} have been proposed as alternatives to usual bootstrap. However, these ideas cannot be directly adopted to develop a computationally efficient alternative to the resampling algorithm used for calibrating the proposed test. 
One needs to develop a new approximation algorithm for this purpose. This may be considered as a future research problem.

}

 



\bibliographystyle{chicago}
\bibliography{refs}

\section{Appendix}
\section{Expression of $\zeta({\Pr})$ for Gaussian distributions}
\label{Gaussian-special-case}
In this section, we present a closed-form expression for  $\zeta(\Pr)$ when $\Pr$ is a multivariate Gaussian distribution with mean ${\bf 0}$ and variance-covariance matrix $\sigmat$. But before that, we present some preliminary lemmas.

\setcounter{lemma}{0}
\renewcommand{\thelemma}{\Alph{section}\arabic{lemma}}  

\begin{lemma}
    If ${\bf X}_1\sim N({\bf 0},\sigmat_1)$ and ${\bf X}_2\sim N({\bf 0},\sigmat_2)$ are independent $d$-dimensional random vectors,
    $$\E\Big\{\exp\big(-\frac{1}{2d}\|{\bf X}_1-{\bf X}_2\|^2\big)\Big\} = \Big|\frac{1}{d}(\sigmat_1+\sigmat_2)+{\bf I}_d\Big|^{-1/2},$$
    where $\big| {\bf A}\big|$ denotes the determinant of a matrix ${\bf A}$, and ${\bf I}_d$ is the $d\times d$ identity matrix. 
    \label{aux-1}
\end{lemma}

\begin{proof}[of Lemma \ref{aux-1}]
    Here ${\bf N} = {\bf X}_1-{\bf X}_2$ follows $N( {\bf 0}, \sigmat_1+\sigmat_2)$. So, we have 
    $$\E\Big\{\exp\big(-\frac{1}{2d}\|{\bf N}\|^2\big)\Big\} = \int \frac{1}{(2\pi)^{d/2}\Big|\sigmat_1+\sigmat_2\Big|^{1/2}}\exp\Big(-\frac{1}{2d}\|{\bf u}\|^2-\frac{1}{2}{\bf u}^{\top} (\sigmat_1+\sigmat_2)^{-1} {\bf u}\Big) {\mathrm d}{\bf u}.$$

    Note that the exponent on the right side is the same as that of the density of a normal distribution with mean ${\bf 0}$ and variance-covariance matrix $\big(\frac{1}{d}{\bf I}_d+(\sigmat_1+\sigmat_2)^{-1}\big)^{-1}$. Therefore, we have

    $$\E\Big\{\exp\big(-\frac{1}{2d}\|{\bf N}\|^2\big)\Big\} = \frac{1}{\Big|\sigmat_1+\sigmat_2\Big|^{1/2}\Big|\frac{1}{d}{\bf I}_d+(\sigmat_1+\sigmat_2)^{-1}\Big|^{1/2}} = \Big|\frac{1}{d}(\sigmat_1+\sigmat_2)+ {\bf I}_d\Big|^{-1/2}.$$
    This completes the proof.
\end{proof}

\begin{lemma}
    If ${\bf X}$ follows a $d$-variate distribution with density $f$,  ${\bf U}\sim$ $\mathrm{Unif}(\mathcal{S}^{d-1}$), and they are independent, then $\|{\bf X}\|{\bf U}$ has the density $\int f({\bf H}^{\top}{\bf u})\pi({\bf H})$, where $\pi$ is the Haar measure on the set of all $d\times d$ orthogonal matrices.
    \label{variant-density}
\end{lemma}

\begin{proof}[of Lemma \ref{variant-density}]
    Let $\mu$ be the probability measure corresponding to distribution Unif$(\mathcal{S}^{d-1})$ and $g$ be a bounded continuous. Then, we have
    $$\E\{g(\|{\bf X}\|{\bf U})\} = \int \int g(\|{\bf x}\|{\bf u}) f({\bf x}) \mathrm d\mu({\bf u})\mathrm d{\bf x} = \int \int g({\bf H}{\bf x}) f({\bf x}) \mathrm d\pi({\bf H})\mathrm d{\bf x}.$$
    Since for any fixed ${\bf x}$ and ${\bf u}$, there exists a unique orthogonal matrix ${\bf H}$ such that ${\bf H}{\bf x} = \|{\bf x}\|{\bf u}$, the second equality follows by substitution. Now if we substitute ${\bf H}{\bf x}={\bf v}$ in the above integration we get 
    $$\E\{g(\|{\bf X}\|{\bf U})\} = \int  g({\bf v}) \big(\int f({\bf H}^\top{\bf v}) \mathrm d\pi({\bf H})\big) \Big|\big|{\bf H}^\top\big|\Big| \mathrm d{\bf v} = \int  g({\bf v}) \big(\int f({\bf H}^\top{\bf v}) \mathrm d\pi({\bf H})\big) \mathrm d{\bf v}.$$
    Since $g$ is an arbitrary bounded and continuous function, the result follows.
\end{proof}

\begin{prop}
    If ${\bf X}$ follows a $d$-variate normal distribution with mean ${\bf 0}$ and variance-covariance matrix $\sigmat$, then
    \begin{align*}
        \zeta(\Pr) = & ~\Big|\frac{2}{d}\sigmat+{\bf I}_d\Big|^{-1/2} + \int \int \Big|\frac{1}{d}({\bf H}_1\sigmat {\bf H}_1^\top+{\bf H}_2\sigmat{\bf H}_2^\top)+{\bf I}_d\Big|^{-1/2} \hspace{-0.15in}\mathrm d\pi({\bf H}_1) \mathrm d\pi({\bf H}_2)\\
        & - 2 \int \Big|\frac{1}{d}(\sigmat+{\bf H}\sigmat {\bf H}^\top)+{\bf I}_d\Big|^{-1/2} \hspace{-0.15in}\mathrm d\pi({\bf H}),
    \end{align*}
    where $\pi$ is the Haar measure on the set of all orthogonal matrices of order $d\times d$.
    \label{gaussian-zeta}
\end{prop}

\begin{proof}[of Proposition \ref{gaussian-zeta}]

Recall, that our measure can be written as (see Theorem 1)
$$\zeta(\Pr) = \E\big\{\exp\{-\frac{1}{2d}\|{ {\bf X}_1-{\bf X}_2}\|^2\}\big\} + \E\big\{\exp\{-\frac{1}{2d}\|{ {\bf X}_1^\prime-{\bf X}_2^\prime}\|^2\}\big\}-2\E\big\{\exp\{-\frac{1}{2d}\|{ {\bf X}_1-{\bf X}_2^\prime}\|^2\}\big\}. $$
Let us look at the individual terms separately. The first term on the right side is the same as the term in Lemma \ref{aux-1} with ${\bf X}_1$ and ${\bf X}_2$ being i.i.d. Gaussian random variables with mean zero and variance-covariance matrix $\sigmat$. Therefore,
$$\E\big\{\exp\{-\frac{1}{2d}\|{ {\bf X}_1-{\bf X}_2}\|^2\}\big\} = \Big|{\bf I}_d+\frac{2}{d}\sigmat\Big|^{-1/2}.$$

Now note that the second term can be written as
\begin{align*}
    & \E\big\{\exp\{-\frac{1}{2d}\|{ {\bf X}_1^\prime-{\bf X}_2^\prime}\|^2\}\big\}\\
    & = \int \exp\{-\frac{1}{2d}\|{ {\bf x}_1-{\bf x}_2}\|^2\} \frac{1}{(2\pi)^d\Big|\sigmat\Big|} \exp\Big\{-\frac{1}{2}{\bf x}_1^\top ({\bf H}_1\sigmat {\bf H}_1^\top)^{-1}{\bf  x}_1 -\frac{1}{2}{\bf x}_2^\top ({\bf H}_2\sigmat {\bf H}_2^\top)^{-1} {\bf x}_2\Big\}\\
    &\hspace{4.5in}\mathrm d{\bf x}_1\mathrm d{\bf x}_2 \mathrm d\pi({\bf H}_1)\mathrm d\pi({\bf H}_2)\\
    & = \int \frac{\Big|\sigmat\Big|^{-1}\Big|{\bf H}_1\sigmat {\bf H}_1^\top\Big|^{1/2}\Big|{\bf H}_2\sigmat{\bf H}_2^\top\Big|^{1/2}}{\Big|\frac{1}{d}\big({\bf H}_1\sigmat {\bf H}_1^\top+{\bf H}_2\sigmat {\bf H}_2^\top\big)+{\bf I}_d\Big|^{1/2}} \mathrm d\pi({\bf H}_1)\mathrm d\pi({\bf H}_2)\\
    & = \int \Big|\frac{1}{d}\big({\bf H}_1\sigmat {\bf H}_1^\top+{\bf H}_2\sigmat {\bf H}_2^\top\big)+{\bf I}_d\Big|^{-1/2} \mathrm d\pi({\bf H}_1)\mathrm d\pi({\bf H}_2).
\end{align*}

Similarly, one can also show that,
\begin{align*}
    \E\big\{\exp\{-\frac{1}{2d}\|{ {\bf X}_1-{\bf X}_2^\prime}\|^2\}\big\} = \int \Big|\frac{1}{d}\big({\bf H}\sigmat {\bf H}^\top+\sigmat\big)+{\bf I}_d\Big|^{-1/2} \mathrm d\pi({\bf H}).
\end{align*}
This completes the proof.
\end{proof}

\section{Proofs of the results stated in Section 2}

\vspace{0.1in}
\begin{proof}[\bf Proof of Proposition 2.1]
    It is easy to see that if ${\bf X} \sim {\Pr}$ is spherically distributed, then $\varphi_1=\varphi_2$ and hence $\zeta_W(\Pr)=0$. So, let us prove the only if part. Recall that $\zeta_W(\Pr)$ is given by
    $$\zeta_W(\Pr) = \int \big|\varphi_1({\bf t}) - \varphi_{2} ({\bf t})\big|^2 {\mathrm d}W({\bf t}).$$
    Since $W(\cdot)$ is a non-negative measure, the non-negativity of $\zeta_W(\Pr)$ follows from the non-negativity of the integrand. Now $\zeta_W(\Pr) = 0$ implies $\varphi_1({\bf t}) = \varphi_{2} ({\bf t})$ over the set $\{{\bf t}\mid W({\bf t})>0\}$. Since $W(\cdot)$ is equivalent to the Lebesgue measure, this implies $\varphi_1({\bf t})=\varphi_2({\bf t})$ almost everywhere with respect to the Lebesgue measure. {Since, both $\varphi_1$ and $\varphi_2$ are continuous functions, $\zeta_W(\Pr)=0$ implies
    $\varphi_1({\bf t})=\varphi_2({\bf t})$ for all ${\bf t} \in \R^d$.} Hence, ${\bf X}$ and $\|{\bf X}\|{\bf U}$ are identically distributed, i.e., ${\Pr}$ is spherically distributed.   
\end{proof}

\begin{proof}[\bf Proof of Lemma 2.2]
    Note that
    \begin{align}
        \int \exp\big\{i\langle {\bf T}, {\bf X}_1-{\bf X}_2\rangle\big\} \frac{d^{d/2}}{(2\pi)^{d/2}} e^{-{d\|{\bf T}\|^2}/{2}} \mathrm d{\bf T} = \E\Big\{\exp\big\{i\langle {\bf T}, {\bf X}_1-{\bf X}_2\rangle\big\}\Big\},
        \label{characteristic-representation}
    \end{align}
    where ${\bf T}$ follows a $d$-variate Gaussian distribution with mean zero and variance-covariance matrix $d^{-1}{\bf I}_d$. The right side of equation (\ref{characteristic-representation}) is the characteristic function of ${\bf T}$ evaluated at ${\bf X}_1-{\bf X}_2$. Hence, we have the desired result. 
\end{proof}

\begin{proof}[\bf Proof of Theorem 2.1]
    First note that
    $|\varphi_1({\bf t})-\varphi_2({\bf t})|^2 = \big(\varphi_1({\bf t})-\varphi_2({\bf t})\big)\big(\varphi_1(-{\bf t})-\varphi_2(-{\bf t})\big)$.
    So, expanding the characteristic functions in terms of expectations, we get 
    \begin{align}
        |\varphi_1({\bf t})-\varphi_2({\bf t})|^2 = & \E\Big\{\exp\big\{i\langle {\bf t}, {\bf X}_1-{\bf X}_2 \rangle\big\}\Big\}+ \E\Big\{\exp\big\{i\langle {\bf t}, {\bf X}^{'}_1-{\bf X}^{'}_2 \rangle\big\}\Big\} \nonumber \\ 
        &-\E\Big\{\exp\big\{i\langle {\bf t}, {\bf X}_1-{\bf X}^{'}_2 \rangle\big\}\Big\}-\E\Big\{\exp\big\{i\langle {\bf t}, {\bf X}_1^{'}-{\bf X}_2 \rangle\big\}\Big\}.
        \label{eq:try}
    \end{align}
    From Lemma 2, for any ${\bf V}_1$ and ${\bf V}_2$, we have
    $$ \int \E\Big\{\exp\big\{i\langle {\bf t}, {\bf V}_1-{\bf V}_2 \rangle\big\}\Big\} \frac{d^{d/2}}{(2\pi)^{d/2}} e^{-{d\|{\bf t}\|^2}/{2}} {\mathrm d}{\bf t}=
    \E \Big\{\exp\big\{-\frac{1}{2d}\|{\bf V}_1-{\bf V}_2\|^2\big\}\Big\}.$$
    Applying this to all four terms in (\ref{eq:try}) (note that the last two terms are equal), we get the result.
\end{proof}

\begin{proof}[\bf Proof of Theorem 2.2]
    As introduced in Section 2.2 we can write our estimator as
    $$\hat\zeta_n = \frac{2}{n(n-1)}\sum_{1\leq i< j\leq n} g\big(({\bf X}_i,{\bf X}_i^\prime),({\bf X}_j,{\bf X}_j^\prime)\big),$$
    where
    \begin{align*}
        g\big({ ({\bf x}_1,{\bf x}_2),({\bf y}_1,{\bf y}_2)}\big) = &  \exp\bigg\{-\frac{\|{\bf x}_1-{\bf y}_1\|^2}{2d}\bigg\}+\exp\bigg\{-\frac{\|{\bf x}_2-{\bf y}_2\|^2}{2d}\bigg\} \\
        & \hspace{1in}-\exp\bigg\{-\frac{\|{\bf x}_1-{\bf y}_2\|^2}{2d}\bigg\}-\exp\bigg\{-\frac{\|{\bf x}_2-{\bf y}_1\|^2}{2d}\bigg\}.
    \end{align*}
    Now let $\hat\zeta_n^{(i)}$ denote our estimator when the $i^{th}$ observation $({\bf X}_i,{\bf X}_i^\prime)$ is replaced by an independent copy $({\bf Y}_i,{\bf Y}_i^\prime)$ of the same. Note that
    \begin{align*}
         |\hat\zeta_n-\hat\zeta_n^{(i)}|
         & \leq \frac{2}{n(n-1)}\Bigg\{\sum_{j=1}^{i-1} \Big|g\big(({\bf X}_j,{\bf X}_j^\prime),({\bf X}_i,{\bf X}_i^\prime)\big)-g\big(({\bf X}_j,{\bf X}_j^\prime),({\bf Y}_i,{\bf Y}_i^\prime)\big)\Big|\\ & ~~~~~~~~~~~~~~~~~~~+ \sum_{j=i+1}^n \Big|g\big(({\bf X}_i,{\bf X}_i^\prime),({\bf X}_j,{\bf X}_j^\prime)\big)-g\big(({\bf Y}_i,{\bf Y}_i^\prime),({\bf X}_j,{\bf X}_j^\prime)\big)\Big|\Bigg\} \\
    \end{align*}
    Since $|g(\cdot,\cdot)|\leq 2$, this implies 
        $|\hat\zeta_n-\hat\zeta_n^{(i)}| \leq \frac{8(n-1)}{n(n-1)} \leq \frac{8}{n}$. So, applying bounded difference inequality \citep[see page 37 in][]{wainwright2019} 
    \begin{align*}
        \P\{\big|\hat\zeta_n - \zeta(\Pr)\big|> \epsilon\} \leq \exp\Big\{-\frac{2\epsilon^2}{\sum_{i=1}^n \frac{64}{n^2}}\Big\} = \exp\Big\{-\frac{n\epsilon^2}{32}\Big\}.
    \end{align*}
    This completes the proof.
\end{proof}

\begin{proof}[\bf Proof of Theorem 2.3]
    Note that our estimator $\hat\zeta_n$ is a U-statistic with the kernel
       \begin{align*}
        g\big({ ({\bf x}_1,{\bf x}_2),({\bf y}_1,{\bf y}_2)}\big) = &  \exp\bigg\{-\frac{\|{\bf x}_1-{\bf y}_1\|^2}{2d}\bigg\}+\exp\bigg\{-\frac{\|{\bf x}_2-{\bf y}_2\|^2}{2d}\bigg\} \\
        & \hspace{1in}-\exp\bigg\{-\frac{\|{\bf x}_1-{\bf y}_2\|^2}{2d}\bigg\}-\exp\bigg\{-\frac{\|{\bf x}_2-{\bf y}_1\|^2}{2d}\bigg\}.
    \end{align*}
    of degree 2. The first order Hoeffding projection of $g(.,.)$ is 
    \begin{align}
      g_1\big(({\bf x}_1,{\bf x}_2)\big) = \E\Big\{K({\bf x}_1,{\bf X}_1)\Big\}+\E\Big\{K({\bf x}_2,{\bf X}_1^\prime)\Big\} -\E\Big\{K({\bf x}_1,{\bf X}_1^\prime)\Big\}-\E\Big\{K({\bf x}_2,{\bf X}_1)\Big\}
      \label{first-projection}
    \end{align}
    where $K({\bf x},{\bf y})= \exp\{-\|{\bf x}-{\bf y}\|^2/(2d)\}$, ${\bf X}_1 \sim \Pr$ and ${\bf X}_1^\prime = \|{\bf X}_1\|{\bf U}_1$ for ${\bf U}_1 \sim $  Unif$(\mathcal{S}^{d-1})$ independent of ${\bf X}_1$. Under $H_0$, ${\bf X}_1$ and ${\bf X}_1^\prime$ are identically distributed and hence $g_1\big(({\bf x}_1,{\bf x}_2)\big) = 0$. Therefore, using Theorem 1 from \cite{lee2019u} p.79, we get that $n\hat\zeta_n$ converges in distribution to $\sum_{i = 1}^\infty \lambda_i (Z_i^2-1)$ where $\{Z_i\}$ are independent standard normal random variables and $\{\lambda_i\}$ are the eigenvalues of the integral equation 
    $$\int g\big(({\bf x}_1,{\bf x}_2),({\bf y}_1,{\bf y}_2)\big) f\big(({\bf y}_1,{\bf y}_2)\big) \mathrm dF\big(({\bf y}_1,{\bf y}_2)\big) = \lambda f\big(({\bf x}_1,{\bf x}_2)\big),$$
    where $F$ is the joint distribution of $({\bf X}_1,{\bf X}_1^\prime)$. 
    
\end{proof}

\begin{proof}[\bf Proof of Theorem 2.4]
    Under $H_1$, the function $g_1\big(({\bf x}_1,{\bf x}_2)\big)$ defined in equation \eqref{first-projection} is non-degenerate. Therefore, using Theorem 1 from \cite{lee2019u} p.76, we get the asymptotic normality of $\sqrt{n}(\hat\zeta_n-\zeta(\Pr))$ with the mean zero and the variance $4\sigma_1^2$, where $\sigma_1^2 = \var\big(g_1\big(({\bf X}_1,{\bf X}_1^\prime)\big)\big)$.
\end{proof}

\begin{proof}[\bf Proof of Lemma 2.3]
    Define ${\bf X}_1^\prime,{\bf X}_2^\prime,\ldots, {\bf X}_n^\prime$ as in {Theorem 2.3}. Under $H_0$, $({\bf X}_i,{\bf X}_i^\prime)$ and $({\bf X}_i^\prime, {\bf X}_i)$ are identically distributed for each $i=1,2,\ldots, n$. Therefore, {for any fixed $n$,} the joint distribution of $(\hat\zeta_n,c_{1-\alpha})$ is identical to the joint distribution of $(\hat\zeta_n(\pi),c_{1-\alpha})$, where $\pi$ is an element of $\{0,1\}^n$. Let $\mathcal{D}^\prime = \{({\bf X}_i,{\bf X}_i^\prime)\mid i=1,2,\ldots, n\}$ denote the augmented data. Then,
    $$\P\{\hat\zeta_n>c_{1-\alpha}\} = \P\{\hat\zeta_n(\pi)>c_{1-\alpha}\} = \E\big\{\P\{\hat\zeta_n(\pi)>c_{1-\alpha}\mid \mathcal{D}^\prime\}\big\}\leq \alpha.$$
    The last inequality follows from the definition of $c_{1-\alpha}$ (given in the { Resampling Algorithm A-C}).
\end{proof}

\begin{proof}[\bf Proof of Lemma 2.4]
    Here ${\bf X}_1,{\bf X}_2,\ldots, {\bf X}_n$ are independent copies of ${\bf X}\sim \Pr$. Define ${\bf X}_1^\prime,{\bf X}_2^\prime,\ldots, {\bf X}_n^\prime$ as in {Theorem 2.3}, and $\mathcal{D}^\prime$ as in the proof of {Lemma 2.3}. Let $\pi=(\pi(1),\pi(2),\ldots,\pi(n))$ be uniformly distributed on the set $\{0,1\}^n$. For $i=1,2,\ldots,n$, define ${\bf Y}_i = \pi(i) {\bf X}_i + (1-\pi(i)) {\bf X}_i^\prime$ and ${\bf Y}_i^\prime = (1-\pi(i)) {\bf X}_i + \pi(i) {\bf X}_i^\prime$. Since $g\big(({\bf Y}_i,{\bf Y}_i^\prime),({\bf Y}_i,{\bf Y}_i^\prime)\big)\ge 0$ for all $i=1,2,\ldots,n$, we have
    \begin{align*}
        n(n-1)\hat\zeta_n(\pi)  = \sum_{1\leq i\not=j\leq n} g\big(({\bf Y}_i,{\bf Y}_i^\prime),({\bf Y}_j,{\bf Y}_j^\prime)\big) \leq \sum_{1\leq i,j\leq n} g\big(({\bf Y}_i,{\bf Y}_i^\prime),({\bf Y}_j,{\bf Y}_j^\prime)\big) = n^2 \zeta(F_n),
    \end{align*}
    where $F_n$ denotes the empirical probability distribution of $({\bf Y}_1,{\bf Y}_1^\prime),({\bf Y}_2,{\bf Y}_2^\prime),\ldots, ({\bf Y}_n,{\bf Y}_n^\prime)$. Since, $\zeta(F_n)$ is a non-negative random variable, using Markov inequality, for any $\epsilon>0$, we have
    $$\P\big\{\hat\zeta_n(\pi)>\epsilon\mid \mathcal{D}^\prime\big\} \leq \P\big\{n^2\zeta(F_n)>n(n-1)\epsilon\mid \mathcal{D}^\prime\big\} \leq \frac{n^2}{n(n-1)\epsilon} \E\big\{\zeta(F_n)\mid \mathcal{D}^\prime\big\}.$$
    Now, taking $\epsilon=  \frac{n}{(n-1)\alpha}\E\big\{\zeta(F_n)\mid \mathcal{D}^\prime\big\}$, we get
    $\P\big\{\hat\zeta_n(\pi)>\epsilon\mid \mathcal{D}^\prime\big\}\leq \alpha.$ Therefore, from the definition of $c_{1-\alpha}$ (see {Resampling Algorithm A-C}), we have
    $$c_{1-\alpha} \leq \frac{n}{(n-1)\alpha}\E\big\{\zeta(F_n)\mid \mathcal{D}^\prime\big\}.$$
    Also, note that
    $$\E\{\zeta(F_n)\mid \mathcal{D}^\prime\} = \frac{1}{n^2}\left[\sum_{1\leq i\not=j\leq n} \E\Big\{g\big(({\bf Y}_i,{\bf Y}_i^\prime),({\bf Y}_j,{\bf Y}_j^\prime)\big)\mid \mathcal{D}^\prime\Big\} + \sum_{1\leq i\leq n} \E\Big\{g\big(({\bf Y}_i,{\bf Y}_i^\prime),({\bf Y}_i,{\bf Y}_i^\prime)\big)\mid \mathcal{D}^\prime\Big\}\right].$$
    
    \begin{align*}
      & \text{Now, for any $1\leq i\not=j\leq n$,} ~~  \E\Big\{\exp\{-\frac{\|{\bf Y}_i-{\bf Y}_j\|^2}{2d}\}\mid \mathcal{D}^\prime\Big\}\\
        & \hspace{1in}= \frac{1}{4}\sum_{\pi\in\{0,1\}^2} \exp\Big\{-\frac{\|\pi(1){\bf X}_i+(1-\pi(1)){\bf X}_i^\prime-\pi(2){\bf X}_j-(1-\pi(2)){\bf X}_j^\prime\|^2}{2d}\Big\} =\delta_n(i,j) \text{ (say)}.
    \end{align*}
    Similarly, one can also show that 
    \begin{align*}
        \E\Big\{\exp\{-\frac{\|{\bf Y}_i^\prime-{\bf Y}_j^\prime\|^2}{2d}\}\mid \mathcal{D}^\prime\Big\} = & ~\E\Big\{\exp\{-\frac{\|{\bf Y}_i-{\bf Y}_j^\prime\|^2}{2d}\}\mid \mathcal{D}^\prime\Big\}\\
        =&~\E\Big\{\exp\{-\frac{\|{\bf Y}_i^\prime-{\bf Y}_j\|^2}{2d}\}\mid \mathcal{D}^\prime\Big\} = \delta_n(i,j).
    \end{align*}
    Hence,
    $\E\Big\{g\big(({\bf Y}_i,{\bf Y}_i^\prime),({\bf Y}_j,{\bf Y}_j^\prime)\big)\mid \mathcal{D}^\prime\Big\} = \delta_n(i,j)+\delta_n(i,j)-\delta_n(i,j)-\delta_n(i,j) = 0.$ 
    Similarly, one can also show that for any $1\leq i\leq n$,
    \begin{align*}
        & \E\Big\{g\big(({\bf Y}_i,{\bf Y}_i^\prime),({\bf Y}_i,{\bf Y}_i^\prime)\big)\mid \mathcal{D}^\prime\Big\}\\
        & = 2\Big(1-\frac{1}{2}\sum_{\pi\in\{0,1\}} \exp\Big\{-\frac{\|\pi(1){\bf X}_i+(1-\pi(1)){\bf X}_i^\prime-(1-\pi(1)){\bf X}_i-\pi(1){\bf X}_i^\prime\|^2}{2d}\Big\}\Big)\leq 2
    \end{align*}
    Combining these, we get $c_{1-\alpha}\leq 2(\alpha (n-1))^{-1}$. This completes the proof.
\end{proof}

\begin{proof}[\bf Proof of Theorem 2.5]
    Note that $\hat\zeta_n$ is a consistent estimator of $\zeta(\Pr)$ (follows from {Theorem 2.2}), where $\zeta(\Pr)=0$ under $H_0$ and positive under $H_1$ (follows from {Proposition 2.1}). So,  by Lemma 4 under $H_1$, the power of the test ${\P}(\hat\zeta_n>c_{1-\alpha})$ converges to one as $n$ diverges to infinity. 
\end{proof}

\begin{proof}[\bf Proof of Theorem 2.6]
    
Let us define the distribution functions
$F(t) = \frac{1}{2^n}\Big\{\sum_{\pi\in \{0,1\}^n}I[\hat\zeta_n(\pi)\leq t]\Big\}$ and $F_B(t) = \frac{1}{B}\left\{\sum_{i=1}^BI[\hat\zeta_n(\pi_i)\leq t]\right\}$
conditioned on the augmented data $\mathcal{D}^\prime$. Then,
\begin{equation*}
    \begin{split}
        \big|p_{n}-p_{n,B}\big| & = \bigg|\frac{1}{2^n}\bigg\{\sum_{\pi\in \{0,1\}^n}I[\hat\zeta_n(\pi)\geq \hat\zeta_{n}]\bigg\}-\frac{1}{B+1}\bigg\{\sum_{i=1}^BI[\hat\zeta_n(\pi_i)\geq \hat\zeta_{n}]+1\bigg\}\bigg|\\
        & = \bigg|\frac{1}{2^n}\bigg\{\sum_{\pi\in \{0,1\}^n}I[\hat\zeta_n(\pi)< \hat\zeta_{n}]\bigg\}-\frac{1}{B+1}\bigg\{\sum_{i=1}^BI[\hat\zeta_n(\pi_i)< \hat\zeta_{n}]\bigg\}\bigg|\\
        &  = \big|F(\hat\zeta_{n}-)-\frac{B}{B+1}F_B(\hat\zeta_{n}-)\Big|\\
        & \leq \Big|F(\hat\zeta_{n}-)-F_B(\hat\zeta_{n}-)\Big| + \Big|\frac{F_B(\hat\zeta_{n}-)}{B+1}\Big| \leq \sup_{t\in\R}\big|F(t)-F_B(t)\big|+\frac{1}{B+1}
    \end{split}
\end{equation*}

Conditioned on $\mathcal{D}^\prime$, the Dvoretzky-Keifer-Wolfwitz inequality (\cite{massart1990tight}) gives us $\P\{\sup_{t\in\R}|F(t)-F_B(t)|>\epsilon\}\leq 2e^{-2B\epsilon^2}.$ 
\end{proof}

\section{Proofs of the results stated in Section 3}

\begin{proof}[\bf Proof of Lemma 3.1]
    Recall that $\zeta(\Pr)$ can be expressed as 
    $$\zeta(\Pr) = \E\{K({\bf X}_1,{\bf X}_2)\}+\E\{K({\bf X}_1^\prime, {\bf X}_2^\prime)\}-2\E\{K({\bf X}_1,{\bf X}_2^\prime)\},$$
    where ${\bf X}_1,{\bf X}_2$ are independent copies of ${\bf X}{\sim} \Pr$, ${\bf X}_1^\prime = \|{\bf X}_1\|{\bf U}_1, {\bf X}_2^\prime = \|{\bf X}_2\|{\bf U}_2$
    with
    ${\bf U}_1,{\bf U}_2 \stackrel{iid}{\sim} \mathrm{Unif} (\mathcal{S}^{d-1})$ independent of ${\bf X}_1,{\bf X}_2$,  and $K({\bf x},{\bf y}) = \exp\big\{-\frac{1}{2d}\|{\bf x}-{\bf y}\|^2\big\}$. Now if {$F_\delta = (1-\delta) F+ \delta G$ }($0<\delta<1$), then
    \begin{align*}
        & \E\{K({\bf X}_1,{\bf X}_2)\} = \int K({\bf x}_1,{\bf x}_2) {\mathrm d}F_\delta({\bf x}_1) {\mathrm d}F_\delta({\bf x}_2)\\
        & = (1-\delta)^2 \int K({\bf x}_1,{\bf x}_2) {\mathrm d}F({\bf x}_1) {\mathrm d}F({\bf x}_2) + 2\delta (1-\delta) \int K({\bf x}_1,{\bf x}_2) {\mathrm d}G({\bf x}_1) {\mathrm d}F({\bf x}_2)\\
        & \hspace{2in}+\delta^2 \int K({\bf x}_1,{\bf x}_2) {\mathrm d}G({\bf x}_1) {\mathrm d}G({\bf x}_2).
    \end{align*}
    If $\mu$ denotes the uniform distribution on $(\mathcal{S}^{d-1})$, then
    \begin{align*}
        \E\{K({\bf X}_1^\prime,{\bf X}_2^\prime)\} = &\int K(\|{\bf x}_1\|{\bf u}_1,\|{\bf x}_2\|{\bf u}_2) \mathrm dF_\delta({\bf x}_1) \mathrm dF_\delta({\bf x}_2) {\mathrm{d}}\mu({\bf u}_1) {\mathrm{d}}\mu({\bf u}_2)\\
         = & ~(1-\delta)^2 \int K(\|{\bf x}_1\|{\bf u}_1,\|{\bf x}_2\|{\bf u}_2) \mathrm dF({\bf x}_1)\mathrm dF({\bf x}_2) \mathrm d\mu({\bf u}_1)\mathrm d\mu({\bf u}_2) \\ & + \delta^2 \int K(\|{\bf x}_1\|{\bf u}_1,\|{\bf x}_2\|{\bf u}_2)\mathrm dG({\bf x}_1) \mathrm dG({\bf x}_2)\mathrm d\mu({\bf u}_1)\mathrm d\mu({\bf u}_2) \\
        &+ 2\delta (1-\delta) \int K(\|{\bf x}_1\|{\bf u}_1,\|{\bf x}_2\|{\bf u}_2)\mathrm dG({\bf x}_1)\mathrm dF({\bf x}_2)\mathrm d\mu({\bf u}_1)\mathrm d\mu({\bf u}_2)
    \end{align*}
    \begin{align*}
        \mbox{and }& \E\{K({\bf X}_1,{\bf X}_2^\prime)\} = \int K({\bf x}_1,\|{\bf x}_2\|{\bf u}_2) \mathrm dF_\delta({\bf x}_1) \mathrm dF_\delta({\bf x}_2)\mathrm d\mu({\bf u}_2)\\
        & = (1-\delta)^2 \int K({\bf x}_1,\|{\bf x}_2\|{\bf u}_2)\mathrm dF({\bf x}_1)\mathrm dF({\bf x}_2)\mathrm d\mu({\bf u}_2) +  \delta (1-\delta) \int K({\bf x}_1,\|{\bf x}_2\|{\bf u}_2)\mathrm dG({\bf x}_1)\mathrm dF({\bf x}_2)\mathrm d\mu({\bf u}_2) \\
        & ~~~~~~~~+ \delta (1-\delta) \int K({\bf x}_1,\|{\bf x}_2\|{\bf u}_2) \mathrm dF({\bf x}_1)\mathrm dG({\bf x}_2)\mathrm d\mu({\bf u}_2) +\delta^2 \int K({\bf x}_1,\|{\bf x}_2\|{\bf u}_2) \mathrm dG({\bf x}_1) \mathrm dG({\bf x}_2)\mathrm d\mu({\bf u}_2).
    \end{align*}
    Therefore, if {$F_\delta = (1-\delta) F + \delta G$}, then,
    $$\zeta({F_\delta}) = (1-\delta)^2 \zeta(F)+\delta^2 \zeta(G) + 2\delta(1-\delta) \zeta^\prime (G,F)$$
    where 
    $$\zeta^\prime(G,F) = \E\{K({\bf Y}_1,{\bf Y}_2)\}+\E\{K({\bf Y}_1^\prime, {\bf Y}_2^\prime)\}-\E\{K({\bf Y}_1,{\bf Y}_2^\prime)\}-\E\{K({\bf Y}_1^\prime,{\bf Y}_2)\}$$
    where ${\bf Y}_1\sim G$ and ${\bf Y}_2\sim F$ are independent, ${\bf Y}_1^\prime = \|{\bf Y}_1\| {\bf U}_1$ and ${\bf Y}_2^\prime = \|{\bf Y}_2\| {\bf U}_2$ where ${\bf U}_1,{\bf U}_2 \stackrel{iid}{\sim}$Unif$(\mathcal{S}^{d-1})$ independent of ${\bf Y}_1$ and ${\bf Y}_2$. This completes the proof.
\end{proof}

\begin{lemma}
    If ${\bf X}_1,{\bf X}_2,\ldots, {\bf X}_n$ are independent copies of ${\bf X}\sim\Pr$, then
    $$\var(\hat\zeta_n) \leq {n\choose 2}^{-1}\left[4(n-2) \zeta(\Pr) + 4\right],$$
    \label{var-inequality}
\end{lemma}

\begin{proof}[\bf  Proof of Lemma \ref{var-inequality}]
    Recall that $\hat\zeta_n$ is a U-statistics with the kernel
    $$g\big(({\bf x}_1,{\bf x}_1^\prime),({\bf x}_2,{\bf x}_2^\prime)\big) = K({\bf x}_1,{\bf x}_2) + K({\bf x}_1^\prime,{\bf x}_2^\prime) - K({\bf x}_1,{\bf x}_2^\prime) - K({\bf x}_1^\prime,{\bf x}_2),$$
    where $K({\bf x},{\bf y}) = \exp\{-\frac{1}{2d}\|{\bf x}-{\bf y}\|^2\}$. Then, by the theory of U-statistics \citep[see p.12 in ][]{lee2019u}, we have
    $$\var(\hat\zeta_n) = {n\choose 2}^{-1}\left[{2\choose 1}{n-2\choose 1} \var\Big(g_1\big({\bf X}_1,{\bf X}_1^\prime\big)\Big) + {2\choose 2}{n-2\choose 0} \var\Big(g\big(({\bf X}_1,{\bf X}_1^\prime),({\bf X}_2,{\bf X}_2^\prime)\big)\Big)\right]$$
    where,
    \begin{equation*}
    \begin{split}
        & g_1({\bf x}_1,{\bf x}_1^\prime)\\
        & = \E\{\exp(-\frac{1}{2d}\|{\bf x}_1-{\bf X}_2\|^2)\}-\E\{\exp(-\frac{1}{2d}\|{\bf x}_1-{\bf X}_2^\prime\|^2)\}\\
        &~~~~~~~~+\E\{\exp(-\frac{1}{2d}\|{\bf x}_1^\prime-{\bf X}_2^\prime\|^2)\}-\E\{\exp(-\frac{1}{2d}\|{\bf x}_1^\prime-{\bf X}_2\|^2)\}.
    \end{split}
\end{equation*}    
Since, $|g(\cdot,\cdot)|\leq 2$, $\var\Big(g\big(({\bf X}_1,{\bf X}_1^\prime),({\bf X}_2,{\bf X}_2^\prime)\big)\Big)$ is bounded by $4$. Whereas to bound the first term note that
    \begin{align*}
        g_1({\bf x}_1,{\bf x}_1^\prime) & = \E\{\int \exp(i\langle {\bf t}, {\bf x}_1-{\bf X}_2 \rangle) \phi_0({\bf t}) \mathrm d{\bf t}\} - \E\{\int \exp(i\langle {\bf t}, {\bf x}_1-{\bf X}_2^\prime \rangle) \phi_0({\bf t}) \mathrm d{\bf t}\}\\
        &\hspace{1in}-\E\{\int \exp(i\langle {\bf t}, {\bf x}_1^\prime-{\bf X}_2 \rangle) \phi_0({\bf t}) \mathrm d{\bf t}\}+\E\{\int \exp(i\langle {\bf t}, {\bf x}_1^\prime-{\bf X}_2^\prime \rangle) \phi_0({\bf t}) \mathrm d{\bf t}\}\tag*{\text{(where $\phi_0({\bf t})$ denotes the density of $N({\bf 0}, d^{-1}{\bf I}_d)$)}}\\
        & = \int \big(\exp(i\langle {\bf t}, {\bf x}_1 \rangle) - \exp(i\langle {\bf t}, {\bf x}_1^\prime \rangle)\big)\big(\E\{\exp(-i\langle {\bf t}, {\bf X}_2 \rangle)\} - \E\{\exp(-i\langle {\bf t}, {\bf X}_2^\prime \rangle)\}\big) \phi_0({\bf t}) \mathrm d{\bf t}\\
        & = \int (\exp(i\langle {\bf t}, x_1 \rangle) - \exp(i\langle {\bf t}, {\bf x}_1^\prime \rangle))(\varphi_{1}(-{\bf t}) - \varphi_{2}(-{\bf t})) \phi_0({\bf t}) \mathrm d{\bf t}\\
        & = \int (\exp(i\langle {\bf t}, {\bf x}_1 \rangle) - \exp(i\langle {\bf t}, {\bf x}_1^\prime \rangle))\overline{(\varphi_{1}({\bf t}) - \varphi_{2}({\bf t}))} \phi_0({\bf t}) {\mathrm d}{\bf t}.
     \end{align*}
Then using Cauchy-Schwartz inequality we have,
$$g_1^2({\bf x}_1,{\bf x}_1^\prime) \leq \int |\exp(i\langle {\bf t}, {\bf x}_1 \rangle) - \exp(i\langle {\bf t}, {\bf x}_1^\prime \rangle)|^2 \phi_0({\bf t}) {\mathrm d}{\bf t} \times \int |\varphi_{1}({\bf t}) - \varphi_{2}({\bf t})|^2 \phi_0({\bf t}) {\mathrm d}{\bf t} \leq 2 \zeta(\Pr).$$

This gives us,
$$\var(\hat\zeta_n) \leq {n\choose 2}^{-1}\left[4(n-2) \zeta(\Pr) + 4\right].$$
This completes the proof.    
\end{proof}

\begin{proof}[\bf Proof of Theorem 3.1]
Here ${\bf X}_1,{\bf X}_2,\ldots, {\bf X}_n$ be independent copies of ${\bf X}\sim F_\delta= (1-\delta)F + \delta G$, where $\zeta(G)=0$ and $\zeta(F)=\gamma_0>0$. So, we have $\zeta(F_\delta) = (1-\delta)^2 \gamma_0$ (follows from Lemma 5) and
\begin{align*}
    \P\{\hat\zeta_n>c_{1-\alpha}\} = 1 - \P\{\hat\zeta_n\leq c_{1-\alpha}\} \geq 1- \P\{\hat\zeta_n\leq 2((n-1)\alpha)^{-1}\}.
\end{align*}
Here, the last inequality follows from Lemma 4. Now, choose a large $n$ so that $2((n-1)\alpha)^{-1}<\zeta(F_\delta)$. Then we have
\begin{align*}
    \P\{\hat\zeta_n\leq 2 ((n-1)\alpha)^{-1}\} & = \P\{\hat\zeta_n-\zeta(F_\delta)\leq 2 ((n-1)\alpha)^{-1} -\zeta(F_\delta)\}\\
    & \leq \frac{\var(\hat\zeta_n)}{\big(\zeta(F_\delta)-2((n-1)\alpha)^{-1}\big)^2}\tag*{\text{(by Chebyshev's inequality)}}\\
    & \leq \frac{{n\choose 2}^{-1}\left[4(n-2) \zeta(F_\delta) + 4\right]}{\big(\zeta(F_\delta)-2((n-1)\alpha)^{-1}\big)^2}\tag*{\text{(by Lemma \ref{var-inequality})}}\\
    & = \frac{{n\choose 2}^{-1}\left[4(n-2) (1-\delta)^2\gamma_0 + 4\right]}{\big((1-\delta)^2\gamma_0-2((n-1)\alpha)^{-1}\big)^2}. 
\end{align*}

Note that the upper bound in the above inequality does not depend on $G$ and hence,
$$\inf_{G: \zeta(G)=0}\P\{\hat\zeta_n>c_{1-\alpha}\} \geq 1-\frac{{n\choose 2}^{-1}\left[4(n-2) (1-\delta)^2\gamma_0 + 4\right]}{\big((1-\delta)^2\gamma_0-2((n-1)\alpha)^{-1}\big)^2}.$$
Now taking limit as $n$ goes to infinity, $\inf\limits_{G:\zeta(G) = 0} \P\{\hat\zeta_n>c_{1-\alpha}\}$ goes to one. This completes the proof.
\end{proof}

\begin{proof}[\bf Proof of Theorem 3.2]
    We prove this theorem using a standard application of the Neyman-Pearson lemma. Let $Q_1$ and $Q_2$ be the joint distribution of the sample ${\bf X}_1,{\bf X}_2,\ldots, {\bf X}_n$ under the null and alternative hypotheses respectively. Then we can lower bound the minimax risk $R_{n}(\epsilon)$ as follows
    $$R_n(\epsilon)\geq 1-\alpha - d_{TV}(Q_1,Q_2)\geq 1-\alpha-\sqrt{\frac{1}{2}KL(Q_1,Q_2)}.$$
    The first inequality follows using the fact that $\P_{Q_1}\{\phi=1\}\leq \alpha$ and
    $$\P_{Q_1}\{\phi=1\} + \P_{Q_2}\{\phi=0\} = 1-(\P_{Q_1}\{\phi=0\}-\P_{Q_2}\{\phi=0\}) \geq 1-d_{TV}(Q_1,Q_2)$$
    where $d_{TV}$ denotes the total variation distance between $Q_1$ and $Q_2$. The second inequality follows from Pinsker's inequality \citep[see.][]{tsybakov2009introduction}. Suppose $G$ is a spherically symmetric distribution with density $g$ and $F$ is a distribution with density $f$ and $\zeta(F) = \gamma_0$. Assume that $\int (f({\bf u})/g({\bf u})-1)^2 g({\bf u})\mathrm d{\bf u} = \gamma_1<\infty$. Then setting $Q_1 = \prod_{i=1}^n \Big(\big(1-\frac{\delta}{\sqrt{n}}\big)G+\frac{\delta}{\sqrt{n}} F\Big)$ for some $\delta>0$ and $Q_2 = G^n$, we have
    \begin{align*}
        KL(Q_1,Q_2) & = \int \log\prod_{i=1}^n\left\{1+\frac{\delta}{\sqrt{n}}\big(\frac{f({\bf u}_i)}{g({\bf u}_i)}-1\big)\right\} \prod_{i=1}^n d\Big(\big(1-\frac{\delta}{\sqrt{n}}\big)G+\frac{\delta}{\sqrt{n}} F\Big)({\bf u}_i)\\
        & = n \int \log\left\{1+\frac{\delta}{\sqrt{n}}\big(\frac{f({\bf u}_1)}{g({\bf u}_1)}-1\big)\right\} d\Big(\big(1-\frac{\delta}{\sqrt{n}}\big)G+\frac{\delta}{\sqrt{n}} F\Big)({\bf u}_1)\\
        & = n \big(1-\frac{\delta}{\sqrt{n}}\big) \int \log\left\{1+\frac{\delta}{\sqrt{n}}\big(\frac{f({\bf u}_1)}{g({\bf u}_1)}-1\big)\right\} g({\bf u}_1) \mathrm d {\bf u}_1\\
        &~~~~~~~~~~~+n\frac{\delta}{\sqrt{n}} \int \log\left\{1+\frac{\delta}{\sqrt{n}}\big(\frac{f({\bf u}_1)}{g({\bf u}_1)}-1\big)\right\} f({\bf u}_1) \mathrm d{\bf u}_1
    \end{align*}
    Using the inequality $\log(1+y)\leq y$ we get,
    \begin{align*}
        KL(Q_1,Q_2) & \leq n\left[ \frac{\delta}{\sqrt{n}} (1-\frac{\delta}{\sqrt{n}}) \int \big(\frac{f({\bf u})}{g({\bf u})}-1\big) g({\bf u}) d{\bf u} + \frac{\delta^2}{n} \int \big(\frac{f({\bf u})}{g({\bf u})}-1\big) f({\bf u}) \mathrm d{\bf u} \right]\\
        & = \delta^2 \left[\int \frac{f^2({\bf u})}{g({\bf u})} - 1\right] \mathrm d {\bf u}= \delta^2 \int \left(\frac{f({\bf u})}{g({\bf u})}-1\right)^2 g({\bf u}) {\mathrm d}{\bf u} = \delta^2\gamma_1.
    \end{align*}
    Also by Lemma 5 we have,
    $$\zeta\left(\Big(\big(1-\frac{\delta}{\sqrt{n}}\big)G+\frac{\delta}{\sqrt{n}} F\Big)\right) = \frac{\delta^2}{n} \zeta(F) = \frac{\delta^2}{n} \gamma_0. $$
    Now for any $0<\beta<1-\alpha$ if we choose $\delta = \sqrt{2/\gamma_1}(1-\alpha-\beta)$ we get, 
    $$\zeta\left(\Big(\big(1-\frac{\delta}{\sqrt{n}}\big)G+\frac{\delta}{\sqrt{n}} F\Big)\right) = \frac{1}{n} \left(\frac{2\gamma_0(1-\alpha-\beta)^2}{\gamma_1}\right). $$
    Now define, $c(\alpha,\beta) = \big(2\gamma_0(1-\alpha-\beta)^2\big)/\gamma_1$. Then, $\big((1-\frac{\delta}{\sqrt{n}})G+\frac{\delta}{\sqrt{n}} F\big)\in \mathcal{F}(cn^{-1}) = \{F\mid \zeta(F)>cn^{-1}\}$ for all $0<c<c(\alpha,\beta)$. For this choice of alternative, we also have $R_n(cn^{-1})\geq \beta$ for all $0<c<c(\alpha,\beta)$. Since $\beta$ and $c(\alpha,\beta)$ does not depend on $n$, this trivially satisfies the condition $\liminf_{n\to\infty} R_n(cn^{-1}) \geq \beta$ for all $0<c<c(\alpha,\beta)$.
    
\end{proof}

\begin{proof}[\bf Proof of Theorem 3.3]
    Here we want to show that for every $0<\alpha<1$ and $0<\beta<1-\alpha$ there exists a constant $C(\alpha,\beta)>0$ such that
    $$\limsup_{n\to\infty}\sup_{F\in \mathcal{F}(cn^{-1})} \P_{F^n}\{\hat\zeta_n\leq c_{1-\alpha}\} \leq \beta$$
    for all $c>C(\alpha,\beta)$. Now take any $\Pr\in\mathcal{F}(cn^{-1})$ with $c>4/\alpha$ (i.e. $\zeta({\Pr})>4/n\alpha$). Using the fact $c_{1-\alpha}\leq 2((n-1)\alpha)^{-1}$ and Chebyshev's inequality, we have
    \begin{align*}
        \P_{F^n}\{\hat\zeta_n\leq c_{1-\alpha}\} \leq \P_{F^n}\{\hat\zeta_n\leq 2((n-1)\alpha)^{-1}\} & \leq \P_{F^n}\{\zeta(\Pr)-\hat\zeta_n\geq \zeta(\Pr) - 2((n-1)\alpha)^{-1}\}\\
        & \leq \frac{\var\big(\hat\zeta_n\big)}{\big(\zeta(\Pr) - 2((n-1)\alpha)^{-1}\big)^2},
    \end{align*}
    which holds since $\zeta(\Pr)-2((n-1)\alpha)^{-1}>4(n\alpha)^{-1}-2((n-1)\alpha)^{-1}=
    \frac{2n-4}{n(n-1)\alpha}>0$ for all $n\geq 2$. Now,
    \begin{align}
    \label{eq:chebyshev-bound}
        \frac{\var\big(\hat\zeta_n\big)}{\big(\zeta(\Pr) - 2((n-1)\alpha)^{-1}\big)^2} \leq \frac{{n\choose 2}^{-1}\left[4(n-2) \zeta(\Pr) + 4\right]}{\big(\zeta(\Pr)-2((n-1)\alpha)^{-1}\big)^2}~~~~~~ \mbox{(follows from Lemma \ref{var-inequality})}
    \end{align}
    which implies that
    \begin{align*}
        \limsup_{n\to\infty}\sup_{F\in\mathcal{F}(cn^{-1})} \P_{F^n}\{\hat\zeta_n\leq c_{1-\alpha}\} \leq \frac{4c + 4}{\big(c-2/\alpha\big)^2}.
    \end{align*}
    It is easy to see that the upper bound is a monotonically decreasing function of $c$ for $c>4/\alpha$ and it converges to $0$ as $c$ increases. Hence for any $\beta<1-\alpha$, there exists a $r(\alpha,\beta)$ such that the upper bound is smaller than $\beta$ whenever $c>r(\alpha,\beta)$. Now set $C(\alpha,\beta) = \max\{r(\alpha,\beta),4/\alpha\}$. Then for any $c>C(\alpha,\beta)$ the maximum type II error rate of our test is upper bounded by $\beta$. This completes the proof of this theorem.

\end{proof}

\begin{proof}[\bf Proof of Theorem 3.4]
    {If the distribution $F^{(d)}$ is such that $n\zeta(F^{(d)})$ diverges to infinity, then from equation \eqref{eq:chebyshev-bound} we get $\lim\P\{\hat\zeta_n\leq c_{1-\alpha}\} = 0$. Hence, under the above condition, the power of our test converges to one. }
\end{proof}

\begin{proof}[\bf Proof of Proposition 3.1]
    It is easy to see that the likelihood ratio of $F_{1-\frac{\beta}{\sqrt{n}}} = (1-\beta_n/\sqrt{n})G+\beta_n/\sqrt{n}F$ and $G$ is $\Big(1+\frac{\beta_n}{\sqrt{n}}\big(f({\bf u})/g({\bf u})-1\big)\Big)$. Hence if ${\bf X}_1,{\bf X}_2,\ldots,{\bf X}_n$ are i.i.d. observations from $G$, then the log-likelihood ratio is given by,
    $$L_N = \log\Big\{\prod_{i=1}^n\frac{dF_{1-\frac{\beta_n}{\sqrt{n}}}}{dG}({\bf X}_i)\Big\} = \sum_{i=1}^n \log\Big\{\frac{dF_{1-\frac{\beta_n}{\sqrt{n}}}}{dG}({\bf X}_i)\Big\} = \sum_{i=1}^n \log\Big(1+\frac{\beta_n}{\sqrt{n}}\big(f({\bf X}_i)/g({\bf X}_i)-1\big)\Big).$$
    Using the fact that
    $\log(1+y) = y - \frac{y^2}{2}+\frac{1}{2}y^2h(y)$ where $h(\cdot)$ is continuous and $\lim\limits_{y\to0}h(y)=0,$ we get 
    \begin{align*}
        L_N = & ~\sum_{i=1}^n \frac{\beta_n}{\sqrt{n}}\big(f({\bf X}_i)/g({\bf X}_i)-1\big)-\sum_{i=1}^n \frac{\beta_n^2}{2n}\big(f({\bf X}_i)/g({\bf X}_i)-1\big)^2\\
        &~~~~~~~+\sum_{i=1}^n \frac{\beta_n^2}{2n}\big(f({\bf X}_i)/g({\bf X}_i)-1\big)^2h\Big(\frac{\beta_n}{\sqrt{n}}\big(f({\bf X}_i)/g({\bf X}_i)-1\big)\Big).
    \end{align*}

    Under assumption $\int (f({\bf u})/g({\bf u})-1)^2 g({\bf u}) \mathrm d{\bf u}$ is finite and $\beta_n\rightarrow\beta$, as $n$ grows to infinity, we have
    $$\sum_{i=1}^n \frac{\beta_n^2}{n}\big(f({\bf X}_i)/g({\bf X}_i)-1\big)^2\stackrel{a.s.}{\to} \beta^2 \E\Big(\big(f({\bf X}_1)/g({\bf X}_1)-1\big)^2\Big).$$
    Hence we only need to show that 
    $$\sum_{i=1}^n \frac{\beta_n^2}{n}\big(f({\bf X}_i)/g({\bf X}_i)-1\big)^2h\Big(\frac{\beta_n}{\sqrt{n}}\big(f({\bf X}_i)/g({\bf X}_i)-1\big)\Big)$$
    converges to zero in probability. Notice that
    \begin{align*}
        &\sum_{i=1}^n \frac{\beta_n^2}{n}\big(f({\bf X}_i)/g({\bf X}_i)-1\big)^2h\Big(\frac{\beta_n}{\sqrt{n}}\big(f({\bf X}_i)/g({\bf X}_i)-1\big)\Big)\\
        & \leq \max_{1\leq i\leq n}|h\Big(\frac{\beta_n}{\sqrt{n}}\big(f({\bf X}_i)/g({\bf X}_i)-1\big)\Big)|\sum_{i=1}^n \frac{\beta_n^2}{n}\big(f({\bf X}_i)/g({\bf X}_i)-1\big)^2.
    \end{align*}
    Therefore, it suffices to show that $\max_{1\leq i\leq N}|h\Big(\frac{\beta_n}{\sqrt{n}}\big(f({\bf X}_i)/g({\bf X}_i)-1\big)\Big)|$ converges to zero in probability, which follows if $\max_{1\leq i\leq N}|\frac{\beta_n}{\sqrt{n}}\big(f({\bf X}_i)/g({\bf X}_i)-1\big)|$ converges to zero in probability (as $\lim_{y\to 0} h(y) = 0$ and it is continuous). Note that,
    \begin{align*}
        & \P\Big\{\max_{1\leq i\leq n}\big|\frac{1}{\sqrt{n}}\big(f({\bf X}_i)/g({\bf X}_i)-1\big)\big|>\epsilon\Big\}\\
        & \leq\sum_{i=1}^n \P\Big\{\big|\frac{1}{\sqrt{n}}\big(f({\bf X}_i)/g({\bf X}_i)-1\big)\big|>\epsilon\Big\}\\
        & = n \P\Big\{\big|\frac{1}{\sqrt{n}}\big(f({\bf X}_1)/g({\bf X}_1)-1\big)\big|>\epsilon\Big\}\\
        & = n \E\Big\{I\big(\big|\frac{1}{\sqrt{n}}\big(f({\bf X}_1)/g({\bf X}_1)-1\big)\big|>\epsilon\big)\Big\}\\
        & \leq n \E\Big\{\frac{\big(f({\bf X}_1)/g({\bf X}_1)-1\big)^2}{n\epsilon^2}I\big|\frac{1}{\sqrt{n}}\big(f({\bf X}_1)/g({\bf X}_1)-1\big)\big|>\epsilon\Big\}\\
        & \leq \frac{1}{\epsilon^2}\E\Big\{\big(f({\bf X}_1)/g({\bf X}_1)-1\big)^2I\big(\big|\frac{1}{\sqrt{n}}\big(f({\bf X}_1)/g({\bf X}_1)-1\big)\big|>\epsilon\big)\Big\}.
    \end{align*}
    Since $I\big(\big|\frac{1}{\sqrt{n}}\big(f({\bf X}_1)/g({\bf X}_1)-1\big)\big|>\epsilon\big)$ converges to zero in probability, the right-hand side converges to zero by the Dominated Convergence Theorem. Hence, we have
    $$\left|\log\Big\{\prod_{i=1}^n\frac{dF_{1-\beta_n/\sqrt{n}}}{dG}({\bf X}_i)\Big\}-\frac{\beta_n}{\sqrt{n}}\sum_{i=1}^n \Big(f({\bf X}_i)/g({\bf X}_i)-1\Big)+\frac{\beta_n^2}{2}\E\Big\{f({\bf X}_1)/g({\bf X}_1)-1\Big\}^2\right|\to 0,$$
    in probability as $n$ goes to infinity. This completes the proof.
\end{proof}

\begin{proof}[\bf Proof of Theorem 3.5]
    Let ${\bf X}_1,{\bf X}_2,\ldots, {\bf X}_n \stackrel{i.i.d.}{\sim}G$ and ${\bf U}_1,{\bf U}_2,\ldots, {\bf U}_n\stackrel{i.i.d.}{\sim}$ Unif$(\mathcal{S}^{d-1})$ be independent and 
    $\zeta(G)=0$. For $i=1,2,\ldots,n$, define, ${\bf X}_i^\prime = \|{\bf X}_i\|{\bf U}_i$. 
     To prove this theorem we only need to find the limit distribution of $\sqrt{n}\big(\frac{1}{n}\sum_{i=1}^n h({\bf X}_i,{\bf X}_i^\prime) - \E\big\{h({\bf X}_1,{\bf X}_1^\prime)\big\}\big)$ for some square-integrable function $h$ under the contiguous alternative $F_{1-\beta/\sqrt{n}}= \big(1-\frac{\beta_n}{\sqrt{n}}\big) G + \frac{\beta_n}{\sqrt{n}} F$, where $\zeta(F)>0$ and $\beta_n\rightarrow\beta$ as $n\rightarrow\infty$. Using the bivariate central limit theorem, we can say that as $n$ diverges to infinity, the joint distribution of 
    $$\sqrt{n}\big(\frac{1}{n}\sum_{i=1}^n h({\bf X}_i,{\bf X}_i^\prime) - \E\big\{h({\bf X}_1,{\bf X}_1^\prime)\big\}\big) \text{ and } \frac{\beta_n}{\sqrt{n}}\sum_{i=1}^n \big(\frac{f({\bf X}_i)}{g({\bf X}_i)}-1\big)-\frac{\beta_n^2}{2}\E\big\{\frac{f({\bf X}_1)}{g({\bf X}_1)}-1\big\}^2$$
    converges to a bivariate normal distribution with\\
    $$\text{mean } \mu = \begin{pmatrix}0\\ -\frac{\beta^2}{2}\E\big\{\frac{f({\bf X}_1)}{g({\bf X}_1)}-1\big\}^2\end{pmatrix}$$ 
    and variance
    $$\Sigma = 
    \Biggl(
    \begin{tabular}{cc}
    $\var\big( h({\bf X}_1,{\bf X}_1^\prime)\big)$ & $\tau$\\ $\tau$  & $-\frac{\beta^2}{2}\E\big\{\frac{f({\bf X}_1)}{g({\bf X}_1)}-1\big\}^2$
    \end{tabular}
     \Biggr),$$
    where 
    \begin{align*}
    \tau  & = \E\Big\{\big\{h({\bf X}_1,{\bf X}_1^\prime) - \E\{h({\bf X}_1,{\bf X}_1^\prime)\}\big\}\beta\big\{\frac{f({\bf X}_1)}{g({\bf X}_1)}-1\big\}\Big\}\\
    & = \beta\int \big\{h({\bf x},\|{\bf x}\|{\bf u}) - \E\{h({\bf X}_1,{\bf X}_1^\prime)\}\big\}\big(f({\bf x})-g({\bf x})\big) \mathrm d{\bf x} \mathrm d\mu({\bf u}),    
    \end{align*}
    $\mu$ being the distribution Unif$(\mathcal{S}^{d-1})$. Now using Le Cam's third lemma \citep[see.][]{van2000asymptotic}, as $n$ diverges to infinity, under $F_{1-\beta_n/\sqrt{n}}$, we have 
    $$\sqrt{n}\big(\frac{1}{n}\sum_{i=1}^n h({\bf X}_i,{\bf X}_i^\prime) - \E\big\{h({\bf X}_1,{\bf X}_1^\prime)\big\}\big)\stackrel{D}{\to} N\Big(\tau, \var\big(h({\bf X}_1,{\bf X}_1^\prime)\big)\Big).$$
    Now using similar arguments as in Theorem 1 in p.79 from \cite{lee2019u} and contiguity arguments, we get 
    \begin{align*}
        n\hat\zeta_n\stackrel{D}{\to} & \sum_{i=1}^\infty \lambda_i \left(\big(Z_i+\beta~\E_F\{f_i({\bf X}_1,{\bf X}_1^\prime)\}\big)\Big)^2-1\right)
    \end{align*}
    under $F_{1-\beta_n/\sqrt{n}}$, where $\{Z_i\}$ is a sequence of i.i.d. normal random variables. This completes the proof.
\end{proof}

{
\begin{lemma}
    Let $\Xvec_1,\ldots, \Xvec_n$ be independent copies of  $\Xvec \sim \Pr$ and $\Xvec_1^\prime, \ldots, \Xvec_n^\prime$ be their spherically symmetric variants. Let $\mathcal{D}^\prime = \{(\Xvec_i,\Xvec_i^\prime)\}_{1\le i\leq n}$ be the augmented data set. For $i=1,2,\ldots,n$, define $\Yvec_i = \pi_i \Xvec_i + (1-\pi_i)\Xvec_i^\prime$ and $\Yvec_i^\prime = \pi_i \Xvec_i^\prime + (1-\pi_i)\Xvec_i$, where $\pi_1,\ldots, \pi_n\stackrel{iid}{\sim} \textnormal{Bernoulli}(0.5)$. Let $f(\cdot,\cdot)$ be a function such that $f(\xvec, \xvec^\prime) = -f(\xvec^\prime, \xvec)$ and $\E[f(\Yvec_1,\Yvec_1^\prime)^2]$ is finite. Then, the distribution of $n^{-1/2}\sum_{i=1}^n f(\Yvec_i,\Yvec_i^\prime)$ conditioned on $\mathcal{D}^\prime$ converges weakly to a $\mathcal{N}_1(0,\E[f^2(\Yvec_1,\Yvec_1^\prime)])$ in probability as $n$ grows to infinity. 
    \label{conditional-CLT}
\end{lemma}

\begin{proof}[\bf Proof of Lemma \ref{conditional-CLT}]
    Since $\E[f(\Yvec_1, \Yvec_1^\prime)^2]$ is finite, both $\E[f(\Xvec_1, \Xvec_1^\prime)^2]$ and $\E[f(\Xvec_1^\prime, \Xvec_1)^2]$ are also finite. By $f(\xvec,\xvec^\prime) = - f(\xvec^\prime, \xvec)$, we also have $\E[f(\Yvec_1, \Yvec_1^\prime)\mid \mathcal{D}^\prime]$ = 0. Now, note that conditioned on $\mathcal{D}^\prime$, $\{f(\Yvec_j,\Yvec_j^\prime)\}_{1\leq j\leq n}$ form a triangular array of independent random variables. Therefore, we can use arguments similar to the Lindeberg's CLT to prove this result. Let us first define
    \begin{align*}
        S_n^2 & = \sum_{j=1}^n \var[f(\Yvec_j, \Yvec_j^\prime)\big|\mathcal{D}^\prime] = \sum_{j=1}^n \E[f(\Yvec_j, \Yvec_j^\prime)^2\big|\mathcal{D}^\prime]= \sum_{j=1}^n \frac{1}{2} f(\Xvec_j, \Xvec_j^\prime)^2 + \sum_{j = 1}^n \frac{1}{2} f(\Xvec_j^\prime, \Xvec_j)^2.
    \end{align*}
    So, by the strong law of large numbers, we get
    \begin{align*}
        \frac{S_n^2}{n} \stackrel{a.s.}{\rightarrow}\frac{1}{2}\E[f(\Xvec_1, \Xvec_1^\prime)^2] + \frac{1}{2} \E[f(\Xvec_1^\prime, \Xvec_1)^2] = \E[f(\Yvec_1, \Yvec_1^\prime)^2] \mbox{    as } n \rightarrow \infty.
    \end{align*}
     Now, using the inequality 
       $\Big|e^{itx} - \big(1 + itx - \frac{1}{2} t^2 x^2\big)\Big| \leq \min\Big\{\big|tx\big|^2, \big|tx\big|^3\Big\}$,
    we get
    \begin{align*}
        \Big|\E\big\{e^{it\frac{1}{\sqrt{n}}f({\bf Y}_j,{\bf Y}_j^\prime)} \big|\mathcal{D} \big\} - &\big(1-\frac{t^2}{2n} \E\big\{f^2(\Yvec_j\Yvec_j^\prime)\big|\mathcal{D}^\prime\big\}\big)\Big|\\
        & \leq \E\Bigg\{\min\Big\{\big|\frac{t}{\sqrt{n}}f({\bf Y}_j,{\bf Y}_j^\prime)\big|^2, \big|\frac{t}{\sqrt{n}}f({\bf Y}_j,{\bf Y}_j^\prime)\big|^3\Big\} \big|\mathcal{D}^\prime\Bigg\}.
    \end{align*}
    Note that the expectation on the right hand side of the above inequality is finite. Now take any arbitrarily small $\epsilon>0$ and note that
    \begin{align*}
        & \E\Bigg\{\min\Big\{\big|\frac{t}{\sqrt{n}}f({\bf Y}_j,{\bf Y}_j^\prime)\big|^2, \big|\frac{t}{\sqrt{n}}f({\bf Y}_j,{\bf Y}_j^\prime)\big|^3\Big\} \big|\mathcal{D}^\prime\Bigg\}\\
        & = \E\Bigg\{\min\Big\{\big|\frac{t}{\sqrt{n}}f({\bf Y}_j,{\bf Y}_j^\prime)\big|^2, \big|\frac{t}{\sqrt{n}}f({\bf Y}_j,{\bf Y}_j^\prime)\big|^3\Big\} I[|f({\bf Y}_j,{\bf Y}_j^\prime)|<\epsilon S_n] \big|\mathcal{D}^\prime\Bigg\}\\
        & ~~~~~ + \E\Bigg\{\min\Big\{\big|\frac{t}{\sqrt{n}}f({\bf Y}_j,{\bf Y}_j^\prime)\big|^2, \big|\frac{t}{\sqrt{n}}f({\bf Y}_j,{\bf Y}_j^\prime)\big|^3\Big\} I[|f({\bf Y}_j,{\bf Y}_j^\prime)|\geq \epsilon S_n] \big|\mathcal{D}^\prime\Bigg\}\\
        & \leq \frac{\epsilon S_n |t|^3}{n^{3/2}} \E\Big\{ f({\bf Y}_j,{\bf Y}_j^\prime)^2\big|\mathcal{D}^\prime\Big\} \\
        & ~~~~~~~~~~~~ + \frac{t^2}{2n}\Big\{f({\bf X}_j,{\bf X}_j^\prime)^2 I[|f({\bf X}_j,{\bf X}_j^\prime)|\geq \epsilon S_n] + f({\bf X}_j^\prime,{\bf X}_j)^2 I[|f({\bf X}_j^\prime,{\bf X}_j)|\geq \epsilon S_n]\Big\}.
    \end{align*}
    Then, by the assumption $\E[f({\bf Y}_j,{\bf Y}_j^\prime)^2]<\infty$, we get
    \begin{align*}
        & \sum_{j=1}^n \Big|\E[e^{it \frac{1}{\sqrt{n}}f({\bf Y}_j,{\bf Y}_j^\prime)}\big|\mathcal{D}^\prime] - \big(1-\frac{t^2}{2n} \E[f({\bf Y}_j,{\bf Y}_j^\prime)^2\big|\mathcal{D}^\prime]\big)\Big|\\
        & \leq \frac{\epsilon S_n |t|^3}{\sqrt{n}} \Big\{\frac{1}{2n}\sum_{j=1}^n f({\bf X}_j,{\bf X}_j^\prime)^2 + \frac{1}{2n}\sum_{j=1}^n f({\bf X}_j^\prime,{\bf X}_j)^2\Big\}\\
        &~~~~~~ + \frac{t^2}{2} \Big\{\frac{1}{n} \sum_{j=1}^n f({\bf X}_j,{\bf X}_j^\prime)^2 I[|f({\bf X}_j,{\bf X}_j^\prime)|\geq \epsilon S_n] + \frac{1}{n} \sum_{j=1}^n f({\bf X}_j^\prime,{\bf X}_j)^2 I[|f({\bf X}_j^\prime,{\bf X}_j)|\geq \epsilon S_n]\Big\}.
    \end{align*}
    Now, by the strong law of large numbers, we have
    \begin{align*}
        & \frac{\epsilon S_n |t|^3}{\sqrt{n}} \Big\{\frac{1}{2n}\sum_{j=1}^n f({\bf X}_j,{\bf X}_j^\prime)^2 + \frac{1}{2n}\sum_{j=1}^n f({\bf X}_j^\prime,{\bf X}_j)^2\Big\}\\
        & \stackrel{a.s.}{\rightarrow} \epsilon |t|^3 \sqrt{\E[f({\bf Y}_1, {\bf Y}_1)^2]} \Big\{\frac{1}{2} \E[f({\bf X}_1,{\bf X}_1^\prime)^2] + \frac{1}{2} \E[f({\bf X}_1^\prime,{\bf X}_1)^2]\Big\} = \epsilon |t|^3 \E[f({\bf Y}_1, {\bf Y}_1)^2]^{3/2},
    \end{align*}
    which is arbitrarily small (since $\epsilon$ is arbitrarily small). Again, using DCT, we have
    \begin{align*}
        & \E\left[\frac{1}{n} \sum_{j=1}^n f({\bf X}_j,{\bf X}_j^\prime)^2 I[|f({\bf X}_j,{\bf X}_j^\prime)|\geq \epsilon S_n] \right] = \E\left[f({\bf X}_1,{\bf X}_1^\prime)^2 I[|f({\bf X}_1,{\bf X}_1^\prime)|\geq \epsilon S_n] \right] \rightarrow 0
    \end{align*}
    as $n$ grows to infinity. Therefore, $\frac{1}{n} \sum_{j=1}^n f({\bf X}_j,{\bf X}_j^\prime)^2 I[|f({\bf X}_j,{\bf X}_j^\prime)|\geq \epsilon S_n]$ converges to zero in probability as $n$ grows to infinity. Similarly, we can also show that $$\frac{1}{n} \sum_{j=1}^n f({\bf X}_j^\prime,{\bf X}_j)^2 I[|f({\bf X}_j^\prime,{\bf X}_j)|\geq \epsilon S_n]\stackrel{P}{\rightarrow} 0$$ 
    as $n$ grows to infinity. Therefore, by repeated application of the triangle inequality, we get
    \begin{align*}
        & \left|\prod_{j=1}^n \E\Big[e^{it \frac{1}{\sqrt{n}} f({\bf Y}_j^\prime,{\bf Y}_j) }\big|\mathcal{D}^\prime\Big] - \prod_{j=1}^n \big(1-\frac{t^2}{2n} \E\big[f({\bf Y}_j, {\bf Y}_j^\prime)^2\big|\mathcal{D}^\prime\big]\big)\right|\\
        & \leq \sum_{j=1}^n \Big|\E\Big[e^{it \frac{1}{\sqrt{n}} f({\bf Y}_j^\prime,{\bf Y}_j) }\big|\mathcal{D}^\prime\Big] - \big(1-\frac{t^2}{2n} \E\big[f({\bf Y}_j, {\bf Y}_j^\prime)^2\big|\mathcal{D}^\prime\big]\big)\Big| \stackrel{P}{\rightarrow} 0,
    \end{align*}
    as $n$ grows to infinity. Before proceeding further let us note that for any $\delta>0$, 
    \begin{align*}
        &\Big|\frac{1}{n} \E\big[f({\bf Y}_j,{\bf Y}_j^\prime)^2\big|\mathcal{D}^\prime\big]\Big| \\
        & \leq \delta^2 +\frac{1}{n}\E\big[f({\bf Y}_j, {\bf Y}_j^\prime)^2 I[|f({\bf Y}_j, {\bf Y}_j^\prime)|\geq \delta\sqrt{n}]\big|\mathcal{D}^\prime\big]\\
        & = \delta^2 + \frac{1}{2n}f({\bf X}_j, {\bf X}_j^\prime)^2 I[|f({\bf X}_j, {\bf X}_j^\prime)|\geq \delta\sqrt{n}]+\frac{1}{2n}f({\bf X}_j^\prime, {\bf X}_j)^2 I[|f({\bf X}_j^\prime, {\bf X}_j)|\geq \delta\sqrt{n}].
    \end{align*}
    \begin{align*}
        \text{So,  } &
    \max_{1\leq j\leq n} \Big|\frac{1}{n} \E\big[f({\bf Y}_j, {\bf Y}_j^\prime)^2\big|\mathcal{D}^\prime\big]\Big|\\
        & \leq \delta^2 + \frac{1}{2n}\sum_{j=1}^n f({\bf X}_j, {\bf X}_j^\prime)^2 I[|f({\bf X}_j, {\bf X}_j^\prime)|\geq \delta\sqrt{n}] + \frac{1}{2n}\sum_{j=1}^n f({\bf X}_j^\prime, {\bf X}_j)^2 I[|f({\bf X}_j^\prime, {\bf X}_j)|\geq \delta\sqrt{n}].
    \end{align*}
    Since, $\delta^2$ is of arbitrary and the other term is of order $o_p(1)$, $\max_{1\leq j\leq n} \Big|\frac{1}{n} \E\big[f({\bf Y}_j, {\bf Y}_j^\prime)^2\big|\mathcal{D}^\prime\big]\Big|\stackrel{P}{\rightarrow} 0$ as $n \rightarrow \infty$. Also note that
    \begin{align*}
        & \left|\prod_{j=1}^n \exp\Big\{-\frac{t^2}{2n} \big(\frac{1}{2} f({\bf X}_j, {\bf X}_j^\prime)^2 + \frac{1}{2} f({\bf X}_j^\prime, {\bf X}_j)^2\big)\Big\} - \prod_{j=1}^n \Big\{1-\frac{t^2}{2n} \big(\frac{1}{2} f({\bf X}_j, {\bf X}_j^\prime)^2 + \frac{1}{2} f({\bf X}_j^\prime, {\bf X}_j)^2\big)\Big\}\right|\\
        & \leq \sum_{j=1}^n \left|\exp\Big\{-\frac{t^2}{2n} \big(\frac{1}{2} f({\bf X}_j, {\bf X}_j^\prime)^2 + \frac{1}{2} f({\bf X}_j^\prime, {\bf X}_j)^2\big)\Big\} - \Big\{1-\frac{t^2}{2n} \big(\frac{1}{2} f({\bf X}_j, {\bf X}_j^\prime)^2 + \frac{1}{2} f({\bf X}_j^\prime, {\bf X}_j)^2\big)\Big\}\right|\\
        & \leq \sum_{j=1}^n \left|\frac{t^2}{4n} \big( f({\bf X}_j, {\bf X}_j^\prime)^2 + f({\bf X}_j^\prime, {\bf X}_j)^2\big)\right|^2 \exp\Big\{\frac{t^2}{4n} \big( f({\bf X}_j, {\bf X}_j^\prime)^2 + f({\bf X}_j^\prime, {\bf X}_j)^2\big)\Big\}\\
        &\leq \exp\Big\{\max_{1\leq j\leq n}\frac{t^2}{4n} \big( f({\bf X}_j, {\bf X}_j^\prime)^2 + f({\bf X}_j^\prime, {\bf X}_j)^2\big)\Big\} \max_{1\leq j\leq n}\frac{t^2}{4n} \big( f({\bf X}_j, {\bf X}_j^\prime)^2 + f({\bf X}_j^\prime, {\bf X}_j)^2\big)\\
        & ~~~~~~~~~~~~~~~~~~~~~\sum_{j=1}^n \frac{t^2}{4n} \big( f({\bf X}_j, {\bf X}_j^\prime)^2 + f({\bf X}_j^\prime, {\bf X}_j)^2\big).
    \end{align*}
    Clearly, this converges to $0$ as $n$ grows to infinity. Therefore,
    \begin{align*}
        \prod_{j=1}^n \E\Big[e^{it \frac{1}{\sqrt{n}} f({\bf Y}_j^\prime,{\bf Y}_j) }\big|\mathcal{D}^\prime\Big] & = \exp\Big\{-\frac{t^2}{2} \frac{1}{n}\sum_{j=1}^n \big(f({\bf X}_j,{\bf X}_j^\prime)^2+f({\bf X}_j^\prime,{\bf X}_j)^2\big)\Big\} + o_P(1)\\
        & = \exp\Big\{-\frac{t^2}{2} \E[f({\bf Y}_1,{\bf Y}_1^\prime)^2]\Big\} + o_P(1)
    \end{align*}
    This completes the proof. 
\end{proof}

}

{
\begin{proof}[\bf Proof of Theorem 3.6]
    Let ${\bf X}_1,{\bf X}_2,\ldots, {\bf X}_n \stackrel{iid}{\sim}F$ where $\zeta(F) = 0$ and ${\bf U}_1,{\bf U}_2,\ldots, {\bf U}_n\stackrel{iid}{\sim}$ Unif$(\mathcal{S}^{d-1})$ be independent. For $i=1,2,\ldots,n$, define ${\bf X}_i^\prime = \|{\bf X}_i\|{\bf U}_i$. 
   Also, assume that $\pi_1,\pi_2,\ldots,\pi_n\stackrel{iid}{\sim}\textnormal{Bernoulli}(0.5)$. The resampled test statistic $\hat\zeta_n(\pi)$ can be written as
    $$\hat\zeta_n(\pi)=\frac{1}{n(n-1)}\sum_{1\leq i\not=j\leq n} g\big(({\bf Y}_i,{\bf Y}_i^\prime),({\bf Y}_j,{\bf Y}_j^\prime)\big),$$
    where ${\bf Y}_i = \pi_i {\bf X}_i+(1-\pi_i){\bf X}_i^\prime$, ${\bf Y}_i^\prime = (1-\pi_i){\bf X}_i+\pi_i {\bf X}_i^\prime$ and $g(\cdot, \cdot)$ {  as defined in equation (3) in the main article}. Denote, $\mathcal{D}^\prime = \{(\Xvec_i,\Xvec_i^\prime)\}_{1\leq i\leq n}$ as the augmented data set as before. So, conditioned on $\mathcal{D}^\prime$, the randomness within $\hat\zeta_n(\pi)$ comes from $\pi_1,\ldots, \pi_n$ only. Then, we get
    \begin{align*}
        \E\Big[\exp\big\{-\frac{\|\Yvec_1 - \Yvec_2\|^2}{2d}\big\}\big|(\Yvec_1, \Yvec_1^\prime), \mathcal{D}^\prime\Big] = \frac{1}{2} \sum_{\pi\in\{0,1\}} \left[\exp\big\{-\frac{\|\Yvec_1 - \pi \Xvec_2 - (1-\pi) \Xvec_2^\prime\|^2} {2d}\big\} \right].
    \end{align*}
    (Note that here the randomness is due to $\pi_2$ only). Similarly, we can also get
    \begin{align*}
        \E\Big[\exp\big\{-\frac{\|\Yvec_1 - \Yvec_2^\prime\|^2}{2d}\big\}\big| (\Yvec_1, \Yvec_1^\prime),\mathcal{D}^\prime\Big] & = \frac{1}{2} \sum_{\pi\in\{0,1\}} \left[\exp\big\{-\frac{\|\Yvec_1 - \pi \Xvec_2^\prime - (1-\pi) \Xvec_2\|^2} {2d}\big\} \right],\\
        \E\Big[\exp\big\{-\frac{\|\Yvec_1^\prime - \Yvec_2\|^2}{2d}\big\}\big|(\Yvec_1, \Yvec_1^\prime), \mathcal{D}^\prime\Big] & = \frac{1}{2} \sum_{\pi\in\{0,1\}} \left[\exp\big\{-\frac{\|\Yvec_1^\prime - \pi \Xvec_2- (1-\pi) \Xvec_2^\prime\|^2} {2d}\big\} \right],\text{ and}\\
        \E\Big[\exp\big\{-\frac{\|\Yvec_1^\prime - \Yvec_2^\prime\|^2}{2d}\big\}\big| (\Yvec_1, \Yvec_1^\prime),\mathcal{D}^\prime\Big] & = \frac{1}{2} \sum_{\pi\in\{0,1\}} \left[\exp\big\{-\frac{\|\Yvec_1^\prime - \pi \Xvec_2 - (1-\pi) \Xvec_2^\prime\|^2} {2d}\big\} \right].       
    \end{align*}
    Hence, we get $\E[g\big((\Yvec_1,\Yvec_1^\prime),(\Yvec_2,\Yvec_2^\prime)\big)\big|(\Yvec_1, \Yvec_1^\prime),\mathcal{D}^\prime] = 0$. Therefore, the limiting conditional distribution of $n\hat\zeta_n(\pi)$ should be evaluated as in the case of unconditional degenerate U-statistics. Using the eigen decomposition of $g(\cdot,\cdot)$ with respect to  $F$, we get
    $$g\big((\xvec, \xvec^\prime),(\yvec, \yvec^\prime)\big) = \sum_{i=1}^\infty \lambda_i \varphi_i(\xvec, \xvec^\prime) \varphi_i(\yvec,\yvec^\prime),$$
    where $\lambda_i$ s and $\varphi_i$s are as in Theorem 2.3. For any $K\in\mathbb{N}$, define $$g_K\big((\xvec, \xvec^\prime),(\yvec, \yvec^\prime)\big) = \sum_{i=1}^K \lambda_i \varphi_i(\xvec, \xvec^\prime) \varphi_i(\yvec,\yvec^\prime).$$ 
    Let $U_n(f) := \frac{1}{(n-1)} \sum_{1\leq i \not = j\leq n} f\big((\Yvec_i,\Yvec_i^\prime),(\Yvec_j,\Yvec_j^\prime)\big)$. Then,
    \begin{align*}
        & \E\big[\big(U_n(g-g_K)\big)^2\big|\mathcal{D}^\prime\big]\\
        & = \frac{1}{(n-1)^2} \sum_{1\leq i_1\not=\j_1\leq n} \sum_{1\leq i_2\not=\j_2\leq n}\E\left[(g-g_K)\big((\Yvec_{i_1},\Yvec_{i_1}^\prime),(\Yvec_{j_1},\Yvec_{j_1}^\prime)\big)(g-g_K)\big((\Yvec_{i_2},\Yvec_{i_2}^\prime),(\Yvec_{j_2},\Yvec_{j_2}^\prime)\big)\big| \mathcal{D}^\prime\right] \\
        & =  \frac{2}{(n-1)^2} \sum_{1\leq i_1\not= \j_1\leq n} \E\Big[\Big((g-g_K)\big((\Yvec_{i_1},\Yvec_{i_1}^\prime),(\Yvec_{j_1},\Yvec_{j_1}^\prime)\big)\Big)^2\big| \mathcal{D}^\prime\Big]\\ & ~~~~~~~\text{( since the expectation of the cross-product terms are 0)}\\
        & = \frac{2}{(n-1)^2} \sum_{1\leq i_1\not= \j_1\leq n} \frac{1}{4}\sum_{\pi_{i_1}, \pi_{j_1}\in\{0,1\}}\Big((g-g_K)\big((\pi_{i_1}\Xvec_{i_1} + (1-\pi_{i_1})\Xvec_{i_1}^\prime,\pi_{i_1}\Xvec_{i_1}^\prime + (1-\pi_{i_1})\Xvec_{i_1}), \\
        &\hspace{3.5in} (\pi_{j_1}\Xvec_{j_1} + (1-\pi_{j_1})\Xvec_{j_1}^\prime,\pi_{j_1}\Xvec_{j_1}^\prime + (1-\pi_{j_1})\Xvec_{j_1})\big)\Big)^2\\
        & = \frac{1}{2} [\bm \Sigma_I + \bm \Sigma_{II} + \bm \Sigma_{III} + \bm \Sigma_{IV}],   
    \end{align*}
    \begin{align*}
        \text{where   } \bm \Sigma_I & = \frac{1}{(n-1)^2}  \sum_{1\leq i_1\not= \j_1\leq n} \Big((g-g_K)\big((\Xvec_{i_1} ,\Xvec_{i_1}^\prime),(\Xvec_{j_1},\Xvec_{j_1}^\prime\big)\Big)^2,\\
        \bm \Sigma_{II} & = \frac{1}{(n-1)^2}  \sum_{1\leq i_1\not= \j_1\leq n} \big((g-g_K)\big((\Xvec_{i_1}^\prime ,\Xvec_{i_1}),(\Xvec_{j_1},\Xvec_{j_1}^\prime\big)\big)^2,\\
        \bm \Sigma_{III} & = \frac{1}{(n-1)^2}  \sum_{1\leq i_1\not= \j_1\leq n} \big((g-g_K)\big((\Xvec_{i_1} ,\Xvec_{i_1}^\prime),(\Xvec_{j_1}^\prime,\Xvec_{j_1}\big)\big)^2, \textnormal{ and}\\
        \bm \Sigma_{IV} & = \frac{1}{(n-1)^2}  \sum_{1\leq i_1\not= \j_1\leq n} \big((g-g_K)\big((\Xvec_{i_1}^\prime,\Xvec_{i_1}),(\Xvec_{j_1}^\prime,\Xvec_{j_1}\big)\big)^2.
    \end{align*}
    Clearly, by the strong consistency of U-statistics, we get
    $$\E\big[\big(U_n(g-g_K)\big)^2\big|\mathcal{D}^\prime\big] \stackrel{a.s.}{\rightarrow}2 \E\left[\Big((g-g_K)\big((\Yvec_{1} ,\Yvec_{1}^\prime),(\Yvec_{2},\Yvec_{2}^\prime\big)\Big)^2\right]$$
    as $n$ grows to infinity. The limiting quantity converges to zero as $K$ diverges to infinity. This ensures the closeness of the limiting distribution of $U_n(g)$ and $U_n(g_K)$ when $K$ is large.
    
    Now, let us focus on the limiting distribution of $U_n(g_K)$. First, we find the joint limiting distribution of $\frac{1}{\sqrt{n}}\sum_{i=1}^n\Big(\varphi_1\big((\Xvec_i,\Xvec_i^\prime)\big), \ldots, \varphi_K\big((\Xvec_i,\Xvec_i^\prime)\big)\Big)$ conditioned on the augmented data $\mathcal{D}^\prime$. By the Cramer-Wold device, it is enough to focus on the conditional limit distribution of
    $$\frac{1}{\sqrt{n}}\sum_{i=1}^n \left[t_1\varphi_1\big((\Yvec_1,\Yvec_1^\prime)\big) + \ldots + t_K\varphi_K\big((\Yvec_K,\Yvec_K^\prime)\big)\right]$$
    for some real $t_1,\ldots,t_K$. Note that since $g\big(({\bf x},{\bf x}^\prime),({\bf y},{\bf y}^\prime)\big) = -g\big(({\bf x}^\prime,{\bf x}),({\bf y},{\bf y}^\prime)\big)$, we get
    \begin{align}
    \label{eq:anti-symm-phi}
        \varphi_i({\bf x}, {\bf x}^\prime) & = \int g\big(({\bf x},{\bf x}^\prime),({\bf y},{\bf y}^\prime)\big) \varphi_i({\bf y},{\bf y}^\prime) \mathrm d \nu({\bf y},{\bf y}^\prime)\nonumber\\
        & = - \int g\big(({\bf x}^\prime,{\bf x}),({\bf y},{\bf y}^\prime)\big) \varphi_i({\bf y},{\bf y}^\prime) \mathrm d \nu({\bf y},{\bf y}^\prime) = -\varphi_i({\bf x}^\prime,{\bf x}),
    \end{align}
    for $i=1,2,\ldots,K$, where $\nu$ denotes the joint distribution of $(\Xvec_1, \Xvec_1^\prime)$. So, 
    \begin{align*}
        &\E[\varphi_i(\Yvec_1, \Yvec_1^\prime)\big| \mathcal{D}] = \frac{1}{2}\big(\varphi_i(\Xvec_1, \Xvec_1^\prime) + \varphi_i(\Xvec_1^\prime, \Xvec_1)\big) = \frac{1}{2}\big(\varphi_i(\Xvec_1, \Xvec_1^\prime) - \varphi_i(\Xvec_1, \Xvec_1^\prime)\big) = 0, \text{ 
 and}\\
       & \E[\varphi_i(\Yvec_1, \Yvec_1^\prime)^2] = \frac{1}{2}\E\big(\varphi_i(\Xvec_1, \Xvec_1^\prime)^2 + \varphi_i(\Xvec_1^\prime, \Xvec_1)^2\big) = \frac{1}{2}\E\big(\varphi_i(\Xvec_1, \Xvec_1^\prime)^2 +  \varphi_i(\Xvec_1, \Xvec_1^\prime)^2\big) = 1.
    \end{align*} 
    
    Therefore, 
$\sum_{\ell=1}^K t_\ell\varphi_\ell\big((\Yvec_i,\Yvec_i^\prime)\big)$ satisfies the assumptions of Lemma \ref{conditional-CLT}. Hence, using Lemma \ref{conditional-CLT}, we get that conditioned on $\mathcal{D}^\prime$, 
    \begin{align*}
        &\frac{1}{\sqrt{n}}\sum_{i=1}^n \left[t_1\varphi_1\big((\Yvec_i,\Yvec_i^\prime)\big) + \ldots + t_K\varphi_K\big((\Yvec_i,\Yvec_i^\prime)\big)\right]\\ &\hspace{2.0in} \stackrel{D}{\rightarrow} \mathcal{N}_1\big(0, \sum_{\ell=1}^K t_\ell^2\E[\varphi_\ell^2(\Yvec_1, \Yvec_1^\prime)]\big)
         =  \mathcal{N}_1\big(0, \sum_{\ell=1}^K t_\ell^2\big)
    \end{align*}
    in probability as $n$ grows to infinity. By simple application of weak law of large numbers, it is easy to see that conditioned on $\mathcal{D}$, for any $\ell\in \mathbb{N}$, $\frac{1}{n}\sum_{j=1}^n \varphi_\ell(\Yvec_j, \Yvec_j^\prime)^2$ converges in probability to $\E[\varphi_\ell(\Yvec_1, \Yvec_1^\prime)^2] = 1$ as $n$ grows to infinity. Now, by continuous mapping theorem we can conclude that
    $$U_n(g_K) = \sum_{\ell=1}^K \lambda_\ell \Big( \big( \frac{1}{\sqrt{n}} \sum_{j = 1}^n \varphi_\ell(\Yvec_j, \Yvec_j^\prime) \big)^2  - \frac{1}{n} \sum_{j=1}^n \varphi_{\ell}(\Yvec_j, \Yvec_j^\prime)^2\Big)\stackrel{D}{\rightarrow} \sum_{\ell=1}^K \lambda_\ell (Z_\ell ^2 - 1)$$
    as $n$ grows to infinity, where $Z_\ell$s are i.i.d. standard normal random variables. The rest of the arguments can be established using approximation arguments using characteristic functions {  of the random variables $U_n(g_K), U_n(g), \sum_{\ell=1}^K \lambda_\ell (Z_\ell ^2 - 1)$ and $\sum_{\ell=1}^\infty \lambda_\ell (Z_\ell ^2 - 1)$}, as in Lemma 1 from \cite{banerjee2024twosam}. {  We also refer the interested readers to Chapter 12 of \cite{van2000asymptotic} or page 79 of \cite{lee2019u} for similar arguments.}
    

    Now suppose that ${\bf X}_1,{\bf X}_2,\ldots, {\bf X}_n \stackrel{iid}{\sim} F_{1-\beta n^{-1/2}} = (1-\frac{\beta}{\sqrt{n}}) G+ \frac{\beta}{\sqrt{n}} F$, where $\zeta(G) = 0$ and $\zeta(F)>0$. Then to find the limiting distribution of $n\zeta_n(\pi)$, we need to find the limiting distribution of
    $\frac{1}{\sqrt{n}}\sum_{i=1}^n \varphi_\ell\big(({\bf Y}_i,{\bf Y}_i^\prime)\big)$,
    where $\varphi_\ell$s are the solutions of the integral equation
    \begin{align}
        \int g(({\bf x}_1,{\bf x}_1^\prime), ({\bf y}_1,{\bf y}_1^\prime)) \varphi_\ell(({\bf y}_1,{\bf y}_1^\prime)) \mathrm d\nu\big(({\bf y}_1,{\bf y}_1^\prime)\big) = \lambda \varphi_\ell\big(({\bf x}_1,{\bf x}_1^\prime)\big)
        \label{eq:phi-g-relation}
    \end{align}
    where $\nu$ is the joint distribution of $({\bf Y}_1,{\bf Y}_1^\prime)$ when ${\bf X}_1\sim F_{1-\beta n^{-1/2}}$. Now, note that
    \begin{align*}
        &\E\Big\{\varphi_\ell(\Yvec_1,\Yvec_1^\prime)\big|\mathcal{D}^\prime\Big\} = \frac{1}{2} \Big\{\varphi_\ell(\Xvec_1,\Xvec_1^\prime) + \varphi_\ell(\Xvec_1^\prime,\Xvec_1)\Big\} = 0 \text{  (by \eqref{eq:anti-symm-phi})},\\
        &\E\Big\{\varphi_\ell(\Yvec_1,\Yvec_1^\prime)^2\big|\mathcal{D}^\prime\Big\} = \frac{1}{2} \Big\{\varphi_\ell(\Xvec_1,\Xvec_1^\prime)^2 + \varphi_\ell(\Xvec_1^\prime,\Xvec_1)^2\Big\} = \varphi_\ell(\Xvec_1,\Xvec_1^\prime)^2 \text{   and}\\
    &    \E\Big\{\varphi_\ell(\Xvec_1,\Xvec_1^\prime)^2\Big\}  = \Big(1-\frac{\beta}{\sqrt{n}}\Big) \int \varphi_\ell(\xvec,\xvec^\prime)^2 \mathrm d\nu_+ (\xvec, \xvec^\prime) + \frac{\beta}{\sqrt{n}} \int \varphi_\ell(\xvec,\xvec^\prime)^2 \mathrm d\nu_+^\ast (\xvec, \xvec^\prime)\\
        & \hspace{2.0in} \rightarrow \int \varphi_\ell(\xvec,\xvec^\prime)^2 \mathrm d\nu_+ (\xvec, \xvec^\prime), \textnormal{ as $n\rightarrow \infty$,}
    \end{align*}
    where $\nu_+$ and $\nu_+^\ast$ are the joint distribution of $({\bf X}_1,{\bf X}_1^\prime)$ when ${\bf X}_1\sim G$ and ${\bf X}_1\sim F$, respectively. Now, by contiguity of $F_{1-\beta/\sqrt{n}}$ and $G$, we can say that all probability convergence with respect to $G$ is still valid with respect to $F_{1-\beta/\sqrt{n}}$. Then, following similar arguments as in Lemma \ref{conditional-CLT}, under $F_{1-\beta/\sqrt{n}}$, we  get
    \begin{align*}
        \prod_{j=1}^n \E\Big[e^{it \frac{1}{\sqrt{n}} \varphi_\ell({\bf Y}_j^\prime,{\bf Y}_j) }\big|\mathcal{D}^\prime\Big] & = \exp\Big\{-\frac{t^2}{2} \frac{1}{n}\sum_{j=1}^n \big(\varphi_\ell({\bf X}_j,{\bf X}_j^\prime)^2+\varphi_\ell({\bf X}_j^\prime,{\bf X}_j)^2\big)\Big\} + o_P(1)\\
        & = \exp\Big\{-\frac{t^2}{2} \E[\varphi({\bf Y}_1,{\bf Y}_1^\prime)^2]\Big\} + o_P(1).
    \end{align*}
    The probability convergence of $\frac{1}{n}\sum_{j=1}^n \big(\varphi_\ell({\bf X}_j,{\bf X}_j^\prime)^2+\varphi_\ell({\bf X}_j^\prime,{\bf X}_j)^2\big)$ to $\E[\varphi({\bf Y}_1,{\bf Y}_1^\prime)^2]$ also follows by the contiguity of $F_{1-\beta/\sqrt{n}}$ and $G$. Now using similar arguments as before, we get $$n\hat\zeta_n(\pi)\stackrel{D}{\rightarrow}\sum_{i=1}^\infty \lambda_i (Z_i^2-1) $$
    as $n$ grows to infinity, where $\{\lambda_i\}$ is a square-integrable sequence and $\{Z_i\}$ is a sequence of i.i.d. standard normal random variables.  
\end{proof}
}

{
\begin{lemma}
    The function $\psi(x) = \exp\{-x^2\}$ is Lipschitz continuous, or more generally, for any $x,y\in \R$,
    $~|\psi(x) - \psi(y)| \leq K |x-y|,$
    where $K = \sup_{t\in R}|\psi^\prime(t)|$.
    \label{lem:lipschitz-pdf}
\end{lemma}

\begin{proof}
    This follows from a simple application of the mean value theorem.
\end{proof}
}

{

\begin{proof}[\bf Proof of Theorem 5.1]
    First, we use Lemma \ref{lem:lipschitz-pdf} to prove the following two inequalities.  Note that for any $\Xvec_i, \Xvec_j$ and ${\bf U}_j$, we have
    \begin{align}
        & \Big|\exp\Big\{-\frac{1}{2d}\Big\|\Xvec_i - \xvec_1 - \|\Xvec_j - \xvec_1\| {\bf U}_j\Big\|^2\Big\} - \exp\Big\{-\frac{1}{2d}\Big\|\Xvec_i - \xvec_2 - \|\Xvec_j - \xvec_2\| {\bf U}_j\Big\|^2\Big\}\Big|\nonumber\\
        & \leq K \Big|\frac{1}{2d}\Big\|\Xvec_i - \xvec_1 - \|\Xvec_j - \xvec_1\| {\bf U}_j\Big\| - \frac{1}{2d}\Big\|\Xvec_i - \xvec_2 - \|\Xvec_j - \xvec_2\| {\bf U}_j\Big\|\Big|\nonumber\\
        & \leq \frac{K}{2d} \Big\| \Xvec_i - \xvec_1 - \|\Xvec_j - \xvec_1\| {\bf U}_j - \Xvec_i + \xvec_2 + \|\Xvec_j - \xvec_2\| {\bf U}_j \Big\|\nonumber\\
        & \leq \frac{K}{2d} \Big\|\xvec_2 - \xvec_2 + \big(\|\Xvec_j - \xvec_2\| - \|\Xvec_j - \xvec_1\|\big){\bf U}_j \Big\|\nonumber\\
        & \leq \frac{K}{2d} \Big(\|\xvec_2 - \xvec_1\| + \Big|(\|\Xvec_j - \xvec_2\| - \|\Xvec_j - \xvec_1\|\big)\Big| \|{\bf U}_j\| \Big)\nonumber\\
        & \leq  \frac{K}{2d} \Big(\|\xvec_2 - \xvec_1\| + \|\Xvec_j - \xvec_2 - \Xvec_j + \xvec_1\|\Big) = \frac{K}{d} \|\xvec_1- \xvec_2\|.\label{eq:first-conv}
    \end{align}
    Similarly, for any $\Xvec_i, \Xvec_j$ and ${\bf U}_i, {\bf U}_j$, we  have
    \begin{align}
        & \Big|\exp\Big\{-\frac{1}{2d}\Big\|\|\Xvec_i - \xvec_1\|{\bf U}_i - \|\Xvec_j - \xvec_1\| {\bf U}_j\Big\|^2\Big\} - \exp\Big\{-\frac{1}{2d}\Big\|\|\Xvec_i - \xvec_2\|{\bf U}_i - \|\Xvec_j - \xvec_2\| {\bf U}_j\Big\|^2\Big\}\Big|\nonumber\\
        & \leq K \Big|\frac{1}{2d}\Big\|\|\Xvec_i - \xvec_1\|{\bf U}_i - \|\Xvec_j - \xvec_1\| {\bf U}_j\Big\| - \frac{1}{2d}\Big\|\|\Xvec_i - \xvec_2\|{\bf U}_i - \|\Xvec_j - \xvec_2\| {\bf U}_j\Big\|\Big|\nonumber\\
        & \leq \frac{K}{2d} \Big\| \|\Xvec_i - \xvec_1\|{\bf U}_i - \|\Xvec_j - \xvec_1\| {\bf U}_j - \|\Xvec_i - \xvec_2\|{\bf U}_i + \|\Xvec_j - \xvec_2\| {\bf U}_j \Big\|\nonumber\\
        & \leq \frac{K}{2d} \Big\| \big(\|\Xvec_i - \xvec_1\| - \|\Xvec_1 - \xvec_2\|\big){\bf U}_i  + \Big(\|\Xvec_j - \xvec_2\| - \|\Xvec_j - \xvec_1\|\Big){\bf U}_j \Big\|\nonumber\\
        & \leq \frac{K}{2d} \Big( \Big|(\|\Xvec_i - \xvec_1\| - \|\Xvec_i - \xvec_2\|\big)\Big| \|{\bf U}_i\| + \Big|\big(\|\Xvec_j - \xvec_2\| - \|\Xvec_j - \xvec_1\|\big)\Big| \|{\bf U}_j\| \Big)\nonumber\\
        & \leq  \frac{K}{2d} \big(\|\Xvec_i - \xvec_1 - \Xvec_i + \xvec_2\| + \|\Xvec_j - \xvec_2 - \Xvec_j + \xvec_1\|\big)= \frac{K}{d} \|\xvec_1- \xvec_2\|. \label{eq:second-conv}
    \end{align}
   Now, form \eqref{eq:first-conv} and \eqref{eq:second-conv} we get,
   \begin{align*}
       & \Big|\frac{1}{n(n-1)}\sum_{1\leq i\not = j \leq n} \Big\{\exp\Big\{-\frac{1}{2d}\Big\|\Xvec_i - \Xvec_j \Big\|^2\Big\} + \exp\Big\{-\frac{1}{2d}\Big\|\|\Xvec_i - \hat {\bm \theta}\|{\bf U}_i - \|\Xvec_j - \hat {\bm \theta}\| {\bf U}_j\Big\|^2\Big\} \\
       & \hspace{1.5in} - 2 \exp\Big\{-\frac{1}{2d}\Big\|\Xvec_i - \hat {\bm \theta} - \|\Xvec_j - \hat {\bm \theta}\| {\bf U}_j\Big\|^2\Big\}\Big\}\\
       & -\frac{1}{n(n-1)}\sum_{1\leq i\not = j \leq n} \Big\{\exp\Big\{-\frac{1}{2d}\Big\|\Xvec_i - \Xvec_j \Big\|^2\Big\} + \exp\Big\{-\frac{1}{2d}\Big\|\|\Xvec_i - \bm \theta\|{\bf U}_i - \|\Xvec_j - \bm \theta\| {\bf U}_j\Big\|^2\Big\} \\
       & \hspace{1.5in} - 2 \exp\Big\{-\frac{1}{2d}\Big\|\Xvec_i - \bm \theta - \|\Xvec_j - \bm \theta\| {\bf U}_j\Big\|^2\Big\}\Big\}\Big|\\
       & \leq \Big|\frac{1}{n(n-1)}\sum_{1\leq i\not = j \leq n} \exp\Big\{-\frac{1}{2d}\Big\|\Xvec_i - \hat {\bm \theta} - \|\Xvec_j - \hat {\bm \theta}\| {\bf U}_j\Big\|^2\Big\}\nonumber\\
        & \hspace{1in}- \frac{1}{n(n-1)}\sum_{1\leq i\not = j \leq n} \exp\Big\{-\frac{1}{2d}\Big\|\Xvec_i - \bm \theta - \|\Xvec_j - \bm \theta\| {\bf U}_j\Big\|^2\Big\}\Big|\\
        &~~~~ + \Big|\frac{1}{n(n-1)}\sum_{1\leq i\not = j \leq n} \exp\Big\{-\frac{1}{2d}\Big\|\|\Xvec_i - \hat {\bm \theta}\|{\bf U}_i - \|\Xvec_j -\hat {\bm \theta} \| {\bf U}_j\Big\|^2\Big\}\nonumber\\
        & \hspace{1in}- \frac{1}{n(n-1)}\sum_{1\leq i\not = j \leq n} \exp\Big\{-\frac{1}{2d}\Big\|\|\Xvec_i -\bm \theta\|{\bf U}_i - \|\Xvec_j - \bm \theta\| {\bf U}_j\Big\|^2\Big\}\Big|\\
        & \leq \frac{1}{n(n-1)}\sum_{1\leq i\not = j \leq n} \Big|\exp\Big\{-\frac{1}{2d}\Big\|\Xvec_i - \hat {\bm \theta} - \|\Xvec_j - \hat {\bm \theta}\| {\bf U}_j\Big\|^2\Big\}\\
        & \hspace{3in}- \exp\Big\{-\frac{1}{2d}\Big\|\Xvec_i - \bm \theta - \|\Xvec_j - \bm \theta\| {\bf U}_j\Big\|^2\Big\}\Big|\\
        &~~~~ + \frac{1}{n(n-1)}\sum_{1\leq i\not = j \leq n} \Big|\exp\Big\{-\frac{1}{2d}\Big\|\|\Xvec_i - \hat {\bm \theta}\|{\bf U}_i - \|\Xvec_j -\hat {\bm \theta} \| {\bf U}_j\Big\|^2\Big\}\\
        & \hspace{3in}- \exp\Big\{-\frac{1}{2d}\Big\|\|\Xvec_i -\bm \theta\|{\bf U}_i - \|\Xvec_j - \bm \theta\| {\bf U}_j\Big\|^2\Big\}\Big|\\
        & \leq \frac{1}{n(n-1)}\sum_{1\leq i\not=j\leq n} \frac{2K}{d} \|\hat{\bm\theta} - \bm \theta\| + \frac{1}{n(n-1)}\sum_{1\leq i\not=j\leq n} \frac{2K}{d} \|\hat{\bm\theta} - \bm \theta\| \\
        & \leq \frac{2K}{d} \|\hat{\bm\theta} - \bm \theta\|.
   \end{align*}
   Hence, as $n$ grows to infinity, if $\|\hat{\bm \theta} - \bm \theta_0\|$ converges to zero, the estimator $\Tilde{\zeta}_n$ based on the centered observations $\{\Xvec_i - \hat{\bm \theta}\}_{1\leq i\leq n}$ converges in probability to $\zeta(\Pr_{\Xvec-\bm \theta_0})$.

  As in Lemma 2.4, the bound on $c_{1-\alpha}$ can be established by computing the conditional expectation of the resampled test statistic conditioned on $\mathcal{D}^\prime$. Therefore, to avoid repetition, we omit the details of the proof here. 
\end{proof}

\begin{proof}[\bf Proof of Lemma 5.1]
    Since $\Xvec_1,\Xvec_2$ are symmetric about $\bm\theta$, the characteristic function of $\Xvec_1$ (and that of $\Xvec_2$) is of the form $\phi(\bm t) = \exp\{i\langle {\bf t}, {\bm\theta} \rangle\} g({\bf t})$, where $g(\cdot)$ is some real-valued function with $g({\bf t}) = g(-{\bf t})$ for all ${\bf t}$. Note that if $\Xvec_1,\Xvec_2$ are spherically symmetric about ${\bm \theta}$, then it is easy to show that $\Xvec_1-\Xvec_2$ is spherically symmetric about zero. 
				
    If $\Xvec_1-\Xvec_2$ is spherically symmetric about zero, then its characteristic function is of the form $f(\|{\bf t}\|)$ where $f(\cdot)$ is some real-valued function. Also, note that
    \begin{align*}
	\phi_{\tiny \Xvec_1-\Xvec_2}(\bm t) = \phi_{\tiny \Xvec_1}({\bf t}) \phi_{\tiny -\Xvec_2}({\bf t}) = \phi_{\tiny\Xvec_1}({\bf t})\phi_{\tiny\Xvec_1}(-{\bf t}) = g^2({\bf t}).
    \end{align*}
    Hence, $f(\cdot)$ is non-negative and $g^2({\bf t}) = f(\|{\bf t}\|)~\forall {\bf t}\in \R^d$. Therefore, $\phi_{\tiny\Xvec_1}({\bf t}) = \exp\{i\langle {\bf t},{\bm \theta} \rangle\} h(\|{\bf t}\|),$ where $|h(\|{\bf t}\|)|=f^{1/2}(\|{\bf t}\|)$. This gives us the desired result.
\end{proof}
}

\end{document}